\documentclass[12pt]{article}
\usepackage{e-jc}

\usepackage{amsrefs}

\usepackage{epsfig}  		
\usepackage{epic,eepic}       
\usepackage{graphicx}

\usepackage{enumerate}
\usepackage{mathbbol}
\usepackage{amssymb}
\usepackage{braket}
\usepackage{longtable}
\usepackage{caption}
\usepackage{mathtools}
\usepackage{microtype}
\usepackage{color, soul}

\DeclareSymbolFontAlphabet{\mathbb}{AMSb}
\DeclareSymbolFontAlphabet{\mathbbl}{bbold}

\numberwithin{equation}{section}

\newtheorem{notation}[theorem]{Notation}

\newcommand{\bs}{\backslash}
\newcommand{\ceil}[1]{\left \lceil #1 \right \rceil }

\newcommand{\meet}{\wedge}
\newcommand{\join}{\vee}
\newcommand{\proves}{\vdash}
\newcommand{\satisfies}{\vDash}

\renewcommand\AA{{\mathcal A}}

\newcommand\EE{{\mathcal E}}

\newcommand\LL{{\mathcal L}}
\newcommand\MM{{\mathcal M}}

\newcommand\QQ{{\mathcal Q}}

\newcommand\bbzero{\mathbbl 0} 
\newcommand\bbone{\mathbbl 1}  

\newcommand\acl{\hbox{\rm acl}}

\newcommand\dcl{\hbox{\rm dcl}}

\newcommand\<{\langle}
\renewcommand\>{\rangle}

\newcommand\fraisse{Fra\"\i ss\'e }
\newcommand{\tp}{\mathrm{tp}}

\usepackage{amsmath,amssymb}
\usepackage{graphicx}
\usepackage[colorlinks=true,citecolor=black,linkcolor=black,urlcolor=blue]{hyperref}

\expandafter\ifx\csname ProvidesPackage\endcsname\relax\toks0=\expandafter{%
\fi\ProvidesPackage{diagrams}[2014/12/31 v3.94 Paul Taylor's commutative
diagrams]
\toks0=\bgroup}
\ifx\diagram\isundefined\else\expandafter\endinput
\fi

\edef\cdrestoreat{
\noexpand\catcode`\noexpand\@=\the\catcode`\@
\noexpand\catcode`\noexpand\#=\the\catcode`\#
\noexpand\catcode`\noexpand\$=\the\catcode`\$
\noexpand\catcode`\noexpand\<=\the\catcode`\<
\noexpand\catcode`\noexpand\>=\the\catcode`\>
\noexpand\catcode`\noexpand\:=\the\catcode`\:
\noexpand\catcode`\noexpand\;=\the\catcode`\;
\noexpand\catcode`\noexpand\!=\the\catcode`\!
\noexpand\catcode`\noexpand\?=\the\catcode`\?
\noexpand\catcode`\noexpand\+=\the\catcode'53
}\catcode`\@=11 \catcode`\#=6 \catcode`\<=12 \catcode`\>=12 \catcode'53=12
\catcode`\:=12 \catcode`\;=12 \catcode`\!=12 \catcode`\?=12

\ifx\diagram@help@messages\CD@qK\let\diagram@help@messages y\fi

\def\cdps@Rokicki#1{\special{ps:#1}}\let\cdps@dvips\cdps@Rokicki\let
\cdps@RadicalEye\cdps@Rokicki\let\CD@HB\cdps@Rokicki\let\CD@IK\cdps@Rokicki
\let\CD@HB\cdps@Rokicki
\def\cdps@Bechtolsheim#1{\special{dvitps: Literal "#1"}}%
\let\cdps@dvitps\cdps@Bechtolsheim\let\cdps@IntegratedComputerSystems
\cdps@Bechtolsheim
\def\cdps@Clark#1{\special{dvitops: inline #1}}
\let\cdps@dvitops\cdps@Clark
\let\cdps@OzTeX\empty\let\cdps@oztex\empty\let\cdps@Trevorrow\empty
\def\cdps@Coombes#1{\special{ps-string #1}}

\def\CD@DE{\global\let}\def\CD@RH{\outer\def}

{\escapechar\m@ne\xdef\CD@o{\string\{}\xdef\CD@yC{\string\}}
\catcode`\&=4 \CD@DE\CD@Q=&\xdef\CD@S{\string\&}
\catcode`\$=3 \CD@DE\CD@k=$\CD@DE\CD@ND=$
\xdef\CD@nC{\string\$}\gdef\CD@LG{$$}
\catcode`\_=8 \CD@DE\CD@lJ=_
\obeylines\catcode`\^=7 \CD@DE\@super=^
\ifnum\newlinechar=10 \gdef\CD@uG{^^J}
\else\ifnum\newlinechar=13 \gdef\CD@uG{^^M}
\else\ifnum\newlinechar=-1 \gdef\CD@uG{^^J}
\else\CD@DE\CD@uG\space\expandafter\message{! input error: \noexpand
\newlinechar\space is ASCII \the\newlinechar, not LF=10 or CR=13.}
\fi\fi\fi}

\mathchardef\lessthan='30474 \mathchardef\greaterthan='30476

\ifx\tenln\CD@qK
\font\tenln=line10\relax
\fi\ifx\tenlnw\CD@qK\ifx\tenln\nullfont\let\tenlnw\nullfont\else
\font\tenlnw=linew10\relax
\fi\fi

\ifx\inputlineno\CD@qK\csname newcount\endcsname\inputlineno\inputlineno\m@ne
\fi

\def\cd@shouldnt#1{\CD@KB{* THIS (#1) SHOULD NEVER HAPPEN! *}}

\def\get@round@pair#1(#2,#3){#1{#2}{#3}}
\def\get@square@arg#1[#2]{#1{#2}}
\def\CD@AE#1{\CD@PK\let\CD@DH\CD@@E\CD@@E#1,],}
\def\CD@m{[}\def\CD@RD{]}\def\commdiag#1{{\let\enddiagram\relax\diagram[]#1%
\enddiagram}}

\def\CD@BF{{\ifx\CD@EH[\aftergroup\get@square@arg\aftergroup\CD@YH\else
\aftergroup\CD@JH\fi}}
\def\CD@CF#1#2{\def\CD@YH{#1}\def\CD@JH{#2}\futurelet\CD@EH\CD@BF}

\def\CD@KK{|}

\def\CD@PB{
\tokcase\CD@DD:\CD@y\break@args;\catcase\@super:\upper@label;\catcase\CD@lJ:%
\lower@label;\tokcase{~}:\middle@label;
\tokcase<:\CD@iF;
\tokcase>:\CD@iI;
\tokcase(:\CD@BC;
\tokcase[:\optional@;
\tokcase.:\CD@JJ;
\catcase\space:\eat@space;\catcase\bgroup:\positional@;\default:\CD@@A
\break@args;\endswitch}

\def\switch@arg{
\catcase\@super:\upper@label;\catcase\CD@lJ:\lower@label;\tokcase[:\optional@
;
\tokcase.:\CD@JJ;
\catcase\space:\eat@space;\catcase\bgroup:\positional@;\tokcase{~}:%
\middle@label;
\default:\CD@y\break@args;\endswitch}

\let\CD@tJ\relax\ifx\protect\CD@qK\let\protect\relax\fi\ifx\AtEndDocument
\CD@qK\def\CD@PG{\CD@gB}\def\CD@GF#1#2{}\else\def\CD@PG#1{\edef\CD@CH{#1}%
\expandafter\CD@oC\CD@CH\CD@OD}\def\CD@oC#1\CD@OD{\AtEndDocument{\typeout{%
\CD@tA: #1}}}\def\CD@GF#1#2{\gdef#1{#2}\AtEndDocument{#1}}\fi\def\CD@ZA#1#2{%
\def#1{\CD@PG{#2\CD@mD\CD@W}\CD@DE#1\relax}}\def\CD@uF#1\repeat{\def\CD@p{#1}%
\CD@OF}\def\CD@OF{\CD@p\relax\expandafter\CD@OF\fi}\def\CD@sF#1\repeat{\def
\CD@q{#1}\CD@PF}\def\CD@PF{\CD@q\relax\expandafter\CD@PF\fi}\def\CD@tF#1%
\repeat{\def\CD@r{#1}\CD@QF}\def\CD@QF{\CD@r\relax\expandafter\CD@QF\fi}\def
\CD@tG#1#2#3{\def#2{\let#1\iftrue}\def#3{\let#1\iffalse}#3}\if y%
\diagram@help@messages\def\CD@rG#1#2{\csname newtoks\endcsname#1#1=%
\expandafter{\csname#2\endcsname}}\else\csname newtoks\endcsname\no@cd@help
\no@cd@help{See the manual}\def\CD@rG#1#2{\let#1\no@cd@help}\fi\chardef\CD@lF
=1 \chardef\CD@lI=2 \chardef\CD@MH=5 \chardef\CD@tH=6 \chardef\CD@sH=7
\chardef\CD@PC=9 \dimendef\CD@hI=2 \dimendef\CD@hF=3 \dimendef\CD@mF=4
\dimendef\CD@mI=5 \dimendef\CD@wJ=6 \dimendef\CD@tI=8 \dimendef\CD@sI=9
\skipdef\CD@uB=1 \skipdef\CD@NF=2 \skipdef\CD@tB=3 \skipdef\CD@ZE=4 \skipdef
\countdef\CD@JC=9 \countdef\CD@gD=8 \countdef\CD@A=7 \def\sdef#1#2{\def#1{#2}%
}\def\CD@L#1{\expandafter\aftergroup\csname#1\endcsname}\def\CD@RC#1{%
\expandafter\def\csname#1\endcsname}\def\CD@sD#1{\expandafter\gdef\csname#1%
\endcsname}\def\CD@vC#1{\expandafter\edef\csname#1\endcsname}\def\CD@nF#1#2{%
\expandafter\let\csname#1\expandafter\endcsname\csname#2\endcsname}\def\CD@EE
#1#2{\expandafter\CD@DE\csname#1\expandafter\endcsname\csname#2\endcsname}%
\def\CD@AK#1{\csname#1\endcsname}\def\CD@XJ#1{\expandafter\show\csname#1%
\endcsname}\def\CD@ZJ#1{\expandafter\showthe\csname#1\endcsname}\def\CD@WJ#1{%
\expandafter\showbox\csname#1\endcsname}\def\CD@tA{Commutative Diagram}\edef
\CD@kH{\string\par}\edef\CD@dC{\string\diagram}\edef\CD@HD{\string\enddiagram
}\edef\CD@EC{\string\\}\def\CD@eF{LaTeX}\ifx\@ignoretrue\CD@qK\expandafter
\CD@tG\csname if@ignore\endcsname\ignore@true\ignore@false\def\@ignoretrue{%
\global\ignore@true}\def\@ignorefalse{\global\ignore@false}\fi

\def\CD@g{{\ifnum0=`}\fi}\def\CD@wC{\ifnum0=`{\fi}}\def\catcase#1:{\ifcat
\noexpand\CD@EH#1\CD@tJ\expandafter\CD@kC\else\expandafter\CD@dJ\fi}\def
\tokcase#1:{\ifx\CD@EH#1\CD@tJ\expandafter\CD@kC\else\expandafter\CD@dJ\fi}%
\def\CD@kC#1;#2\endswitch{#1}\def\CD@dJ#1;{}\let\endswitch\relax\def\default:%
#1;#2\endswitch{#1}\ifx\at@\CD@qK\def\at@{@}\fi\edef\CD@P{\CD@o pt\CD@yC}%
\CD@RC{\CD@P>}#1>#2>{\CD@z\rTo\sp{#1}\sb{#2}\CD@z}\CD@RC{\CD@P<}#1<#2<{\CD@z
\lTo\sp{#1}\sb{#2}\CD@z}\CD@RC{\CD@P)}#1)#2){\CD@z\rTo\sp{#1}\sb{#2}\CD@z}%
\CD@RC{\CD@P(}#1(#2({\CD@z\lTo\sp{#1}\sb{#2}\CD@z}
\def\CD@O{\def\endCD{\enddiagram}\CD@RC{\CD@P A}##1A##2A{\uTo<{##1}>{##2}%
\CD@z\CD@z}\CD@RC{\CD@P V}##1V##2V{\dTo<{##1}>{##2}\CD@z\CD@z}\CD@RC{\CD@P=}{%
\CD@z\hEq\CD@z}\CD@RC{\CD@P\CD@KK}{\vEq\CD@z\CD@z}\CD@RC{\CD@P\string\vert}{%
\vEq\CD@z\CD@z}\CD@RC{\CD@P.}{\CD@z\CD@z}\let\CD@z\CD@Q}\def\CD@IE{\let\tmp
\CD@JE\ifcat A\noexpand\CD@CH\else\ifcat=\noexpand\CD@CH\else\ifcat\relax
\noexpand\CD@CH\else\let\tmp\at@\fi\fi\fi\tmp}\def\CD@JE#1{\CD@nF{tmp}{\CD@P
\string#1}\ifx\tmp\relax\def\tmp{\at@#1}\fi\tmp}\def\CD@z{}\begingroup
\aftergroup\def\aftergroup\CD@T\aftergroup{\aftergroup\def\catcode`\@\active
\aftergroup @\endgroup{\futurelet\CD@CH\CD@IE}}\newcount\CD@uA\newcount\CD@vA
\newcount\CD@wA\newcount\CD@xA\newdimen\CD@OA\newdimen\CD@PA\CD@tG\CD@gE
\CD@@A\CD@y\CD@tG\CD@hE\CD@EA\CD@BA\newdimen\CD@RA\newdimen\CD@SA\newcount
\CD@yA\newcount\CD@zA\newdimen\CD@QA\newbox\CD@DA\CD@tG\CD@lE\CD@dA\CD@bA
\newcount\CD@LH\newcount\CD@TC\def\CD@V#1#2{\ifdim#1<#2\relax#1=#2\relax\fi}%
\def\CD@X#1#2{\ifdim#1>#2\relax#1=#2\relax\fi}\newdimen\CD@XH\CD@XH=1sp
\newdimen\CD@zC\CD@zC\z@\def\CD@cJ{\ifdim\CD@zC=1em\else\CD@nJ\fi}\def\CD@nJ{%
\CD@zC1em\def\CD@NC{\fontdimen8\textfont3 }\CD@@J\CD@NJ\setbox0=\vbox{\CD@t
\noindent\CD@k\null\penalty-9993\null\CD@ND\null\endgraf\setbox0=\lastbox
\unskip\unpenalty\setbox1=\lastbox\global\setbox\CD@IG=\hbox{\unhbox0\unskip
\unskip\unpenalty\setbox0=\lastbox}\global\setbox\CD@KG=\hbox{\unhbox1\unskip
\unpenalty\setbox1=\lastbox}}}\newdimen\CD@@I\CD@@I=1true in \divide\CD@@I300
\def\CD@zH#1{\multiply#1\tw@\advance#1\ifnum#1<\z@-\else+\fi\CD@@I\divide#1%
\tw@\divide#1\CD@@I\multiply#1\CD@@I}\def\MapBreadth{\afterassignment\CD@gI
\CD@LF}\newdimen\CD@LF\newdimen\CD@oI\def\CD@gI{\CD@oI\CD@LF\CD@V\CD@@I{4%
\CD@XH}\CD@X\CD@@I\p@\CD@zH\CD@oI\ifdim\CD@LF>\z@\CD@V\CD@oI\CD@@I\fi\CD@cJ}%
\def\CD@RJ#1{\CD@zD\count@\CD@@I#1\ifnum\count@>\z@\divide\CD@@I\count@\fi
\CD@gI\CD@NJ}\def\CD@NJ{\dimen@\CD@QC\count@\dimen@\divide\count@5\divide
\count@\CD@@I\edef\CD@OC{\the\count@}}\def\CD@AJ{\CD@QJ\z@}\def\CD@QJ#1{%
\CD@tI\axisheight\advance\CD@tI#1\relax\advance\CD@tI-.5\CD@oI\CD@zH\CD@tI
\CD@sI-\CD@tI\advance\CD@tI\CD@LF}\newdimen\CD@DC\CD@DC\z@\newdimen\CD@eJ
\CD@eJ\z@\def\CD@CJ#1{\CD@sI#1\relax\CD@tI\CD@sI\advance\CD@tI\CD@LF\relax}%
\def\horizhtdp{height\CD@tI depth\CD@sI}\def\axisheight{\fontdimen22\the
\textfont\tw@}\def\script@axisheight{\fontdimen22\the\scriptfont\tw@}\def
\ss@axisheight{\fontdimen22\the\scriptscriptfont\tw@}\def\CD@NC{0.4pt}\def
\CD@VK{\fontdimen3\textfont\z@}\def\CD@UK{\fontdimen3\textfont\z@}\newdimen
\PileSpacing\newdimen\CD@nA\CD@nA\z@\def\CD@RG{\ifincommdiag1.3em\else2em\fi}%
\newdimen\CD@YB\def\CellSize{\afterassignment\CD@kB\DiagramCellHeight}%
\newdimen\DiagramCellHeight\DiagramCellHeight-\maxdimen\newdimen
\DiagramCellWidth\DiagramCellWidth-\maxdimen\def\CD@kB{\DiagramCellWidth
\DiagramCellHeight}\def\CD@QC{3em}\newdimen\MapShortFall\def\MapsAbut{%
\MapShortFall\z@\objectheight\z@\objectwidth\z@}\newdimen\CD@iA\CD@iA\z@
\CD@tG\CD@vE\CD@aB\CD@ZB\expandafter\ifx\expandafter\iftrue\csname
ifUglyObsoleteDiagrams\endcsname\CD@ZB\else\CD@aB\fi\CD@nF{%
ifUglyObsoleteDiagrams}{relax}\newif\ifUglyObsoleteDiagrams\def\CD@nK{\CD@aB
\UglyObsoleteDiagramsfalse}\def\CD@oK{\CD@ZB\UglyObsoleteDiagramstrue}\CD@vE
\CD@nK\else\CD@oK\fi\CD@tG\CD@hK\CD@dK\CD@cK\CD@cK\def\CD@sK{\ifx\pdfoutput
\CD@qK\else\ifx\pdfoutput\relax\else\ifnum\pdfoutput>\z@\CD@pK\fi\fi\fi} \def
\CD@pK{\global\CD@dK\global\CD@aB\global\UglyObsoleteDiagramsfalse\global\let
\CD@n\empty\global\let\CD@oK\relax\global\let\CD@pK\relax\global\let\CD@sK
\relax}\def\CD@tK#1{}\ifx\pdfliteral\CD@qK\else\ifx
\pdfliteral\relax\else\let\CD@tK\pdfliteral\fi\fi\ifx\XeTeXrevision\CD@qK
\else\ifx\XeTeXrevision\relax\else\ifdim\XeTeXrevision pt<.996pt \expandafter
\message{! XeTeX version \XeTeXrevision\space does not support PDF literals,
so diagonals will not work!}\else\expandafter\message{RUNNING UNDER XETEX
\XeTeXrevision}\CD@pK\fi\fi\fi\CD@sK\def\newarrowhead{\CD@mG h\CD@BG\CD@GG>}%
\def\newarrowtail{\CD@mG t\CD@BG\CD@GG>}\def\newarrowmiddle{\CD@mG m\CD@BG
\hbox@maths\empty}\def\newarrowfiller{\CD@mG f\CD@bE\CD@MK-}\def\CD@mG#1#2#3#%
4#5#6#7#8#9{\CD@RC{r#1:#5}{#2{#6}}\CD@RC{l#1:#5}{#2{#7}}\CD@RC{d#1:#5}{#3{#8}%
}\CD@RC{u#1:#5}{#3{#9}}\CD@vC{-#1:#5}{\expandafter\noexpand\csname-#1:#4%
\endcsname\noexpand\CD@MC}\CD@vC{+#1:#5}{\expandafter\noexpand\csname+#1:#4%
\endcsname\noexpand\CD@MC}}\CD@ZA\CD@MC{\CD@eF\space diagonals are used unless
PostScript is set}\def\defaultarrowhead#1{\edef\CD@sJ{#1}\CD@@J}\def\CD@@J{%
\CD@IJ\CD@sJ<>ht\CD@IJ\CD@sJ<>th}\def\CD@IJ#1#2#3#4#5{\CD@HJ{r#4}{#3}{l#5}{#2%
}{r#4:#1}\CD@HJ{r#5}{#2}{l#4}{#3}{l#4:#1}\CD@HJ{d#4}{#3}{u#5}{#2}{d#4:#1}%
\CD@HJ{d#5}{#2}{u#4}{#3}{u#4:#1}}\def\CD@HJ#1#2#3#4#5{\begingroup\aftergroup
\CD@GJ\CD@L{#1+:#2}\CD@L{#1:#2}\CD@L{#3:#4}\CD@L{#5}\endgroup}\def\CD@GJ#1#2#%
3#4{\csname newbox\endcsname#1\def#2{\copy#1}\def#3{\copy#1}\setbox#1=\box
\voidb@x}\def\CD@sJ{}\CD@@J\def\CD@GJ#1#2#3#4{\setbox#1=#4}\ifx\tenln
\nullfont\def\CD@sJ{vee}\else\let\CD@sJ\CD@eF\fi\def\CD@xF#1#2#3{\begingroup
\aftergroup\CD@wF\CD@L{#1#2:#3#3}\CD@L{#1#2:#3}\aftergroup\CD@yF\CD@L{#1#2:#3%
-#3}\CD@L{#1#2:#3}\endgroup}\def\CD@wF#1#2{\def#1{\hbox{\rlap{#2}\kern.4%
\CD@zC#2}}}\def\CD@yF#1#2{\def#1{\hbox{\rlap{#2}\kern.4\CD@zC#2\kern-.4\CD@zC
}}}\CD@xF lh>\CD@xF rt>\CD@xF rh<\CD@xF rt<\def\CD@yF#1#2{\def#1{\hbox{\kern-%
.4\CD@zC\rlap{#2}\kern.4\CD@zC#2}}}\CD@xF rh>\CD@xF lh<\CD@xF lt>\CD@xF lt<%
\def\CD@wF#1#2{\def#1{\vbox{\vbox to\z@{#2\vss}\nointerlineskip\kern.4\CD@zC#%
2}}}\def\CD@yF#1#2{\def#1{\vbox{\vbox to\z@{#2\vss}\nointerlineskip\kern.4%
\CD@zC#2\kern-.4\CD@zC}}}\CD@xF uh>\CD@xF dt>\CD@xF dh<\CD@xF dt<\def\CD@yF#1%
#2{\def#1{\vbox{\kern-.4\CD@zC\vbox to\z@{#2\vss}\nointerlineskip\kern.4%
\CD@zC#2}}}\CD@xF dh>\CD@xF ut>\CD@xF uh<\CD@xF ut<\def\CD@BG#1{\hbox{%
\mathsurround\z@\offinterlineskip\CD@k\mkern-1.5mu{#1}\mkern-1.5mu\CD@ND}}%
\def\hbox@maths#1{\hbox{\CD@k#1\CD@ND}}\def\CD@GG#1{\hbox to\CD@LF{\setbox0=%
\hbox{\offinterlineskip\mathsurround\z@\CD@k{#1}\CD@ND}\dimen0.5\wd0\advance
\dimen0-.5\CD@oI\CD@zH{\dimen0}\kern-\dimen0\unhbox0\hss}}\def\CD@sB#1{\hbox
to2\CD@LF{\hss\offinterlineskip\mathsurround\z@\CD@k{#1}\CD@ND\hss}}\def
\CD@vF#1{\hbox{\mathsurround\z@\CD@k{#1}\CD@ND}}\def\CD@bE#1{\hbox{\kern-.15%
\CD@zC\CD@k{#1}\CD@ND\kern-.15\CD@zC}}\def\CD@MK#1{\vbox{\offinterlineskip
\kern-.2ex\CD@GG{#1}\kern-.2ex}}\def\@fillh{\xleaders\vrule\horizhtdp}\def
\@fillv{\xleaders\hrule width\CD@LF}\CD@nF{rf:-}{@fillh}\CD@nF{lf:-}{@fillh}%
\CD@nF{df:-}{@fillv}\CD@nF{uf:-}{@fillv}\CD@nF{rh:}{null}\CD@nF{rm:}{null}%
\CD@nF{rt:}{null}\CD@nF{lh:}{null}\CD@nF{lm:}{null}\CD@nF{lt:}{null}\CD@nF{dh%
:}{null}\CD@nF{dm:}{null}\CD@nF{dt:}{null}\CD@nF{uh:}{null}\CD@nF{um:}{null}%
\CD@nF{ut:}{null}\CD@nF{+h:}{null}\CD@nF{+m:}{null}\CD@nF{+t:}{null}\CD@nF{-h%
:}{null}\CD@nF{-m:}{null}\CD@nF{-t:}{null}\def\CD@@D{\hbox{\vrule height 1pt
depth-1pt width 1pt}}\CD@RC{rf:}{\CD@@D}\CD@nF{lf:}{rf:}\CD@nF{+f:}{rf:}%
\CD@RC{df:}{\CD@@D}\CD@nF{uf:}{df:}\CD@nF{-f:}{df:}\def\CD@BD{\CD@U\null
\CD@@D\null\CD@@D\null}\edef\CD@lG{\string\newarrow}\def\newarrow#1#2#3#4#5#6%
{\begingroup\edef\@name{#1}\edef\CD@oJ{#2}\edef\CD@iD{#3}\edef\CD@QG{#4}\edef
\CD@jD{#5}\edef\CD@LE{#6}\let\CD@HE\CD@sG\let\CD@FK\CD@BH\let\@x\CD@AH\ifx
\CD@oJ\CD@iD\let\CD@oJ\empty\fi\ifx\CD@LE\CD@jD\let\CD@LE\empty\fi\def\CD@LI{%
r}\def\CD@SF{l}\def\CD@IC{d}\def\CD@yJ{u}\def\CD@gH{+}\def\@m{-}\ifx\CD@iD
\CD@jD\ifx\CD@QG\CD@iD\let\CD@QG\empty\fi\ifx\CD@LE\empty\ifx\CD@iD\CD@aE\let
\@x\CD@yG\else\let\@x\CD@zG\fi\fi\else\edef\CD@a{\CD@iD\CD@oJ}\ifx\CD@a\empty
\ifx\CD@QG\CD@jD\let\CD@QG\empty\fi\fi\fi\ifmmode\aftergroup\CD@kG\else\CD@@A
\CD@oB rh{head\space\space}\CD@LE\CD@oB rf{filler}\CD@iD\CD@oB rm{middle}%
\CD@QG\ifx\CD@jD\CD@iD\else\CD@oB rf{filler}\CD@jD\fi\CD@oB rt{tail\space
\space}\CD@oJ\CD@gE\CD@HE\CD@FK\@x\CD@nG l-2+2{lu}{nw}\NorthWest\CD@nG r+2+2{%
ru}{ne}\NorthEast\CD@nG l-2-2{ld}{sw}\SouthWest\CD@nG r+2-2{rd}{se}\SouthEast
\else\aftergroup\CD@b\CD@L{r\@name}\fi\fi\endgroup}\def\CD@sG{\CD@vG\CD@LI
\CD@SF rl\Horizontal@Map}\def\CD@BH{\CD@vG\CD@IC\CD@yJ du\Vertical@Map}\def
\CD@AH{\CD@vG\CD@gH\@m+-\Vector@Map}\def\CD@yG{\CD@vG\CD@gH\@m+-\Slant@Map}%
\def\CD@zG{\CD@vG\CD@gH\@m+-\Slope@Map}\catcode`\/=\active\def\CD@vG#1#2#3#4#%
5{\CD@jG#1#3#5t:\CD@oJ/f:\CD@iD/m:\CD@QG/f:\CD@jD/h:\CD@LE//\CD@jG#2#4#5h:%
\CD@LE/f:\CD@jD/m:\CD@QG/f:\CD@iD/t:\CD@oJ//}\def\CD@jG#1#2#3#4//{\edef\CD@fG
{#2}\aftergroup\sdef\CD@L{#1\@name}\aftergroup{\aftergroup#3\CD@M#4//%
\aftergroup}}\def\CD@M#1/{\edef\CD@EH{#1}\ifx\CD@EH\empty\else\CD@L{\CD@fG#1}%
\expandafter\CD@M\fi}\catcode`\/=12 \def\CD@nG#1#2#3#4#5#6#7#8{\aftergroup
\sdef\CD@L{#6\@name}\aftergroup{\CD@L{#2\@name}\if#2#4\aftergroup\CD@CI\else
\aftergroup\CD@BI\fi\CD@L{#1\@name}%
\aftergroup(\aftergroup#3\aftergroup,\aftergroup#5\aftergroup)\aftergroup}}%
\def\CD@oB#1#2#3#4{\expandafter\ifx\csname#1#2:#4\endcsname\relax\CD@y\CD@gB{%
arrow#3 "#4" undefined}\fi}\CD@rG\CD@VE{All five components must be defined
before an arrow.}\CD@rG\CD@SE{\CD@lG, unlike \string\HorizontalMap, is a
declaration.}\def\CD@b#1{\CD@YA{Arrows \string#1 etc could not be defined}%
\CD@VE}\def\CD@kG{\CD@YA{misplaced \CD@lG}\CD@SE}\def\newdiagramgrid#1#2#3{%
\CD@RC{cdgh@#1}{#2,],}
\CD@RC{cdgv@#1}{#3,],}}
\CD@tG\ifincommdiag\incommdiagtrue\incommdiagfalse\CD@tG\CD@@F\CD@IF\CD@HF
\newcount\CD@VA\CD@VA=0 \def\CD@yH{\CD@VA6 }\def\CD@OB{\CD@VA1 \global\CD@yA1
\CD@DE\CD@YF\empty}\def\CD@YF{}\def\CD@nB#1{\relax\CD@MD\edef\CD@vJ{#1}%
\begingroup\CD@rE\else\ifcase\CD@VA\ifmmode\else\CD@YG\CD@E0\fi\or\CD@cE5\or
\CD@YG\CD@F5\or\CD@YG\CD@B5\or\CD@YG\CD@B5\or\CD@YG\CD@C5\or\CD@cE7\or\CD@YG
\CD@D7\fi\fi\endgroup\xdef\CD@YF{#1}}\def\CD@pB#1#2#3#4#5{\relax\CD@MD\xdef
\CD@vJ{#4}\begingroup\ifnum\CD@VA<#1 \expandafter\CD@cE\ifcase\CD@VA0\or#2\or
#3\else#2\fi\else\ifnum\CD@VA<6 \CD@tJ\CD@YG\CD@B#2\else\CD@YG\CD@G#2\fi\fi
\endgroup\CD@DE\CD@YF\CD@vJ\ifincommdiag\let\CD@ZD#5\else\let\CD@ZD\CD@LK\fi}%
\def\CD@yI{\global\CD@yA=\ifnum\CD@VA<5 1\else2\fi\relax}\def\CD@OI{\CD@VA
\CD@yA}\def\CD@cE#1{\aftergroup\CD@VA\aftergroup#1\aftergroup\relax}\def
\CD@HH{\def\CD@nB##1{\relax}\let\CD@pB\CD@FH\let\CD@yH\relax\let\CD@OB\relax
\let\CD@yI\relax\let\CD@OI\relax}\def\CD@FH#1#2#3#4#5{\ifincommdiag\let\CD@ZD
#5\else\xdef\CD@vJ{#4}\let\CD@ZD\CD@LK\fi}\def\CD@YG#1{\aftergroup#1%
\aftergroup\relax\CD@cE}\def\CD@B{\CD@YE\CD@S\CD@ME\CD@Q}\def\CD@G{\CD@YE{%
\CD@yC\CD@S}\CD@XE\CD@QD\CD@Q}\def\CD@F{\CD@YE{*\CD@S}\CD@RE\clubsuit\CD@Q}%
\def\CD@C{\CD@YE{\CD@S*\CD@S}\CD@RE\CD@Q\clubsuit\CD@Q}\def\CD@D{\CD@YE\CD@EC
\CD@TE\\}\def\CD@E{\CD@YE\CD@nC\CD@QE\CD@k}\def\CD@LK{\CD@YA{\CD@vJ\space
ignored \CD@dH}\CD@WE}\def\CD@FE{}\def\CD@d{\CD@YA{maps must never be enclosed
in braces}\CD@OE}\def\CD@dH{outside diagram}\def\CD@FC{\string\HonV, \string
\VonH\space and \string\HmeetV}\CD@rG\CD@ME{The way that horizontal and
vertical arrows are terminated implicitly means\CD@uG that they cannot be
mixed with each other or with \CD@FC.}\CD@rG\CD@XE{\string\pile\space is for
parallel horizontal arrows; verticals can just be put together in\CD@uG a cell%
. \CD@FC\space are not meaningful in a \string\pile.}\CD@rG\CD@RE{The
horizontal maps must point to an object, not each other (I've put in\CD@uG one
which you're unlikely to want). Use \string\pile\space if you want them
parallel.}\CD@rG\CD@TE{Parallel horizontal arrows must be in separate layers
of a \string\pile.}\CD@rG\CD@QE{Horizontal arrows may be used \CD@dH s, but
must still be in maths.}\CD@rG\CD@WE{Vertical arrows, \CD@FC\space\CD@dH s don%
't know where\CD@uG where to terminate.}\CD@rG\CD@OE{This prevents them from
stretching correctly.}\def\CD@YE#1{\CD@YA{"#1" inserted \ifx\CD@YF\empty
before \CD@vJ\else between \CD@YF\ifx\CD@YF\CD@vJ s\else\space and \CD@vJ\fi
\fi}}\count@=\year\multiply\count@12 \advance\count@\month\ifnum\count@>24247
\expandafter\endinput\fi\def
\Horizontal@Map{\CD@nB{horizontal map}\CD@LC\CD@TJ\CD@qD}\def\CD@TJ{\CD@GB-%
9999 \let\CD@ZD\CD@XD\ifincommdiag\else\CD@cJ\ifinpile\else\skip2\z@ plus 1.5%
\CD@VK minus .5\CD@UK\skip4\skip2 \fi\fi\let\CD@kD\@fillh\CD@nF{fill@dot}{rf:%
.}}\def\Vector@Map{\CD@HK4}\def\Slant@Map{\CD@HK{\CD@EF255\else6\fi}}\def
\Slope@Map{\CD@HK\CD@OC}\def\CD@HK#1#2#3#4#5#6{\CD@LC\def\CD@WK{2}\def\CD@aK{%
2}\def\CD@ZK{1}\def\CD@bK{1}\let\Horizontal@Map\CD@nI\def\CD@OG{#1}\def\CD@NI
{\CD@U#2#3#4#5#6}}\def\CD@nI{\CD@TJ\CD@JB\let\CD@ZD\CD@TD\CD@qD}\CD@tG\CD@pE
\CD@rA\CD@qA\CD@rA\def\cds@missives{\CD@rA}\def\CD@TD{\CD@vE\let\CD@OG\CD@OC
\CD@x\CD@zE\CD@WF\fi\setbox0\hbox{\incommdiagfalse\CD@HI}\CD@pE\CD@aD\else
\global\CD@YC\CD@bD\fi\ifvoid6 \ifvoid7 \CD@eE\fi\fi\CD@zE\else\CD@BD\global
\CD@YC\let\CD@CG\CD@IH\CD@YD\fi\else\CD@NI\CD@MI\global\CD@YC\CD@YD\fi}\def
\CD@LC{\begingroup\dimen1=\MapShortFall\dimen2=\CD@RG\dimen5=\MapShortFall
\setbox3=\box\voidb@x\setbox6=\box\voidb@x\setbox7=\box\voidb@x\CD@pD
\mathsurround\z@\skip2\z@ plus1fill minus 1000pt\skip4\skip2 \CD@TB}\CD@tG
\CD@tE\CD@UB\CD@TB\def\CD@U#1#2#3#4#5{\let\CD@oJ#1\let\CD@iD#2\let\CD@QG#3%
\let\CD@jD#4\let\CD@LE#5\CD@TB\ifx\CD@iD\CD@jD\CD@UB\fi}\def\CD@qD#1#2#3#4#5{%
\CD@U#1#2#3#4#5\CD@tD}\def\Vertical@Map{\CD@pB433{vertical map}\CD@cD\CD@LC
\CD@GB-9995 \let\CD@kD\@fillv\CD@nF{fill@dot}{df:.}\CD@qD}\def\break@args{%
\def\CD@tD{\CD@ZD}\CD@ZD\endgroup\aftergroup\CD@FE}\def\CD@MJ{\setbox1=\CD@oJ
\setbox5=\CD@LE\ifvoid3 \ifx\CD@QG\null\else\setbox3=\CD@QG\fi\fi\CD@@G2%
\CD@iD\CD@@G4\CD@jD}\def\CD@pF#1{\ifvoid1\else\CD@oF1#1\fi\ifvoid2\else\CD@oF
2#1\fi\ifvoid3\else\CD@oF3#1\fi\ifvoid4\else\CD@oF4#1\fi\ifvoid5\else\CD@oF5#%
1\fi} \def\CD@oF#1#2{\setbox#1\vbox{\offinterlineskip\box#1\dimen@\prevdepth
\advance\dimen@-#2\relax\setbox0\null\dp0\dimen@\ht0-\dimen@\box0}}\def\CD@@G
#1#2{\ifx#2\CD@kD\setbox#1=\box\voidb@x\else\setbox#1=#2\def#2{\xleaders\box#%
1}\fi}\CD@ZA\CD@BK{\string\HorizontalMap, \string\VerticalMap\space and
\string\DiagonalMap\CD@uG are obsolete - use \CD@lG\space to pre-define maps}%
\def\HorizontalMap#1#2#3#4#5{\CD@BK\CD@nB{old horizontal map}\CD@LC\CD@TJ\def
\CD@oJ{\CD@UH{#1}}\CD@SH\CD@iD{#2}\def\CD@QG{\CD@UH{#3}}\CD@SH\CD@jD{#4}\def
\CD@LE{\CD@UH{#5}}\CD@tD}\def\VerticalMap#1#2#3#4#5{\CD@BK\CD@pB433{vertical
map}\CD@cD\CD@LC\CD@GB-9995 \let\CD@kD\@fillv\def\CD@oJ{\CD@GG{#1}}\CD@VH
\CD@iD{#2}\def\CD@QG{\CD@GG{#3}}\CD@VH\CD@jD{#4}\def\CD@LE{\CD@GG{#5}}\CD@tD}%
\def\DiagonalMap#1#2#3#4#5{\CD@BK\CD@LC\def\CD@OG{4}\let\CD@kD\CD@qK\let
\CD@ZD\CD@YD\def\CD@WK{2}\def\CD@aK{2}\def\CD@ZK{1}\def\CD@bK{1}\def\CD@QG{%
\CD@vF{#3}}\ifPositiveGradient\let\mv\raise\def\CD@oJ{\CD@vF{#5}}\def\CD@iD{%
\CD@vF{#4}}\def\CD@jD{\CD@vF{#2}}\def\CD@LE{\CD@vF{#1}}\else\let\mv\lower\def
\CD@oJ{\CD@vF{#1}}\def\CD@iD{\CD@vF{#2}}\def\CD@jD{\CD@vF{#4}}\def\CD@LE{%
\CD@vF{#5}}\fi\CD@tD}\def\CD@aE{-}\def\CD@AD{\empty}\def\CD@SH{\CD@EG\CD@bE
\CD@aE\@fillh}\def\CD@VH{\CD@EG\CD@MK\CD@KK\@fillv}\def\CD@EG#1#2#3#4#5{\def
\CD@CH{#5}\ifx\CD@CH#2\let#4#3\else\let#4\null\ifx\CD@CH\empty\else\ifx\CD@CH
\CD@AD\else\let#4\CD@CH\fi\fi\fi}\def\CD@UH#1{\hbox{\mathsurround\z@
\offinterlineskip\def\CD@CH{#1}\ifx\CD@CH\empty\else\ifx\CD@CH\CD@AD\else
\CD@k\mkern-1.5mu{\CD@CH}\mkern-1.5mu\CD@ND\fi\fi}}\def\CD@yD#1#2{\setbox#1=%
\hbox\bgroup\setbox0=\hbox{\CD@k\labelstyle()\CD@ND}
\setbox1=\null\ht1\ht0\dp1\dp0\box1 \kern.1\CD@zC\CD@k\bgroup\labelstyle
\aftergroup\CD@LD\CD@xD}\def\CD@LD{\CD@ND\kern.1\CD@zC\egroup\CD@tD}\def
\CD@xD{\futurelet\CD@EH\CD@mJ}\def\CD@mJ{
\catcase\bgroup:\CD@v;\catcase\egroup:\missing@label;\catcase\space:\CD@TF;%
\tokcase[:\CD@XF;
\default:\CD@zJ;\endswitch}\def\CD@v{\let\CD@MD\CD@c\let\CD@CH}\def\CD@zJ#1{%
\let\CD@UF\egroup{\let\actually@braces@missing@around@macro@in@label\CD@ZH
\let\CD@MD\CD@xC\let\CD@UF\CD@VF#1%
\actually@braces@missing@around@macro@in@label}\CD@UF}\def
\actually@braces@missing@around@macro@in@label{\let\CD@CH=}\def\missing@label
{\egroup\CD@YA{missing label}\CD@PE}\def\CD@xC{\egroup\missing@label}\outer
\def\CD@ZH{}\def\CD@UF{}\def\CD@VF{\CD@wC\CD@UF}\def\CD@MD{}\def\CD@XF{\let
\CD@N\CD@xD\get@square@arg\CD@AE}\CD@rG\CD@PE{The text which has just been
read is not allowed within map labels.}\def\CD@c{\egroup\CD@YA{missing \CD@yC
\space inserted after label}\CD@PE}\def\upper@label{\CD@oD\CD@yD6}\def
\lower@label{\def\positional@{\CD@@A\break@args}\CD@yD7}\def\middle@label{%
\CD@yD3}\CD@tG\CD@yE\CD@pD\CD@oD\def\CD@iF{\ifPositiveGradient\CD@tJ
\expandafter\upper@label\else\expandafter\lower@label\fi}\def\CD@iI{%
\ifPositiveGradient\CD@tJ\expandafter\lower@label\else\expandafter
\upper@label\fi}\def\positional@{\CD@gB{labels as positional arguments are
obsolete}\CD@yE\CD@tJ\expandafter\upper@label\else\expandafter\lower@label\fi
-}\def\CD@tD{\futurelet\CD@EH\switch@arg}\def\eat@space{\afterassignment
\CD@tD\let\CD@EH= }\def\CD@TF{\afterassignment\CD@xD\let\CD@EH= }\def\CD@BC{%
\get@round@pair\CD@uD}\def\CD@uD#1#2{\def\CD@WK{#1}\def\CD@aK{#2}\CD@tD}\def
\optional@{\let\CD@N\CD@tD\get@square@arg\CD@AE}\def\CD@JJ.{\CD@sC\CD@tD}\def
\CD@sC{\let\CD@iD\fill@dot\let\CD@jD\fill@dot\def\CD@MI{\let\CD@iD\dfdot\let
\CD@jD\dfdot}}\def\CD@MI{}\def\CD@@E#1,{\CD@nH#1,\begingroup\ifx\@name\CD@RD
\CD@FF\aftergroup\CD@e\fi\aftergroup\CD@jC\else\expandafter\def\expandafter
\CD@RF\expandafter{\csname\@name\endcsname}\expandafter\CD@vD\CD@RF\CD@KD\ifx
\CD@RF\empty\aftergroup\CD@pC\expandafter\aftergroup\csname\CD@FB\@name
\endcsname\expandafter\aftergroup\csname\CD@FB @\@name\endcsname\else\gdef
\CD@GE{#1}\CD@gB{\string\relax\space inserted before `[\CD@GE'}\message{(I was
trying to read this as a \CD@tA\ option.)}\aftergroup\CD@H\fi\fi\endgroup}%
\def\CD@vD#1#2\CD@KD{\def\CD@RF{#2}}\def\CD@jC{\let\CD@CH\CD@N\let\CD@N\relax
\CD@CH}\def\CD@H#1],{
\CD@jC\relax\def\CD@RF{#1}\ifx\CD@RF\empty\def\CD@RF{[\CD@GE]}%
\else\def\CD@RF{[\CD@GE,#1]}
\fi\CD@RF}\def\CD@pC#1#2{\ifx#2\CD@qK\ifx#1\CD@qK\CD@gB{option `\@name'
undefined}\else#1\fi\else\CD@FF\expandafter#2\CD@GK\CD@PK\else\CD@QK\fi\fi
\CD@DH}\CD@tG\CD@FF\CD@QK\CD@PK\def\CD@nH#1,{\CD@FF\ifx\CD@GK\CD@qK\CD@e\else
\expandafter\CD@oH\CD@GK,#1,(,),(,)[]%
\fi\fi\CD@FF\else\CD@mH#1==,\fi}\def\CD@e{\CD@gB{option `\@name' needs (x,y)
value}\CD@PK\let\@name\empty}\def\CD@mH#1=#2=#3,{\def\@name{#1}\def\CD@GK{#2}%
\def\CD@RF{#3}\ifx\CD@RF\empty\let\CD@GK\CD@qK\fi}%
\def\CD@oH#1(#2,#3)#4,(#5,#6)#7[]{\def\CD@GK{{#2}{#3}}\def\CD@RF{#1#4#5#6}%
\ifx\CD@RF\empty\def\CD@RF{#7}\ifx\CD@RF\empty\CD@e\fi\else\CD@e\fi}\def
\CD@FB{cds@}\let\CD@N\relax\def\CD@zD#1{\ifx\CD@GK\CD@qK\CD@gB{option `\@name
' needs a value}\else#1\CD@GK\relax\fi}\def\CD@BE#1#2{\ifx\CD@GK\CD@qK#1#2%
\relax\else#1\CD@GK\relax\fi}\def\cds@@showpair#1#2{\message{x=#1,y=#2}}\def
\cds@@diagonalbase#1#2{\edef\CD@ZK{#1}\edef\CD@bK{#2}}\def\CD@DI#1{\def\CD@CH
{#1}\CD@nF{@x}{cdps@#1}\ifx\CD@CH\empty\CD@f\CD@CH{cannot be used}\else\ifx
\CD@CH\relax\CD@f\CD@CH{unknown}\else\let\CD@IK\@x\fi\fi}\def\CD@f#1#2{\CD@gB
{PostScript translator `#1' #2}}\def\CD@PH{}\def\CD@PJ{\CD@fA\edef\CD@PH{%
\noexpand\CD@KB{\@name\space ignored within maths}}}\def\diagramstyle{\CD@cJ
\let\CD@N\relax\CD@CF\CD@AE\CD@AE}\CD@tG\CD@sE
\CD@SB\CD@RB\CD@tG\CD@qE\CD@EB\CD@DB\CD@tG\CD@oE\CD@pA\CD@oA\CD@tG\CD@iE
\CD@HA\CD@GA\CD@HA\CD@tG\CD@jE\CD@JA\CD@IA\CD@tG\CD@kE\CD@LA\CD@KA\CD@tG
\CD@EF\CD@DK\CD@CK\CD@tG\CD@rE\CD@JB\CD@IB\CD@tG\CD@mE\CD@gA\CD@fA\CD@tG
\CD@nE\CD@kA\CD@jA\CD@tG\CD@AF\CD@iG\CD@hG\CD@RC{cds@ }{}\CD@RC{cds@}{}\CD@RC
{cds@1em}{\CellSize1\CD@zC}\CD@RC{cds@1.5em}{\CellSize1.5\CD@zC}\CD@RC{cds@2%
em}{\CellSize2\CD@zC}\CD@RC{cds@2.5em}{\CellSize2.5\CD@zC}\CD@RC{cds@3em}{%
\CellSize3\CD@zC}\CD@RC{cds@3.5em}{\CellSize3.5\CD@zC}\CD@RC{cds@4em}{%
\CellSize4\CD@zC}\CD@RC{cds@4.5em}{\CellSize4.5\CD@zC}\CD@RC{cds@5em}{%
\CellSize5\CD@zC}\CD@RC{cds@6em}{\CellSize6\CD@zC}\CD@RC{cds@7em}{\CellSize7%
\CD@zC}\CD@RC{cds@8em}{\CellSize8\CD@zC}\def\cds@abut{\MapsAbut\dimen1\z@
\dimen5\z@}\def\cds@alignlabels{\CD@IA\CD@KA}\def\cds@amstex{\ifincommdiag
\CD@O\else\def\CD{\diagram[amstex]}
\fi\CD@T\catcode`\@\active}\def\cds@b{\let\CD@dB\CD@bB}\def\cds@balance{\let
\CD@hA\CD@AA}\let\cds@bottom\cds@b\def\cds@center{\cds@vcentre\cds@nobalance}%
\let\cds@centre\cds@center\def\cds@centerdisplay{\CD@HA\CD@PJ\cds@balance}%
\let\cds@centredisplay\cds@centerdisplay\def\cds@crab{\CD@BE\CD@DC{.5%
\PileSpacing}}\CD@RC{cds@crab-}{\CD@DC-.5\PileSpacing}\CD@RC{cds@crab+}{%
\CD@DC.5\PileSpacing}\CD@RC{cds@crab++}{\CD@DC1.5\PileSpacing}\CD@RC{cds@crab%
--}{\CD@DC-1.5\PileSpacing}\def\cds@defaultsize{\CD@BE{\let\CD@QC}{3em}\CD@NJ
}\def\cds@displayoneliner{\CD@DB}\let\cds@dotted\CD@sC\def\cds@dpi{\CD@RJ{1%
truein}}\def\cds@dpm{\CD@RJ{100truecm}}\let\CD@XA\CD@qK\def\cds@eqno{\let
\CD@XA\CD@GK\let\CD@EJ\empty}\def\cds@fixed{\CD@qA}\CD@tG\CD@fE\CD@J\CD@I\def
\cds@flushleft{\CD@I\CD@GA\CD@PJ\cds@nobalance\CD@BE\CD@nA\CD@nA}\def\cds@gap
{\CD@AJ\setbox3=\null\ht3=\CD@tI\dp3=\CD@sI\CD@BE{\wd3=}\MapShortFall} \def
\cds@grid{\ifx\CD@GK\CD@qK\let\h@grid\relax\let\v@grid\relax\else\CD@nF{%
h@grid}{cdgh@\CD@GK}\CD@nF{v@grid}{cdgv@\CD@GK}\ifx\h@grid\relax\CD@gB{%
unknown grid `\CD@GK'}\else\CD@WB\fi\fi}\let\h@grid\relax\let\v@grid\relax
\def\cds@gridx{\ifx\CD@GK\CD@qK\else\cds@grid\fi\let\CD@CH\h@grid\let\h@grid
\v@grid\let\v@grid\CD@CH}\def\cds@h{\CD@zD\DiagramCellHeight}\def\cds@hcenter
{\let\CD@hA\CD@aA}\let\cds@hcentre\cds@hcenter\def\cds@heads{\CD@BE{\let
\CD@sJ}\CD@sJ\CD@@J\CD@vE\else\ifx\CD@sJ\CD@eF\else\CD@MC\fi\fi}\let
\cds@height\cds@h\let\cds@hmiddle\cds@balance\def\cds@htriangleheight{\CD@BE
\DiagramCellHeight\DiagramCellHeight\DiagramCellWidth1.73205%
\DiagramCellHeight}\def\cds@htrianglewidth{\CD@BE\DiagramCellWidth
\DiagramCellWidth\DiagramCellHeight.57735\DiagramCellWidth}\CD@tG\CD@zE\CD@eE
\CD@dE\CD@eE\def\cds@hug{\CD@eE} \def\cds@inline{\CD@gA\let\CD@PH\empty}\def
\cds@inlineoneliner{\CD@EB}\CD@RC{cds@l>}{\CD@zD{\let\CD@RG}\dimen2=\CD@RG}%
\def\cds@labelstyle{\CD@zD{\let\labelstyle}}\def\cds@landscape{\CD@kA}\def
\cds@large{\CellSize5\CD@zC}\let\CD@EJ\empty\def\CD@FJ{\refstepcounter{%
equation}\def\CD@XA{\hbox{\@eqnnum}}}\def\cds@LaTeXeqno{\let\CD@EJ\CD@FJ}\def
\cds@lefteqno{\CD@pA}\def\cds@leftflush{\cds@flushleft\CD@J}\def
\cds@leftshortfall{\CD@zD{\dimen1 }}\def\cds@lowershortfall{%
\ifPositiveGradient\cds@leftshortfall\else\cds@rightshortfall\fi}\def
\cds@loose{\CD@VB}\def\cds@midhshaft{\CD@JA}\def\cds@midshaft{\CD@JA}\def
\cds@midvshaft{\CD@LA}\def\cds@moreoptions{\CD@@A}\let\cds@nobalance
\cds@hcenter\def\cds@nohcheck{\CD@HH}\def\cds@nohug{\CD@dE} \def
\cds@nooptions{\def\CD@aC{\CD@WD}}\let\cds@noorigin\cds@nobalance\def
\cds@nopixel{\CD@@I4\CD@XH\CD@cJ}\def\cds@UO{\CD@oK\global\let\CD@n\empty}%
\def\cds@UglyObsolete{\cds@UO\let\cds@PS\empty}\def\CD@rK#1{\CD@gB{option `#1%
' renamed as `UglyObsolete'}}\def\cds@noPostScript{\CD@rK{noPostScript}}\def
\cds@noPS{\CD@rK{noPostScript}}\def\cds@notextflow{\CD@RB}\def\cds@noTPIC{%
\CD@CK}\def\cds@objectstyle{\CD@zD{\let\objectstyle}}\def\cds@origin{\let
\CD@hA\CD@iB}\def\cds@p{\CD@zD\PileSpacing}\let\cds@pilespacing\cds@p\def
\cds@pixelsize{\CD@zD\CD@@I\CD@gI}\def\cds@portrait{\CD@jA}\def
\cds@PostScript{\CD@nK\global\let\CD@n\empty\CD@BE\CD@DI\empty}\def\cds@PS{%
\CD@nK\global\let\CD@n\empty}\CD@GF\CD@n{\typeout{\CD@tA: try the PostScript
option for better results}}\def\cds@repositionpullbacks{\let\make@pbk\CD@fH
\let\CD@qH\CD@pH}\def\cds@righteqno{\CD@oA}\def\cds@rightshortfall{\CD@zD{%
\dimen5 }}\def\cds@ruleaxis{\CD@zD{\let\axisheight}}\def\cds@cmex{\let\CD@GG
\CD@sB\let\CD@QJ\CD@CJ}\def\cds@s{\cds@height\DiagramCellWidth
\DiagramCellHeight}\def\cds@scriptlabels{\let\labelstyle\scriptstyle}\def
\cds@shortfall{\CD@zD\MapShortFall\dimen1\MapShortFall\dimen5\MapShortFall}%
\def\cds@showfirstpass{\CD@BE{\let\CD@nD}\z@}\def\cds@silent{\def\CD@KB##1{}%
\def\CD@gB##1{}}\let\cds@size\cds@s\def\cds@small{\CellSize2\CD@zC}\def
\cds@snake{\CD@BE\CD@eJ\z@}\def\cds@t{\let\CD@dB\CD@fB}\def\cds@textflow{%
\CD@SB\CD@PJ}\def\cds@thick{\let\CD@rF\tenlnw\CD@LF\CD@NC\CD@BE\MapBreadth{2%
\CD@LF}\CD@@J}\def\cds@thin{\let\CD@rF\tenln\CD@BE\MapBreadth{\CD@NC}\CD@@J}%
\def\cds@tight{\CD@WB}\let\cds@top\cds@t\def\cds@TPIC{\CD@DK}\def
\cds@uppershortfall{\ifPositiveGradient\cds@rightshortfall\else
\cds@leftshortfall\fi}\def\cds@vcenter{\let\CD@dB\CD@cB}\let\cds@vcentre
\cds@vcenter\def\cds@vtriangleheight{\CD@BE\DiagramCellHeight
\DiagramCellHeight\DiagramCellWidth.577035\DiagramCellHeight}\def
\cds@vtrianglewidth{\CD@BE\DiagramCellWidth\DiagramCellWidth
\DiagramCellHeight1.73205\DiagramCellWidth}\def\cds@vmiddle{\let\CD@dB\CD@eB}%
\def\cds@w{\CD@zD\DiagramCellWidth}\let\cds@width\cds@w\def\diagram{\relax
\protect\CD@bC}\def\enddiagram{\protect\CD@SG}\def\CD@bC{\CD@g\CD@uI
\incommdiagtrue\edef\CD@wI{\the\CD@NB}\global\CD@NB\z@\boxmaxdepth\maxdimen
\everycr{}\CD@sK\everymath{}\everyhbox{}\ifx\pdfsyncstop\CD@qK\else
\pdfsyncstop\fi\CD@aC}\def\CD@aC{\CD@y\let\CD@N\CD@ZC\CD@CF\CD@AE\CD@WD}\def
\CD@ZC{\CD@gE\expandafter\CD@aC\else\expandafter\CD@WD\fi}\def\CD@WD{\let
\CD@EH\relax\CD@nE\CD@vE\else\CD@hK\else\CD@KB{landscape ignored without
PostScript}\CD@jA\fi\fi\fi\CD@EJ\setbox2=\vbox\bgroup\CD@JF\CD@VD}\def\CD@cH{%
\CD@nE\CD@fB\else\CD@dB\fi\CD@hA\nointerlineskip\setbox0=\null\ht0-\CD@pI\dp0%
\CD@pI\wd0\CD@kI\box0 \global\CD@QA\CD@kF\global\CD@yA\CD@XB\ifx\CD@NK\CD@qK
\global\CD@RA\CD@kF\else\global\CD@RA\CD@NK\fi\egroup\CD@zF\CD@nE\setbox2=%
\hbox to\dp2{\vrule height\wd2 depth\CD@QA width\z@\global\CD@QA\ht2\ht2\z@
\dp2\z@\wd2\z@\CD@hK\CD@tK{q 0 1 -1 0 0 0 cm}\else\global\CD@iG\CD@IK{0 1
bturn}\fi\box2\CD@gK\hss}\CD@DB\fi\ifnum\CD@yA=1 \else\CD@DB\fi\global
\@ignorefalse\CD@mE\leavevmode\fi\ifvmode\CD@TA\else\ifmmode\CD@PH\CD@GI\else
\CD@qE\CD@gA\fi\ifinner\CD@gA\fi\CD@mE\CD@GI\else\CD@sE\CD@QB\else\CD@TA\fi
\fi\fi\fi\CD@dD}\def\CD@dD{\global\CD@NB\CD@wI\relax\CD@xE\global\CD@ID\else
\aftergroup\CD@mC\fi\if@ignore\aftergroup\ignorespaces\fi\CD@wC\ignorespaces}%
\def\CD@fB{\advance\CD@pI\dimen1\relax}\def\CD@eB{\advance\CD@pI.5\dimen1%
\relax}\def\CD@bB{}\def\CD@cB{\CD@fB\advance\CD@pI\CD@YB\divide\CD@pI2
\advance\CD@pI-\axisheight\relax}\def\CD@aA{}\def\CD@iB{\CD@kF\z@}\def\CD@AA{%
\ifdim\dimen2>\CD@kF\CD@kF\dimen2 \else\dimen2\CD@kF\CD@kI\dimen0 \advance
\CD@kI\dimen2 \fi}\def\CD@QB{\skip0\z@\relax\loop\skip1\lastskip\ifdim\skip1>%
\z@\unskip\advance\skip0\skip1 \repeat\vadjust{\prevdepth\dp\strutbox\penalty
\predisplaypenalty\vskip\abovedisplayskip\CD@UA\penalty\postdisplaypenalty
\vskip\belowdisplayskip}\ifdim\skip0=\z@\else\hskip\skip0 \global\@ignoretrue
\fi}\def\CD@TA{\CD@LG\kern-\displayindent\CD@UA\CD@LG\global\@ignoretrue}\def
\CD@UA{\hbox to\hsize{\CD@fE\ifdim\CD@RA=\z@\else\advance\CD@QA-\CD@RA\setbox
2=\hbox{\kern\CD@RA\box2}\fi\fi\setbox1=\hbox{\ifx\CD@XA\CD@qK\else\CD@k
\CD@XA\CD@ND\fi}\CD@oE\CD@iE\else\advance\CD@QA\wd1 \fi\wd1\z@\box1 \fi\dimen
0\wd2 \advance\dimen0\wd1 \advance\dimen0-\hsize\ifdim\dimen0>-\CD@nA\CD@HA
\fi\advance\dimen0\CD@QA\ifdim\dimen0>\z@\CD@KB{wider than the page by \the
\dimen0 }\CD@HA\fi\CD@iE\hss\else\CD@V\CD@QA\CD@nA\fi\CD@GI\hss\kern-\wd1\box
1 }}\def\CD@GI{\CD@AF\CD@@F\else\CD@SC\global\CD@hG\fi\fi\kern\CD@QA\box2 }%
\CD@tG\CD@wE\CD@YC\CD@XC\def\CD@JF{\CD@cJ\ifdim\DiagramCellHeight=-\maxdimen
\DiagramCellHeight\CD@QC\fi\ifdim\DiagramCellWidth=-\maxdimen
\DiagramCellWidth\CD@QC\fi\global\CD@XC\CD@IF\let\CD@FE\empty\let\CD@z\CD@Q
\let\overprint\CD@eH\let\CD@s\CD@rJ\let\enddiagram\CD@ED\let\\\CD@cC\let\par
\CD@jH\let\CD@MD\empty\let\switch@arg\CD@PB\let\shift\CD@iA\baselineskip
\DiagramCellHeight\lineskip\z@\lineskiplimit\z@\mathsurround\z@\tabskip\z@
\CD@OB}\def\CD@VD{\penalty-123 \begingroup\CD@jA\aftergroup\CD@K\halign
\bgroup\global\advance\CD@NB1 \vadjust{\penalty1}\global\CD@FA\z@\CD@OB\CD@j#%
#\CD@DD\CD@Q\CD@Q\CD@OI\CD@j##\CD@DD\cr}\def\CD@ED{\CD@MD\CD@GD\crcr\egroup
\global\CD@JD\endgroup}\def\CD@j{\global\advance\CD@FA1 \futurelet\CD@EH\CD@i
}\def\CD@i{\ifx\CD@EH\CD@DD\CD@tJ\hskip1sp plus 1fil \relax\let\CD@DD\relax
\CD@vI\else\hfil\CD@k\objectstyle\let\CD@FE\CD@d\fi}\def\CD@DD{\CD@MD\relax
\CD@yI\CD@vI\global\CD@QA\CD@iA\penalty-9993 \CD@ND\hfil\null\kern-2\CD@QA
\null}\def\CD@cC{\cr}\def\across#1{\span\omit\mscount=#1 \global\advance
\CD@FA\mscount\global\advance\CD@FA\m@ne\CD@sF\ifnum\mscount>2 \CD@fJ\repeat
\ignorespaces}\def\CD@fJ{\relax\span\omit\advance\mscount\m@ne}\def\CD@qJ{%
\ifincommdiag\ifx\CD@iD\@fillh\ifx\CD@jD\@fillh\ifdim\dimen3>\z@\else\ifdim
\dimen2>93\CD@@I\ifdim\dimen2>18\p@\ifdim\CD@LF>\z@\count@\CD@bJ\advance
\count@\m@ne\ifnum\count@<\z@\count@20\let\CD@aJ\CD@uJ\fi\xdef\CD@bJ{\the
\count@}\fi\fi\fi\fi\fi\fi\fi}\def\CD@cG#1{\vrule\horizhtdp width#1\dimen@
\kern2\dimen@}\def\CD@uJ{\rlap{\dimen@\CD@@I\CD@V\dimen@{.182\p@}\CD@zH
\dimen@\advance\CD@tI\dimen@\CD@cG0\CD@cG0\CD@cG2\CD@cG6\CD@cG6\CD@cG2\CD@cG0%
\CD@cG0\CD@cG2\CD@cG6\CD@cG0\CD@cG0\CD@cG2\CD@cG2\CD@cG6\CD@cG0\CD@cG0\CD@cG2%
\CD@cG6\CD@cG2\CD@cG2\CD@cG0\CD@cG0}}\def\CD@bJ{10}\def\CD@aJ{}\def\CD@XD{%
\CD@gE\CD@TB\fi\CD@x\CD@WF\CD@HI}\def\CD@x{\CD@QJ\CD@DC\CD@MJ\ifdim\CD@DC=\z@
\else\CD@pF\CD@DC\fi\ifvoid3 \setbox3=\null\ht3\CD@tI\dp3\CD@sI\else\CD@V{\ht
3}\CD@tI\CD@V{\dp3}\CD@sI\fi\dimen3=.5\wd3 \ifdim\dimen3=\z@\CD@tE\else\dimen
3-\CD@XH\fi\else\CD@TB\fi\CD@V{\dimen2}{\wd7}\CD@V{\dimen2}{\wd6}\CD@qJ
\advance\dimen2-2\dimen3 \dimen4.5\dimen2 \dimen2\dimen4 \advance\dimen2%
\CD@eJ\advance\dimen4-\CD@eJ\advance\dimen2-\wd1 \advance\dimen4-\wd5 \ifvoid
2 \else\CD@V{\ht3}{\ht2}\CD@V{\dp3}{\dp2}\CD@V{\dimen2}{\wd2}\fi\ifvoid4 \else
\CD@V{\ht3}{\ht4}\CD@V{\dp3}{\dp4}\CD@V{\dimen4}{\wd4}\fi\advance\skip2\dimen
2 \advance\skip4\dimen4 \CD@tE\advance\skip2\skip4 \dimen0\dimen5 \advance
\dimen0\wd5 \skip3-\skip4 \advance\skip3-\dimen0 \let\CD@jD\empty\else\skip3%
\z@\relax\dimen0\z@\fi}\def\CD@WF{\offinterlineskip\lineskip.2\CD@zC\ifvoid6
\else\setbox3=\vbox{\hbox to2\dimen3{\hss\box6\hss}\box3}\fi\ifvoid7 \else
\setbox3=\vtop{\box3 \hbox to2\dimen3{\hss\box7\hss}}\fi}\def\CD@HI{\kern
\dimen1 \box1 \CD@aJ\CD@iD\hskip\skip2 \kern\dimen0 \ifincommdiag\CD@jE
\penalty1\fi\kern\dimen3 \penalty\CD@GB\hskip\skip3 \null\kern-\dimen3 \else
\hskip\skip3 \fi\box3 \CD@jD\hskip\skip4 \box5 \kern\dimen5}\def\CD@MF{\ifnum
\CD@LH>\CD@TC\CD@V{\dimen1}\objectheight\CD@V{\dimen5}\objectheight\else\CD@V
{\dimen1}\objectwidth\CD@V{\dimen5}\objectwidth\fi}\def\CD@Y{\begingroup
\ifdim\dimen7=\z@\kern\dimen8 \else\ifdim\dimen6=\z@\kern\dimen9 \else\dimen5%
\dimen6 \dimen6\dimen9 \CD@KJ\dimen4\dimen2 \CD@dG{\dimen4}\dimen6\dimen5
\dimen7\dimen8 \CD@KJ\CD@iC{\dimen2}\ifdim\dimen2<\dimen4 \kern\dimen2 \else
\kern\dimen4 \fi\fi\fi\endgroup}\def\CD@jJ{\CD@JI\setbox\z@\hbox{\lower
\axisheight\hbox to\dimen2{\CD@DF\ifPositiveGradient\dimen8\ht\CD@MH\dimen9%
\CD@mI\else\dimen8\dp3 \dimen9\dimen1 \fi\else\dimen8 \ifPositiveGradient
\objectheight\else\z@\fi\dimen9\objectwidth\fi\advance\dimen8
\ifPositiveGradient-\fi\axisheight\CD@Y\unhbox\z@\CD@DF\ifPositiveGradient
\dimen8\dp3 \dimen9\dimen0 \else\dimen8\ht\CD@MH\dimen9\CD@mF\fi\else\dimen8
\ifPositiveGradient\z@\else\objectheight\fi\dimen9\objectwidth\fi\advance
\dimen8 \ifPositiveGradient\else-\fi\axisheight\CD@Y}}}\def\CD@bD{\dimen6
\CD@aK\DiagramCellHeight\dimen7 \CD@WK\DiagramCellWidth\CD@jJ
\ifPositiveGradient\advance\dimen7-\CD@ZK\DiagramCellWidth\else\dimen7 \CD@ZK
\DiagramCellWidth\dimen6\z@\fi\advance\dimen6-\CD@bK\DiagramCellHeight\CD@mK
\setbox0=\rlap{\kern-\dimen7 \lower\dimen6\box\z@}\ht0\z@\dp0\z@\raise
\axisheight\box0 }\def\CD@mK{\setbox0\hbox{\ht\z@\z@\dp\z@\z@\wd\z@\z@\CD@hK
\expandafter\CD@tK{q \CD@eK\space\CD@lK\space\CD@kK\space\CD@eK\space0 0 cm}%
\else\global\CD@iG\CD@eD{\the\CD@TC\space\ifPositiveGradient\else-\fi\the
\CD@LH\space bturn}\fi\box\z@\CD@gK}}\def\CD@vB{\advance\CD@hF-\CD@mI\CD@wJ
\CD@hF\advance\CD@wJ\CD@hI\ifvoid\CD@sH\ifdim\CD@wJ<.1em\ifnum\CD@gD=\@m\else
\CD@aG h\CD@wJ<.1em:objects overprint:\CD@FA\CD@gD\fi\fi\else\ifhbox\CD@sH
\CD@SK\else\CD@TK\fi\advance\CD@wJ\CD@mI\CD@bH{-\CD@mI}{\box\CD@sH}{\CD@wJ}%
\z@\fi\CD@hF-\CD@mF\CD@gD\CD@FA\CD@hI\z@}\def\CD@SK{\setbox\CD@sH=\hbox{%
\unhbox\CD@sH\unskip\unpenalty}\setbox\CD@tH=\hbox{\unhbox\CD@tH\unskip
\unpenalty}\setbox\CD@sH=\hbox to\CD@wJ{\CD@OA\wd\CD@sH\unhbox\CD@sH\CD@PA
\lastkern\unkern\ifdim\CD@PA=\z@\CD@UB\advance\CD@OA-\wd\CD@tH\else\CD@TB\fi
\ifnum\lastpenalty=\z@\else\CD@JA\unpenalty\fi\kern\CD@PA\ifdim\CD@hF<\CD@OA
\CD@JA\fi\ifdim\CD@hI<\wd\CD@tH\CD@JA\fi\CD@jE\CD@hI\CD@wJ\advance\CD@hI-%
\CD@OA\advance\CD@hI\wd\CD@tH\ifdim\CD@hI<2\wd\CD@tH\CD@aG h\CD@hI<2\wd\CD@tH
:arrow too short:\CD@FA\CD@gD\fi\divide\CD@hI\tw@\CD@hF\CD@wJ\advance\CD@hF-%
\CD@hI\fi\CD@tE\kern-\CD@hI\fi\hbox to\CD@hI{\unhbox\CD@tH}\CD@HG}}\CD@tG
\ifinpile\inpiletrue\inpilefalse\inpilefalse\def\pile{\protect\CD@UJ\protect
\CD@uH}\def\CD@uH#1{\CD@l#1\CD@QD}\def\CD@UJ{\CD@nB{pile}\setbox0=\vtop
\bgroup\aftergroup\CD@lD\inpiletrue\let\CD@FE\empty\let\pile\CD@KF\let\CD@QD
\CD@PD\let\CD@GD\CD@FD\CD@yH\baselineskip.5\PileSpacing\lineskip.1\CD@zC
\relax\lineskiplimit\lineskip\mathsurround\z@\tabskip\z@\let\\\CD@wH}\def
\CD@l{\CD@DE\CD@YF\empty\halign\bgroup\hfil\CD@k\let\CD@FE\CD@d\let\\\CD@vH##%
\CD@MD\CD@ND\hfil\CD@Q\CD@R##\cr}\CD@rG\CD@NE{pile only allows one column.}%
\CD@rG\CD@UE{you left it out!}\def\CD@R{\CD@QD\CD@Q\relax\CD@YA{missing \CD@yC
\space inserted after \string\pile}\CD@NE}\def\CD@PD{\CD@MD\crcr\egroup
\egroup}\def\CD@GD{\CD@MD}\def\CD@FD{\CD@MD\relax\CD@QD\CD@YA{missing \CD@yC
\space inserted between \string\pile\space and \CD@HD}\CD@UE}\def\CD@QD{%
\CD@MD}\def\CD@lD{\vbox{\dimen1\dp0 \unvbox0 \setbox0=\lastbox\advance\dimen1%
\dp0 \nointerlineskip\box0 \nointerlineskip\setbox0=\null\dp0.5\dimen1\ht0-%
\dp0 \box0}\ifincommdiag\CD@tJ\penalty-9998 \fi\xdef\CD@YF{pile}}\def\CD@vH{%
\cr}\def\CD@wH{\noalign{\skip@\prevdepth\advance\skip@-\baselineskip
\prevdepth\skip@}}\def\CD@KF#1{#1}\def\CD@TK{\setbox\CD@sH=\vbox{\unvbox
\CD@sH\setbox1=\lastbox\setbox0=\box\voidb@x\CD@tF\setbox\CD@sH=\lastbox
\ifhbox\CD@sH\CD@rC\repeat\unvbox0 \global\CD@QA\CD@ZE}\CD@ZE\CD@QA}\def
\CD@rC{\CD@jE\setbox\CD@sH=\hbox{\unhbox\CD@sH\unskip\setbox\CD@sH=\lastbox
\unskip\unhbox\CD@sH}\ifdim\CD@wJ<\wd\CD@sH\CD@aG h\CD@wJ<\wd\CD@sH:arrow in
pile too short:\CD@FA\CD@gD\else\setbox\CD@sH=\hbox to\CD@wJ{\unhbox\CD@sH}%
\fi\else\CD@gJ\fi\setbox0=\vbox{\box\CD@sH\nointerlineskip\ifvoid0 \CD@tJ\box
1 \else\vskip\skip0 \unvbox0 \fi}\skip0=\lastskip\unskip}\def\CD@gJ{\penalty7
\noindent\unhbox\CD@sH\unskip\setbox\CD@sH=\lastbox\unskip\unhbox\CD@sH
\endgraf\setbox\CD@tH=\lastbox\unskip\setbox\CD@tH=\hbox{\CD@JG\unhbox\CD@tH
\unskip\unskip\unpenalty}\ifcase\prevgraf\cd@shouldnt P\or\ifdim\CD@wJ<\wd
\CD@tH\CD@aG h\CD@wJ<\wd\CD@sH:object in pile too wide:\CD@FA\CD@gD\setbox
\CD@sH=\hbox to\CD@wJ{\hss\unhbox\CD@tH\hss}\else\setbox\CD@sH=\hbox to\CD@wJ
{\hss\kern\CD@hF\unhbox\CD@tH\kern\CD@hI\hss}\fi\or\setbox\CD@sH=\lastbox
\unskip\CD@SK\else\cd@shouldnt Q\fi\unskip\unpenalty}\def\CD@cD{\CD@MJ\ifvoid
3 \setbox3=\null\ht3\axisheight\dp3-\ht3 \dimen3.5\CD@LF\else\dimen4\dp3
\dimen3.5\wd3 \setbox3=\CD@GG{\box3}\dp3\dimen4 \ifdim\ht3=-\dp3 \else\CD@TB
\fi\fi\dimen0\dimen3 \advance\dimen0-.5\CD@LF\setbox0\null\ht0\ht3\dp0\dp3\wd
0\wd3 \ifvoid6\else\setbox6\hbox{\unhbox6\kern\dimen0\kern2pt}\dimen0\wd6 \fi
\ifvoid7\else\setbox7\hbox{\kern2pt\kern\dimen3\unhbox7}\dimen3\wd7 \fi
\setbox3\hbox{\ifvoid6\else\kern-\dimen0\unhbox6\fi\unhbox3 \ifvoid7\else
\unhbox7\kern-\dimen3\fi}\ht3\ht0\dp3\dp0\wd3\wd0 \CD@tE\dimen4=\ht\CD@MH
\advance\dimen4\dp5 \advance\dimen4\dimen1 \let\CD@jD\empty\else\dimen4\ht3
\fi\setbox0\null\ht0\dimen4 \offinterlineskip\setbox8=\vbox spread2ex{\kern
\dimen5 \box1 \CD@iD\vfill\CD@tE\else\kern\CD@eJ\fi\box0}\ht8=\z@\setbox9=%
\vtop spread2ex{\kern-\ht3 \kern-\CD@eJ\box3 \CD@jD\vfill\box5 \kern\dimen1}%
\dp9=\z@\hskip\dimen0plus.0001fil \box9 \kern-\CD@LF\box8 \CD@kE\penalty2 \fi
\CD@tE\penalty1 \fi\kern\PileSpacing\kern-\PileSpacing\kern-.5\CD@LF\penalty
\CD@GB\null\kern\dimen3}\def\CD@cI{\ifhbox\CD@VA\CD@KB{clashing verticals}\ht
\CD@MH.5\dp\CD@VA\dp\CD@MH-\ht5 \CD@yB\ht\CD@MH\z@\dp\CD@MH\z@\fi\dimen1\dp
\CD@VA\CD@xA\prevgraf\unvbox\CD@VA\CD@wA\lastpenalty\unpenalty\setbox\CD@VA=%
\null\setbox\CD@lI=\hbox{\CD@JG\unhbox\CD@lI\unskip\unpenalty\dimen0\lastkern
\unkern\unkern\unkern\kern\dimen0 \CD@HG}\setbox\CD@lF=\hbox{\unhbox\CD@lF
\dimen0\lastkern\unkern\unkern\global\CD@QA\lastkern\unkern\kern\dimen0 }%
\CD@tF\ifnum\CD@xA>4 \CD@zI\repeat\unskip\unskip\advance\CD@mF.5\wd\CD@VA
\advance\CD@mF\wd\CD@lF\advance\CD@mI.5\wd\CD@VA\advance\CD@mI\wd\CD@lI\ifnum
\CD@FA=\CD@lA\CD@OA.5\wd\CD@VA\edef\CD@NK{\the\CD@OA}\fi\setbox\CD@VA=\hbox{%
\kern-\CD@mF\box\CD@lF\unhbox\CD@VA\box\CD@lI\kern-\CD@mI\penalty\CD@wA
\penalty\CD@NB}\ht\CD@VA\dimen1 \dp\CD@VA\z@\wd\CD@VA\CD@tB\CD@vB}\def\CD@zI{%
\ifdim\wd\CD@lF<\CD@QA\setbox\CD@lF=\hbox to\CD@QA{\CD@JG\unhbox\CD@lF}\fi
\advance\CD@xA\m@ne\setbox\CD@VA=\hbox{\box\CD@lF\unhbox\CD@VA}\unskip\setbox
\CD@lF=\lastbox\setbox\CD@lF=\hbox{\unhbox\CD@lF\unskip\unpenalty\dimen0%
\lastkern\unkern\unkern\global\CD@QA\lastkern\unkern\kern\dimen0 }}\def\CD@yB
{\dimen1\dp\CD@VA\ifhbox\CD@VA\CD@xB\else\CD@zB\fi\setbox\CD@VA=\vbox{%
\penalty\CD@NB}\dp\CD@VA-\dp\CD@MH\wd\CD@VA\CD@tB}\def\CD@zB{\unvbox\CD@VA
\CD@wA\lastpenalty\unpenalty\ifdim\dimen1<\ht\CD@MH\CD@aG v\dimen1<\ht\CD@MH:%
rows overprint:\CD@NB\CD@wA\fi}\def\CD@xB{\dimen0=\ht\CD@VA\setbox\CD@VA=%
\hbox\bgroup\advance\dimen1-\ht\CD@MH\unhbox\CD@VA\CD@xA\lastpenalty
\unpenalty\CD@wA\lastpenalty\unpenalty\global\CD@RA-\lastkern\unkern\setbox0=%
\lastbox\CD@tF\setbox\CD@VA=\hbox{\box0\unhbox\CD@VA}\setbox0=\lastbox\ifhbox
0 \CD@kJ\repeat\global\CD@SA-\lastkern\unkern\global\CD@QA\CD@JK\unhbox\CD@VA
\egroup\CD@JK\CD@QA\CD@bH{\CD@SA}{\box\CD@VA}{\CD@RA}{\dimen1}}\def\CD@kJ{%
\setbox0=\hbox to\wd0\bgroup\unhbox0 \unskip\unpenalty\dimen7\lastkern\unkern
\ifnum\lastpenalty=1 \unpenalty\CD@UB\else\CD@TB\fi\ifnum\lastpenalty=2
\unpenalty\dimen2.5\dimen0\advance\dimen2-.5\dimen1\advance\dimen2-%
\axisheight\else\dimen2\z@\fi\setbox0=\lastbox\dimen6\lastkern\unkern\setbox1%
=\lastbox\setbox0=\vbox{\unvbox0 \CD@tE\kern-\dimen1 \else\ifdim\dimen2=\z@
\else\kern\dimen2 \fi\fi}\ifdim\dimen0<\ht0 \CD@aG v\dimen0<\ht0:upper part of
vertical too short:{\CD@tE\CD@NB\else\CD@wA\fi}\CD@xA\else\setbox0=\vbox to%
\dimen0{\unvbox0}\fi\setbox1=\vtop{\unvbox1}\ifdim\dimen1<\dp1 \CD@aG v\dimen
1<\dp1:lower part of vertical too short:\CD@NB\CD@wA\else\setbox1=\vtop to%
\dimen1{\ifdim\dimen2=\z@\else\kern-\dimen2 \fi\unvbox1 }\fi\box1 \kern\dimen
6 \box0 \kern\dimen7 \CD@HG\global\CD@QA\CD@JK\egroup\CD@JK\CD@QA\relax}%
\countdef\CD@u=14 \newcount\CD@CA\newcount\CD@XB\newcount\CD@NB\let\CD@LB
\insc@unt\newcount\CD@FA\newcount\CD@lA\let\CD@mA\CD@XB\newcount\CD@MB\CD@tG
\CD@DF\CD@bI\CD@aI\CD@aI\def\CD@nD{-1}\def\CD@K{\ifnum\CD@nD<\z@\else
\begingroup\scrollmode\showboxdepth\CD@nD\showboxbreadth\maxdimen\showlists
\endgroup\fi\CD@bI\CD@zF\CD@CA=\CD@u\advance\CD@CA1 \CD@XB=\CD@CA\ifnum\CD@NB
=1 \CD@JA\fi\advance\CD@XB\CD@NB\dimen1\z@\skip0\z@\count@=\insc@unt\advance
\count@\CD@u\divide\count@2 \ifnum\CD@XB>\count@\CD@KB{The diagram has too
many rows! It can't be reformatted.}\else\CD@NG\CD@WI\fi\CD@cH}\def\CD@NG{%
\CD@NB\CD@CA\CD@uF\ifnum\CD@NB<\CD@XB\setbox\CD@NB\box\voidb@x\advance\CD@NB1%
\relax\repeat\CD@NB\CD@CA\skip\z@\z@\CD@uF\CD@GB\lastpenalty\unpenalty\ifnum
\CD@GB>\z@\CD@KE\repeat\ifnum\CD@GB=-123 \CD@tJ\unpenalty\else\cd@shouldnt D%
\fi\ifx\v@grid\relax\else\CD@NB\CD@XB\advance\CD@NB\m@ne\expandafter\CD@VJ
\v@grid\fi\CD@MB\CD@mA\CD@tB\z@\CD@XG\ifx\h@grid\relax\else\expandafter\CD@LJ
\h@grid\fi\count@\CD@XB\advance\count@\m@ne\CD@YB\ht\count@}\def\CD@KE{%
\ifcase\CD@GB\or\CD@MG\else\CD@uA-\lastpenalty\unpenalty\CD@vA\lastpenalty
\unpenalty\setbox0=\lastbox\CD@WG\fi\CD@wD}\def\CD@wD{\skip1\lastskip\unskip
\advance\skip0\skip1 \ifdim\skip1=\z@\else\expandafter\CD@wD\fi}\def\CD@MG{%
\setbox0=\lastbox\CD@pI\dp0 \advance\CD@pI\skip\z@\skip\z@\z@\advance\CD@NF
\CD@pI\CD@uE\ifnum\CD@NB>\CD@CA\CD@NF\DiagramCellHeight\CD@pI\CD@NF\advance
\CD@pI-\CD@qI\fi\fi\CD@qI\ht0 \CD@NF\CD@qI\setbox\CD@NB\hbox{\unhbox\CD@NB
\unhbox0}\dp\CD@NB\CD@pI\ht\CD@NB\CD@qI\advance\CD@NB1 }\def\CD@WG{\ifnum
\CD@uA<\z@\advance\CD@uA\CD@XB\ifnum\CD@uA<\CD@CA\CD@UG\else\CD@OA\dp\CD@uA
\CD@PA\ht\CD@uA\setbox\CD@uA\hbox{\box\z@\penalty\CD@vA\penalty\CD@GB\unhbox
\CD@uA}\dp\CD@uA\CD@OA\ht\CD@uA\CD@PA\fi\else\CD@UG\fi}\def\CD@UG{\CD@KB{%
diagonal goes outside diagram (lost)}}\def\CD@fI{\advance\CD@uA\CD@XB\ifnum
\CD@uA<\CD@CA\CD@UG\else\ifnum\CD@uA=\CD@NB\CD@VG\else\ifnum\CD@uA>\CD@NB
\cd@shouldnt M\else\CD@OA\dp\CD@uA\CD@PA\ht\CD@uA\setbox\CD@uA\hbox{\box\z@
\penalty\CD@vA\penalty\CD@GB\unhbox\CD@uA}\dp\CD@uA\CD@OA\ht\CD@uA\CD@PA\fi
\fi\fi}\def\CD@WI{\CD@t\CD@AJ\setbox\CD@PC=\hbox{\CD@k A\@super f\CD@lJ f%
\CD@ND}\CD@ZE\z@\CD@JK\z@\CD@kI\z@\CD@kF\z@\CD@NB=\CD@XB\CD@NF\z@\CD@uB\z@
\CD@uF\ifnum\CD@NB>\CD@CA\advance\CD@NB\m@ne\CD@qI\ht\CD@NB\CD@pI\dp\CD@NB
\advance\CD@NF\CD@qI\CD@rI\advance\CD@uB\CD@NF\CD@KC\CD@ZI\CD@w\ht\CD@NB
\CD@qI\dp\CD@NB\CD@pI\nointerlineskip\box\CD@NB\CD@NF\CD@pI\setbox\CD@NB\null
\ht\CD@NB\CD@uB\repeat\CD@wB\nointerlineskip\box\CD@NB\CD@gG\CD@ZE
\DiagramCellWidth{width}\CD@gG\CD@JK\DiagramCellHeight{height}\CD@VA\CD@LB
\advance\CD@VA-\CD@lA\advance\CD@VA\m@ne\advance\CD@VA\CD@mA\dimen0\wd\CD@VA
\CD@tI\axisheight\dimen1\CD@uB\advance\dimen1-\CD@YB\dimen2\CD@kI\advance
\dimen2-\dimen0 \advance\CD@XB-\CD@CA\advance\CD@LB-\CD@lA}\count@\year
\multiply\count@12 \advance\count@\month\ifnum\count@>24254 \loop\iftrue
\repeat\fi\def\CD@wB{\CD@qI-\CD@NF\CD@pI\CD@NF
\setbox\CD@MH=\null\dp\CD@MH\CD@NF\ht\CD@MH-\CD@NF\CD@mF\z@\CD@mI\z@\CD@lA
\CD@LB\advance\CD@lA-\CD@MB\advance\CD@lA\CD@mA\CD@FA\CD@LB\CD@VA\CD@MB\CD@sF
\ifnum\CD@FA>\CD@lA\advance\CD@FA\m@ne\advance\CD@VA\m@ne\CD@tB\wd\CD@VA
\setbox\CD@FA=\box\voidb@x\CD@yB\repeat\CD@w\ht\CD@NB\CD@qI\dp\CD@NB\CD@pI}%
\def\CD@gG#1#2#3{\ifdim#1>.01\CD@zC\CD@PA#2\relax\advance\CD@PA#1\relax
\advance\CD@PA.99\CD@zC\count@\CD@PA\divide\count@\CD@zC\CD@KB{increase cell #%
3 to \the\count@ em}\fi}\def\CD@rI{\CD@FA=\CD@LB\penalty4 \noindent\unhbox
\CD@NB\CD@sF\unskip\setbox0=\lastbox\ifhbox0 \advance\CD@FA\m@ne\setbox\CD@FA
\hbox to\wd0{\null\penalty-9990\null\unhbox0}\repeat\CD@lA\CD@FA\advance
\CD@FA\CD@MB\advance\CD@FA-\CD@mA\ifnum\CD@FA<\CD@LB\count@\CD@FA\advance
\count@\m@ne\dimen0=\wd\count@\count@\CD@MB\advance\count@\m@ne\CD@tB\wd
\count@\CD@sF\ifnum\CD@FA<\CD@LB\CD@DJ\CD@XG\dimen0\wd\CD@FA\advance\CD@FA1
\repeat\fi\CD@sF\CD@GB\lastpenalty\unpenalty\ifnum\CD@GB>\z@\CD@vA
\lastpenalty\unpenalty\CD@VG\repeat\endgraf\unskip\ifnum\lastpenalty=4
\unpenalty\else\cd@shouldnt S\fi}\def\CD@VG{\advance\CD@vA\CD@lA\advance
\CD@vA\m@ne\setbox0=\lastbox\ifnum\CD@vA<\CD@LB\setbox\CD@vA\hbox{\box0%
\penalty\CD@GB\unhbox\CD@vA}\else\CD@UG\fi}\def\CD@bG{}\CD@tG\CD@uE\CD@WB
\CD@VB\def\CD@DJ{\advance\dimen0\wd\CD@FA\divide\dimen0\tw@\CD@uE\dimen0%
\DiagramCellWidth\else\CD@V{\dimen0}\DiagramCellWidth\CD@pJ\fi\advance\CD@tB
\dimen0 }\def\CD@XG{\setbox\CD@MB=\vbox{}\dp\CD@MB=\CD@uB\wd\CD@MB\CD@tB
\advance\CD@MB1 }\def\CD@LJ#1,{\def\CD@GK{#1}\ifx\CD@GK\CD@RD\else\advance
\CD@tB\CD@GK\DiagramCellWidth\CD@XG\expandafter\CD@LJ\fi}\def\CD@VJ#1,{\def
\CD@GK{#1}\ifx\CD@GK\CD@RD\else\ifnum\CD@NB>\CD@CA\CD@NF\CD@GK
\DiagramCellHeight\advance\CD@NF-\dp\CD@NB\advance\CD@NB\m@ne\ht\CD@NB\CD@NF
\fi\expandafter\CD@VJ\fi}\def\CD@pJ{\CD@wE\CD@OA\dimen0 \advance\CD@OA-%
\DiagramCellWidth\ifdim\CD@OA>2\MapShortFall\CD@KB{badly drawn diagonals (see
manual)}\let\CD@pJ\empty\fi\else\let\CD@pJ\empty\fi}\def\CD@KC{\CD@VA\CD@mA
\CD@sF\ifnum\CD@VA<\CD@MB\dimen0\dp\CD@VA\advance\dimen0\CD@NF\dp\CD@VA\dimen
0 \advance\CD@VA1 \repeat}\def\CD@bH#1#2#3#4{\ifnum\CD@FA<\CD@LB\CD@OA=#1%
\relax\setbox\CD@FA=\hbox{\setbox0=#2\dimen7=#4\relax\dimen8=#3\relax\ifhbox
\CD@FA\unhbox\CD@FA\advance\CD@OA-\lastkern\unkern\fi\ifdim\CD@OA=\z@\else
\kern-\CD@OA\fi\raise\dimen7\box0 \kern-\dimen8 }\ifnum\CD@FA=\CD@lA\CD@V
\CD@kF\CD@OA\fi\else\cd@shouldnt O\fi}\def\CD@w{\setbox\CD@NB=\hbox{\CD@FA
\CD@lA\CD@VA\CD@mA\CD@PA\z@\relax\CD@sF\ifnum\CD@FA<\CD@LB\CD@tB\wd\CD@VA
\relax\CD@eI\advance\CD@FA1 \advance\CD@VA1 \repeat}\CD@V\CD@kI{\wd\CD@NB}\wd
\CD@NB\z@}\def\CD@eI{\ifhbox\CD@FA\CD@OA\CD@tB\relax\advance\CD@OA-\CD@PA
\relax\ifdim\CD@OA=\z@\else\kern\CD@OA\fi\CD@PA\CD@tB\advance\CD@PA\wd\CD@FA
\relax\unhbox\CD@FA\advance\CD@PA-\lastkern\unkern\fi}\def\CD@ZI{\setbox
\CD@sH=\box\voidb@x\CD@VA=\CD@MB\CD@FA\CD@LB\CD@VA\CD@mA\advance\CD@VA\CD@FA
\advance\CD@VA-\CD@lA\advance\CD@VA\m@ne\CD@tB\wd\CD@VA\count@\CD@LB\advance
\count@\m@ne\CD@hF.5\wd\count@\advance\CD@hF\CD@tB\CD@A\m@ne\CD@gD\@m\CD@sF
\ifnum\CD@FA>\CD@lA\advance\CD@FA\m@ne\advance\CD@hF-\CD@tB\CD@PI\wd\CD@VA
\CD@tB\advance\CD@hF\CD@tB\advance\CD@VA\m@ne\CD@tB\wd\CD@VA\repeat\CD@mF
\CD@kF\CD@mI-\CD@mF\CD@vB}\newcount\CD@GB\def\CD@s{}\def\CD@t{\mathsurround
\z@\hsize\z@\rightskip\z@ plus1fil minus\maxdimen\parfillskip\z@\linepenalty
9000 \looseness0 \hfuzz\maxdimen\hbadness10000 \clubpenalty0 \widowpenalty0
\displaywidowpenalty0 \interlinepenalty0 \predisplaypenalty0
\postdisplaypenalty0 \interdisplaylinepenalty0 \interfootnotelinepenalty0
\floatingpenalty0 \brokenpenalty0 \everypar{}\leftskip\z@\parskip\z@
\parindent\z@\pretolerance10000 \tolerance10000 \hyphenpenalty10000
\exhyphenpenalty10000 \binoppenalty10000 \relpenalty10000 \adjdemerits0
\doublehyphendemerits0 \finalhyphendemerits0 \baselineskip\z@\CD@IA\prevdepth
\z@}\newbox\CD@KG\newbox\CD@IG\def\CD@JG{\unhcopy\CD@KG}\def\CD@HG{\unhcopy
\CD@IG}\def\CD@iJ{\hbox{}\penalty1\nointerlineskip}\def\CD@PI{\penalty5
\noindent\setbox\CD@MH=\null\CD@mF\z@\CD@mI\z@\ifnum\CD@FA<\CD@LB\ht\CD@MH\ht
\CD@FA\dp\CD@MH\dp\CD@FA\unhbox\CD@FA\skip0=\lastskip\unskip\else\CD@OK\skip0%
=\z@\fi\endgraf\ifcase\prevgraf\cd@shouldnt Y \or\cd@shouldnt Z \or\CD@RI\or
\CD@XI\else\CD@QI\fi\unskip\setbox0=\lastbox\unskip\unskip\unpenalty\noindent
\unhbox0\setbox0\lastbox\unpenalty\unskip\unskip\unpenalty\setbox0\lastbox
\CD@tF\CD@GB\lastpenalty\unpenalty\ifnum\CD@GB>\z@\setbox\z@\lastbox\CD@lB
\repeat\endgraf\unskip\unskip\unpenalty}\def\CD@YJ{\CD@uA\CD@XB\advance\CD@uA
-\CD@NB\CD@vA\CD@FA\advance\CD@vA-\CD@lA\advance\CD@vA1 \expandafter\message{%
prevgraf=\the\prevgraf at (\the\CD@uA,\the\CD@vA)}}\def\CD@XI{\CD@CE\setbox
\CD@lI=\lastbox\setbox\CD@lI=\hbox{\unhbox\CD@lI\unskip\unpenalty}\unskip
\ifdim\ht\CD@lI>\ht\CD@PC\setbox\CD@MH=\copy\CD@lI\else\ifdim\dp\CD@lI>\dp
\CD@PC\setbox\CD@MH=\copy\CD@lI\else\CD@FG\CD@lI\fi\fi\advance\CD@mF.5\wd
\CD@lI\advance\CD@mI.5\wd\CD@lI\setbox\CD@lI=\hbox{\unhbox\CD@lI\CD@HG}\CD@bH
\CD@mF{\box\CD@lI}\CD@mI\z@\CD@yB\CD@vB}\def\CD@CE{\ifnum\CD@A>0 \advance
\dimen0-\CD@tB\CD@iA-.5\dimen0 \CD@A-\CD@A\else\CD@A0 \CD@iA\z@\fi\setbox
\CD@MH=\lastbox\setbox\CD@MH=\hbox{\unhbox\CD@MH\unskip\unskip\unpenalty
\setbox0=\lastbox\global\CD@QA\lastkern\unkern}\advance\CD@iA-.5\CD@QA\unskip
\setbox\CD@MH=\null\CD@mI\CD@iA\CD@mF-\CD@iA}\def\CD@Z{\ht\CD@MH\CD@tI\dp
\CD@MH\CD@sI}\def\CD@FG#1{\setbox\CD@MH=\hbox{\CD@V{\ht\CD@MH}{\ht#1}\CD@V{%
\dp\CD@MH}{\dp#1}\CD@V{\wd\CD@MH}{\wd#1}\vrule height\ht\CD@MH depth\dp\CD@MH
width\wd\CD@MH}}\def\CD@QI{\CD@CE\CD@Z\setbox\CD@lI=\lastbox\unskip\setbox
\CD@lF=\lastbox\unskip\setbox\CD@lF=\hbox{\unhbox\CD@lF\unskip\global\CD@yA
\lastpenalty\unpenalty}\advance\CD@yA9999 \ifcase\CD@yA\CD@VI\or\CD@YI\or
\CD@TI\or\CD@dI\or\CD@cI\or\CD@SI\else\cd@shouldnt9\fi}\def\CD@VI{\CD@FG
\CD@lI\CD@UI\setbox\CD@sH=\box\CD@lF\setbox\CD@tH=\box\CD@lI}\def\CD@YI{%
\CD@FG\CD@lF\setbox\CD@lI\hbox{\penalty8 \unhbox\CD@lI\unskip\unpenalty\ifnum
\lastpenalty=8 \else\CD@xH\fi}\CD@UI\setbox\CD@lF=\hbox{\unhbox\CD@lF\unskip
\unpenalty\global\setbox\CD@DA=\lastbox}\ifdim\wd\CD@lF=\z@\else\CD@xH\fi
\setbox\CD@sH=\box\CD@DA}\def\CD@xH{\CD@KB{extra material in \string\pile
\space cell (lost)}}\def\CD@UI{\CD@yB\ifvoid\CD@sH\else\CD@KB{Clashing
horizontal arrows}\CD@mI.5\CD@hF\CD@mF-\CD@mI\CD@vB\CD@mI\z@\CD@mF\z@\fi
\CD@hI\CD@hF\advance\CD@hI-\CD@mI\CD@hF-\CD@mF\CD@JC\CD@FA}\def\CD@RI{\setbox
0\lastbox\unskip\CD@iA\z@\CD@Z\ifdim\skip0>\z@\CD@tJ\CD@A0 \else\ifnum\CD@A<1
\CD@A0 \dimen0\CD@tB\fi\advance\CD@A1 \fi}\def\VonH{\CD@MA46\VonH{.5\CD@LF}}%
\def\HonV{\CD@MA57\HonV{.5\CD@LF}}\def\HmeetV{\CD@MA44\HmeetV{-\MapShortFall}%
}\def\CD@MA#1#2#3#4{\CD@pB34#1{\string#3}\CD@SD\CD@GB-999#2 \dimen0=#4\CD@tI
\dimen0\advance\CD@tI\axisheight\CD@sI\dimen0\advance\CD@sI-\axisheight\CD@CF
\CD@HC\CD@ZD}\def\CD@HC#1{\setbox0=\hbox{\CD@k#1\CD@ND}\dimen0.5\wd0 \CD@tI
\ht0 \CD@sI\dp0 \CD@ZD}\def\CD@SD{\setbox0=\null\ht0=\CD@tI\dp0=\CD@sI\wd0=%
\dimen0 \copy0\penalty\CD@GB\box0 }\def\CD@TI{\CD@GC\CD@yB}\def\CD@dI{\CD@GC
\CD@vB}\def\CD@SI{\CD@GC\CD@yB\CD@vB}\def\CD@GC{\setbox\CD@lI=\hbox{\unhbox
\CD@lI}\setbox\CD@lF=\hbox{\unhbox\CD@lF\global\setbox\CD@DA=\lastbox}\ht
\CD@MH\ht\CD@DA\dp\CD@MH\dp\CD@DA\advance\CD@mF\wd\CD@DA\advance\CD@mI\wd
\CD@lI}\CD@tG\ifPositiveGradient\CD@CI\CD@BI\CD@CI\CD@tG\ifClimbing\CD@rB
\CD@qB\CD@rB\newcount\DiagonalChoice\DiagonalChoice\m@ne\ifx\tenln\nullfont
\CD@tJ\def\CD@qF{\CD@KH\ifPositiveGradient/\else\CD@k\backslash\CD@ND\fi}%
\else\def\CD@qF{\CD@rF\char\count@}\fi\let\CD@rF\tenln\def\Use@line@char#1{%
\hbox{#1\CD@rF\ifPositiveGradient\else\advance\count@64 \fi\char\count@}}\def
\CD@cF{\Use@line@char{\count@\CD@TC\multiply\count@8\advance\count@-9\advance
\count@\CD@LH}}\def\CD@ZF{\Use@line@char{\ifcase\DiagonalChoice\CD@gF\or
\CD@fF\or\CD@fF\else\CD@gF\fi}}\def\CD@gF{\ifnum\CD@TC=\z@\count@'33 \else
\count@\CD@TC\multiply\count@\sixt@@n\advance\count@-9\advance\count@\CD@LH
\advance\count@\CD@LH\fi}\def\CD@fF{\count@'\ifcase\CD@LH55\or\ifcase\CD@TC66%
\or22\or52\or61\or72\fi\or\ifcase\CD@TC66\or25\or22\or63\or52\fi\or\ifcase
\CD@TC66\or16\or36\or22\or76\fi\or\ifcase\CD@TC66\or27\or25\or67\or22\fi\fi
\relax}\def\CD@uC#1{\hbox{#1\setbox0=\Use@line@char{#1}\ifPositiveGradient
\else\raise.3\ht0\fi\copy0 \kern-.7\wd0 \ifPositiveGradient\raise.3\ht0\fi
\box0}}\def\CD@jF#1{\hbox{\setbox0=#1\kern-.75\wd0 \vbox to.25\ht0{%
\ifPositiveGradient\else\vss\fi\box0 \ifPositiveGradient\vss\fi}}}\def\CD@jI#%
1{\hbox{\setbox0=#1\dimen0=\wd0 \vbox to.25\ht0{\ifPositiveGradient\vss\fi
\box0 \ifPositiveGradient\else\vss\fi}\kern-.75\dimen0 }}\CD@RC{+h:>}{%
\Use@line@char\CD@fF}\CD@RC{-h:>}{\Use@line@char\CD@gF}\CD@nF{+t:<}{-h:>}%
\CD@nF{-t:<}{+h:>}\CD@RC{+t:>}{\CD@jF{\Use@line@char\CD@fF}}\CD@RC{-t:>}{%
\CD@jI{\Use@line@char\CD@gF}}\CD@nF{+h:<}{-t:>}\CD@nF{-h:<}{+t:>}\CD@RC{+h:>>%
}{\CD@uC\CD@fF}\CD@RC{-h:>>}{\CD@uC\CD@gF}\CD@nF{+t:<<}{-h:>>}\CD@nF{-t:<<}{+%
h:>>}\CD@nF{+h:>->}{+h:>>}\CD@nF{-h:>->}{-h:>>}\CD@nF{+t:<-<}{-h:>>}\CD@nF{-t%
:<-<}{+h:>>}\CD@RC{+t:>>}{\CD@jF{\CD@uC\CD@fF}}\CD@RC{-t:>>}{\CD@jI{\CD@uC
\CD@gF}}\CD@nF{+h:<<}{-t:>>}\CD@nF{-h:<<}{+t:>>}\CD@nF{+t:>->}{+t:>>}\CD@nF{-%
t:>->}{-t:>>}\CD@nF{+h:<-<}{-t:>>}\CD@nF{-h:<-<}{+t:>>}\CD@RC{+f:-}{\CD@EF
\null\else\CD@cF\fi}\CD@nF{-f:-}{+f:-}\def\CD@tC#1#2{\vbox to#1{\vss\hbox to#%
2{\hss.\hss}\vss}}\def\hfdot{\CD@tC{2\axisheight}{.5em}}%
\def\vfdot{\CD@tC{1ex}\z@}
\def\CD@bF{\hbox{\dimen0=.3\CD@zC\dimen1\dimen0 \ifnum\CD@LH>\CD@TC\CD@iC{%
\dimen1}\else\CD@dG{\dimen0}\fi\CD@tC{\dimen0}{\dimen1}}}\newarrowfiller{.}%
\hfdot\hfdot\vfdot\vfdot\def\dfdot{\CD@bF\CD@CK}\CD@RC{+f:.}{\dfdot}\CD@RC{-f%
:.}{\dfdot}\def\CD@@K#1{\hbox\bgroup\def\CD@CH{#1\egroup}\afterassignment
\CD@CH
\count@='}\def\lnchar{\CD@@K\CD@qF}\def\CD@dF#1{\setbox#1=\hbox{\dimen5\dimen
#1 \setbox8=\box#1 \dimen1\wd8 \count@\dimen5 \divide\count@\dimen1 \ifnum
\count@=0 \box8 \ifdim\dimen5<.95\dimen1 \CD@gB{diagonal line too short}\fi
\else\dimen3=\dimen5 \advance\dimen3-\dimen1 \divide\dimen3\count@\dimen4%
\dimen3 \CD@dG{\dimen4}\ifPositiveGradient\multiply\dimen4\m@ne\fi\dimen6%
\dimen1 \advance\dimen6-\dimen3 \loop\raise\count@\dimen4\copy8 \ifnum\count@
>0 \kern-\dimen6 \advance\count@\m@ne\repeat\fi}}\def\CD@CG#1{\CD@EF\CD@xJ{#1%
}\else\CD@dF{#1}\fi}\def\CD@IH#1{}\newdimen\objectheight\objectheight1.8ex
\newdimen\objectwidth\objectwidth1em \def\CD@YD{\dimen6=\CD@aK
\DiagramCellHeight\dimen7=\CD@WK\DiagramCellWidth\CD@KJ\ifnum\CD@LH>0 \ifnum
\CD@TC>0 \CD@aF\else\aftergroup\CD@VC\fi\else\aftergroup\CD@UC\fi}\def\CD@VC{%
\CD@YA{diagonal map is nearly vertical}\CD@NA}\def\CD@UC{\CD@YA{diagonal map
is nearly horizontal}\CD@NA}\CD@rG\CD@NA{Use an orthogonal map instead}\def
\CD@aF{\CD@MJ\dimen3\dimen7\dimen7\dimen6\CD@iC{\dimen7}\advance\dimen3-%
\dimen7 \CD@MF\ifnum\CD@LH>\CD@TC\advance\dimen6-\dimen1\advance\dimen6-%
\dimen5 \CD@iC{\dimen1}\CD@iC{\dimen5}\else\dimen0\dimen1\advance\dimen0%
\dimen5\CD@dG{\dimen0}\advance\dimen6-\dimen0 \fi\dimen2.5\dimen7\advance
\dimen2-\dimen1 \dimen4.5\dimen7\advance\dimen4-\dimen5 \ifPositiveGradient
\dimen0\dimen5 \advance\dimen1-\CD@WK\DiagramCellWidth\advance\dimen1 \CD@ZK
\DiagramCellWidth\setbox6=\llap{\unhbox6\kern.1\ht2}\setbox7=\rlap{\kern.1\ht
2\unhbox7}\else\dimen0\dimen1 \advance\dimen1-\CD@ZK\DiagramCellWidth\setbox7%
=\llap{\unhbox7\kern.1\ht2}\setbox6=\rlap{\kern.1\ht2\unhbox6}\fi\setbox6=%
\vbox{\box6\kern.1\wd2}\setbox7=\vtop{\kern.1\wd2\box7}\CD@dG{\dimen0}%
\advance\dimen0-\axisheight\advance\dimen0-\CD@bK\DiagramCellHeight\dimen5-%
\dimen0 \advance\dimen0\dimen6 \advance\dimen1.5\dimen3 \ifdim\wd3>\z@\ifdim
\ht3>-\dp3\CD@TB\fi\fi\dimen3\dimen2 \dimen7\dimen2\advance\dimen7\dimen4
\ifvoid3 \else\CD@tE\else\dimen8\ht3\advance\dimen8-\axisheight\CD@iC{\dimen8%
}\CD@X{\dimen8}{.5\wd3}\dimen9\dp3\advance\dimen9\axisheight\CD@iC{\dimen9}%
\CD@X{\dimen9}{.5\wd3}\ifPositiveGradient\advance\dimen2-\dimen9\advance
\dimen4-\dimen8 \else\advance\dimen4-\dimen9\advance\dimen2-\dimen8 \fi\fi
\advance\dimen3-.5\wd3 \fi\dimen9=\CD@aK\DiagramCellHeight\advance\dimen9-2%
\DiagramCellHeight\CD@tE\advance\dimen2\dimen4 \CD@CG{2}\dimen2-\dimen0%
\advance\dimen2\dp2 \else\CD@CG{2}\CD@CG{4}\ifPositiveGradient\dimen2-\dimen0%
\advance\dimen2\dp2 \dimen4\dimen5\advance\dimen4-\ht4 \else\dimen4-\dimen0%
\advance\dimen4\dp4 \dimen2\dimen5\advance\dimen2-\ht2 \fi\fi\setbox0=\hbox to%
\z@{\kern\dimen1 \ifvoid1 \else\ifPositiveGradient\advance\dimen0-\dp1 \lower
\dimen0 \else\advance\dimen5-\ht1 \raise\dimen5 \fi\rlap{\unhbox1}\fi\raise
\dimen2\rlap{\unhbox2}\ifvoid3 \else\lower.5\dimen9\rlap{\kern\dimen3\unhbox3%
}\fi\kern.5\dimen7 \lower.5\dimen9\box6 \lower.5\dimen9\box7 \kern.5\dimen7
\CD@tE\else\raise\dimen4\llap{\unhbox4}\fi\ifvoid5 \else\ifPositiveGradient
\advance\dimen5-\ht5 \raise\dimen5 \else\advance\dimen0-\dp5 \lower\dimen0 \fi
\llap{\unhbox5}\fi\hss}\ht0=\axisheight\dp0=-\ht0\box0 }\def\NorthWest{\CD@BI
\CD@rB\DiagonalChoice0 }\def\NorthEast{\CD@CI\CD@rB\DiagonalChoice1 }\def
\SouthWest{\CD@CI\CD@qB\DiagonalChoice3 }\def\SouthEast{\CD@BI\CD@qB
\DiagonalChoice2 }\def\CD@aD{\vadjust{\CD@uA\CD@FA\advance\CD@uA
\ifPositiveGradient\else-\fi\CD@ZK\relax\CD@vA\CD@NB\advance\CD@vA-\CD@bK
\relax\hbox{\advance\CD@uA\ifPositiveGradient-\fi\CD@WK\advance\CD@vA\CD@aK
\hbox{\box6 \kern\CD@DC\kern\CD@eJ\penalty1 \box7 \box\z@}\penalty\CD@uA
\penalty\CD@vA}\penalty\CD@uA\penalty\CD@vA\penalty104}}\def\CD@eH#1{\relax
\vadjust{\hbox@maths{#1}\penalty\CD@FA\penalty\CD@NB\penalty\tw@}}\def\CD@lB{%
\ifcase\CD@GB\or\or\CD@bH{.5\wd0}{\box0}{.5\wd0}\z@\or\unhbox\z@\setbox\z@
\lastbox\CD@bH{.5\wd0}{\box0}{.5\wd0}\z@\unpenalty\unpenalty\setbox\z@
\lastbox\or\CD@TG\else\advance\CD@GB-100 \ifnum\CD@GB<\z@\cd@shouldnt B\fi
\setbox\z@\hbox{\kern\CD@mF\copy\CD@MH\kern\CD@mI\CD@uA\CD@XB\advance\CD@uA-%
\CD@NB\penalty\CD@uA\CD@uA\CD@FA\advance\CD@uA-\CD@lA\penalty\CD@uA\unhbox\z@
\global\CD@yA\lastpenalty\unpenalty\global\CD@zA\lastpenalty\unpenalty}\CD@uA
-\CD@yA\CD@vA\CD@zA\CD@fI\fi}\def\CD@TG{\unhbox\z@\setbox\z@\lastbox\CD@uA
\lastpenalty\unpenalty\advance\CD@uA\CD@mA\CD@vA\CD@XB\advance\CD@vA-%
\lastpenalty\unpenalty\dimen1\lastkern\unkern\setbox3\lastbox\dimen0\lastkern
\unkern\setbox0=\hbox to\z@{\unhbox0\setbox0\lastbox\setbox7\lastbox
\unpenalty\CD@eJ\lastkern\unkern\CD@DC\lastkern\unkern\setbox6\lastbox\dimen7%
\CD@tB\advance\dimen7-\wd\CD@uA\ifdim\dimen7<\z@\CD@CI\multiply\dimen7\m@ne
\let\mv\empty\else\CD@BI\def\mv{\raise\ht1}\kern-\dimen7 \fi\ifnum\CD@vA>%
\CD@NB\dimen6\CD@uB\advance\dimen6-\ht\CD@vA\else\dimen6\z@\fi\CD@jJ\CD@mK
\setbox1\null\ht1\dimen6\wd1\dimen7 \dimen7\dimen2 \dimen6\wd1 \CD@KJ\CD@uA
\CD@LH\CD@vA\CD@TC\dimen6\ht1 \CD@KJ\setbox2\null\divide\dimen2\tw@\advance
\dimen2\CD@eJ\CD@eG{\dimen2}\wd2\dimen2 \dimen0.5\dimen7 \advance\dimen0%
\ifPositiveGradient\else-\fi\CD@eJ\CD@dG{\dimen0}\advance\dimen0-\axisheight
\ht2\dimen0 \dimen0\CD@DC\CD@eG{\dimen0}\advance\dimen0\ht2\ht2\dimen0 \dimen
0\ifPositiveGradient-\fi\CD@DC\CD@dG{\dimen0}\advance\dimen0\wd2\wd2\dimen0
\setbox4\null\dimen0 .6\CD@zC\CD@eG{\dimen0}\ht4\dimen0 \dimen0 .2\CD@zC
\CD@dG{\dimen0}\wd4\dimen0 \dimen0\wd2 \ifvoid6\else\dimen1\ht4 \advance
\dimen1\ht2 \CD@CC6+-\raise\dimen1\rlap{\ifPositiveGradient\advance\dimen0-%
\wd6\advance\dimen0-\wd4 \else\advance\dimen0\wd4 \fi\kern\dimen0\box6}\fi
\dimen0\wd2 \ifvoid7\else\dimen1\ht4 \advance\dimen1-\ht2 \CD@CC7-+\lower
\dimen1\rlap{\ifPositiveGradient\advance\dimen0\wd4 \else\advance\dimen0-\wd7%
\advance\dimen0-\wd4 \fi\kern\dimen0\box7}\fi\mv\box0\hss}\ht0\z@\dp0\z@
\CD@bH{\z@}{\box\z@}{\z@}{\axisheight}}\def\CD@CC#1#2#3{\dimen4.5\wd#1 \ifdim
\dimen4>.25\dimen7\dimen4=.25\dimen7\fi\ifdim\dimen4>\CD@zC\dimen4.4\dimen4
\advance\dimen4.6\CD@zC\fi\CD@eG{\dimen4}\dimen5\axisheight\CD@dG{\dimen5}%
\advance\dimen4-\dimen5 \dimen5\dimen4\CD@eG{\dimen5}\advance\dimen0%
\ifPositiveGradient#2\else#3\fi\dimen5 \CD@dG{\dimen4}\advance\dimen1\dimen4 }
\def\CD@eD#1{\expandafter\CD@IK{#1}}\CD@ZA\CD@EK{output is PostScript
dependent}\def\CD@SC{\CD@IK{/bturn {gsave currentpoint currentpoint translate
4 2 roll neg exch atan rotate neg exch neg exch translate } def /eturn {%
currentpoint grestore moveto} def}}\def\CD@gK{\relax\CD@hK\CD@tK{Q}\else
\CD@IK{eturn}\fi} \def\CD@OJ#1{\count@#1\relax\multiply\count@7\advance
\count@16577\divide\count@33154 }\def\CD@fD#1{\expandafter\special{#1}} \def
\CD@xJ#1{\setbox#1=\hbox{\dimen0\dimen#1\CD@dG{\dimen0}\CD@OJ{\dimen0}\setbox
0=\null\ifPositiveGradient\count@-\count@\ht0\dimen0 \else\dp0\dimen0 \fi\box
0 \CD@uA\count@\CD@OJ\CD@LF\CD@fD{pn \the\count@}\CD@fD{pa 0 0}\CD@OJ{\dimen#%
1}\CD@fD{pa \the\count@\space\the\CD@uA}\CD@fD{fp}\kern\dimen#1}}\def\CD@JI{%
\CD@KJ\begingroup\ifdim\dimen7<\dimen6 \dimen2=\dimen6 \dimen6=\dimen7 \dimen
7=\dimen2 \count@\CD@LH\CD@LH\CD@TC\CD@TC\count@\else\dimen2=\dimen7 \fi
\ifdim\dimen6>.01\p@\CD@KI\global\CD@QA\dimen0 \else\global\CD@QA\dimen7 \fi
\endgroup\dimen2\CD@QA\CD@iK\CD@lK{\ifPositiveGradient\else-\fi\dimen6}\CD@iK
\CD@kK{\ifPositiveGradient-\fi\dimen6}\CD@iK\CD@eK{\dimen7}}\def\CD@KI{\CD@hJ
\ifdim\dimen7>1.73\dimen6 \divide\dimen2 4 \multiply\CD@TC2 \else\dimen2=0.%
353553\dimen2 \advance\CD@LH-\CD@TC\multiply\CD@TC4 \fi\dimen0=4\dimen2 \CD@ZG
4\CD@ZG{-2}\CD@ZG2\CD@ZG{-2.5}}\def\CD@AI{\begingroup\count@\dimen0 \dimen2 45%
pt \divide\count@\dimen2 \ifdim\dimen0<\z@\advance\count@\m@ne\fi\ifodd
\count@\advance\count@1\CD@@A\else\CD@y\fi\advance\dimen0-\count@\dimen2
\CD@gE\multiply\dimen0\m@ne\fi\ifnum\count@<0 \multiply\count@-7 \fi\dimen3%
\dimen1 \dimen6\dimen0 \dimen7 3754936sp \ifdim\dimen0<6\p@\def\CD@OG{4000}%
\fi\CD@KJ\dimen2\dimen3\CD@dG{\dimen2}\CD@hJ\multiply\CD@TC-6 \dimen0\dimen2
\CD@ZG1\CD@ZG{0.3}\dimen1\dimen0 \dimen2\dimen3 \dimen0\dimen3 \CD@ZG3\CD@ZG{%
1.5}\CD@ZG{0.3}\divide\count@2 \CD@gE\multiply\dimen1\m@ne\fi\ifodd\count@
\dimen2\dimen1\dimen1\dimen0\dimen0-\dimen2 \fi\divide\count@2 \ifodd\count@
\multiply\dimen0\m@ne\multiply\dimen1\m@ne\fi\global\CD@QA\dimen0\global
\CD@RA\dimen1\endgroup\dimen6\CD@QA\dimen7\CD@RA}\def\CD@OC{255}\let\CD@OG
\CD@OC\def\CD@KJ{\begingroup\ifdim\dimen7<\dimen6 \dimen9\dimen7\dimen7\dimen
6\dimen6\dimen9\CD@@A\else\CD@y\fi\dimen2\z@\dimen3\CD@XH\dimen4\CD@XH\dimen0%
\z@\dimen8=\CD@OG\CD@XH\CD@lC\global\CD@yA\dimen\CD@gE0\else3\fi\global\CD@zA
\dimen\CD@gE3\else0\fi\endgroup\CD@LH\CD@yA\CD@TC\CD@zA}\def\CD@lC{\count@
\dimen6 \divide\count@\dimen7 \advance\dimen6-\count@\dimen7 \dimen9\dimen4
\advance\dimen9\count@\dimen0 \ifdim\dimen9>\dimen8 \CD@@C\else\CD@AC\ifdim
\dimen6>\z@\dimen9\dimen6 \dimen6\dimen7 \dimen7\dimen9 \expandafter
\expandafter\expandafter\CD@lC\fi\fi}\def\CD@@C{\ifdim\dimen0=\z@\ifdim\dimen
9<2\dimen8 \dimen0\dimen8 \fi\else\advance\dimen8-\dimen4 \divide\dimen8%
\dimen0 \ifdim\count@\CD@XH<2\dimen8 \count@\dimen8 \dimen9\dimen4 \advance
\dimen9\count@\dimen0 \CD@AC\fi\fi}\def\CD@AC{\dimen4\dimen0 \dimen0\dimen9
\advance\dimen2\count@\dimen3 \dimen9\dimen2 \dimen2\dimen3 \dimen3\dimen9 }%
\def\CD@ZG#1{\CD@dG{\dimen2}\advance\dimen0 #1\dimen2 }\def\CD@dG#1{\divide#1%
\CD@TC\multiply#1\CD@LH}\def\CD@eG#1{\divide#1\CD@vA\multiply#1\CD@uA}\def
\CD@iC#1{\divide#1\CD@LH\multiply#1\CD@TC}\def\CD@hJ{\dimen6\CD@LH\CD@XH
\multiply\dimen6\CD@LH\dimen7\CD@TC\CD@XH\multiply\dimen7\CD@TC\CD@KJ}\def
\CD@iK#1#2{\begingroup\dimen@#2\relax\loop\ifdim\dimen2<.4\maxdimen\multiply
\dimen2\tw@\multiply\dimen@\tw@\repeat\divide\dimen2\@cclvi\divide\dimen@
\dimen2\relax\multiply\dimen@\@cclvi\expandafter\CD@jK\the\dimen@\endgroup
\let#1\CD@fK}{\catcode`p=12 \catcode`0=12 \catcode`.=12 \catcode`t=12 \gdef
\CD@jK#1pt{\gdef\CD@fK{#1}}}\ifx\errorcontextlines\CD@qK\CD@tJ\let\CD@GH
\relax\else\def\CD@GH{\errorcontextlines\m@ne}\fi\ifnum\inputlineno<0 \let
\CD@CD\empty\let\CD@W\empty\let\CD@mD\relax\let\CD@uI\relax\let\CD@vI\relax
\let\CD@zF\relax\else\def\CD@W{ at line \number\inputlineno}\def\CD@mD{ - first occurred%
}\def\CD@uI{\edef\CD@h{\the\inputlineno}\global\let\CD@jB\CD@h}\def\CD@h{9999%
}\def\CD@vI{\xdef\CD@jB{\the\inputlineno}}\def\CD@jB{\CD@h}\def\CD@zF{\ifnum
\CD@h<\inputlineno\edef\CD@CD{\space at lines \CD@h--\the\inputlineno}\else
\edef\CD@CD{\CD@W}\fi}\fi\let\CD@CD\empty\def\CD@YA#1#2{\CD@GH\errhelp=#2%
\expandafter\errmessage{\CD@tA: #1}}\def\CD@KB#1{\begingroup\expandafter
\message{! \CD@tA: #1\CD@CD}\ifnum\CD@XB>\CD@NB\ifnum\CD@CA>\CD@NB\else\ifnum
\CD@lA>\CD@FA\else\ifnum\CD@LB>\CD@FA\advance\CD@XB-\CD@NB\advance\CD@FA-%
\CD@lA\advance\CD@FA1\relax\expandafter\message{! (error detected at row \the
\CD@XB, column \the\CD@FA, but probably caused elsewhere)}\fi\fi\fi\fi
\endgroup}\def\CD@gB#1{{\expandafter\message{\CD@tA\space Warning: #1\CD@W}}}%
\def\CD@CB#1#2{\CD@gB{#1 \string#2 is obsolete\CD@mD}}\def\CD@AB#1{\CD@CB{%
Dimension}{#1}\CD@DE#1\CD@BB\CD@BB}\def\CD@BB{\CD@OA=}\def\CD@@B#1{\CD@CB{%
Count}{#1}\CD@DE#1\CD@OH\CD@OH}\def\CD@OH{\count@=}\def\HorizontalMapLength{%
\CD@AB\HorizontalMapLength}\def\VerticalMapHeight{\CD@AB\VerticalMapHeight}%
\def\VerticalMapDepth{\CD@AB\VerticalMapDepth}\def\VerticalMapExtraHeight{%
\CD@AB\VerticalMapExtraHeight}\def\VerticalMapExtraDepth{\CD@AB
\VerticalMapExtraDepth}\def\DiagonalLineSegments{\CD@@B\DiagonalLineSegments}%
\ifx\tenln\nullfont\CD@ZA\CD@KH{\CD@eF\space diagonal line and arrow font not
available}\else\let\CD@KH\relax\fi\def\CD@aG#1#2<#3:#4:#5#6{\begingroup\CD@PA
#3\relax\advance\CD@PA-#2\relax\ifdim.1em<\CD@PA\CD@uA#5\relax\CD@vA#6\relax
\ifnum\CD@uA<\CD@vA\count@\CD@vA\advance\count@-\CD@uA\CD@KB{#4 by \the\CD@PA
}\if#1v\let\CD@CH\CD@JK\edef\tmp{\the\CD@uA--\the\CD@vA,\the\CD@FA}\else
\advance\count@\count@\if#1l\advance\count@-\CD@A\else\if#1r\advance\count@
\CD@A\fi\fi\advance\CD@PA\CD@PA\let\CD@CH\CD@ZE\edef\tmp{\the\CD@NB,\the
\CD@uA--\the\CD@vA}\fi\divide\CD@PA\count@\ifdim\CD@CH<\CD@PA\global\CD@CH
\CD@PA\fi\fi\fi\endgroup}\CD@tG\CD@xE\CD@JD\CD@ID\CD@rG\CD@xI{See the message
above.}\CD@rG\CD@lH{Perhaps you've forgotten to end the diagram before
resuming the text, in\CD@uG which case some garbage may be added to the
diagram, but we should be ok now.\CD@uG Alternatively you've left a blank line
in the middle - TeX will now complain\CD@uG that the remaining \CD@S s are
misplaced - so please use comments for layout.}\CD@rG\CD@hD{You have already
closed too many brace pairs or environments; an \CD@HD\CD@uG command was (%
over)due.}\CD@rG\CD@hH{\CD@dC\space and \CD@HD\space commands must match.}%
\def\CD@jH{\ifnum\inputlineno=0 \else\expandafter\CD@iH\fi}\def\CD@iH{\CD@MD
\CD@GD\crcr\CD@YA{missing \CD@HD\space inserted before \CD@kH- type "h"}%
\CD@lH\enddiagram\CD@AG\CD@kH\par}\def\CD@AG#1{\edef\enddiagram{\noexpand
\CD@rD{#1\CD@W}}}\def\CD@rD#1{\CD@YA{\CD@HD\space(anticipated by #1) ignored}%
\CD@xI\let\enddiagram\CD@SG}\def\CD@SG{\CD@YA{misplaced \CD@HD\space ignored}%
\CD@hH}\def\CD@mC{\CD@YA{missing \CD@HD\space inserted.}\CD@hD\CD@AG{closing
group}}\ifx\ifx
\fi\fi

\catcode`\$=3 
\def\vboxtoz{\vbox to\z@}

\def\scriptaxis#1{\@scriptaxis{$\scriptstyle#1$}}
\def\ssaxis#1{\@ssaxis{$\scriptscriptstyle#1$}}
\def\@scriptaxis#1{\dimen0\axisheight\advance\dimen0-\ss@axisheight\raise
\dimen0\hbox{#1}}\def\@ssaxis#1{\dimen0\axisheight\advance\dimen0-%
\ss@axisheight\raise\dimen0\hbox{#1}}

\ifx\boldmath\CD@qK
\let\boldscriptaxis\scriptaxis
\def\boldscript#1{\hbox{$\scriptstyle#1$}}
\else\def\boldscriptaxis#1{\@scriptaxis{\boldmath$\scriptstyle#1$}}
\def\boldscript#1{\hbox{\boldmath$\scriptstyle#1$}}
\fi

\def\raisehook#1#2#3{\hbox{\setbox3=\hbox{#1$\scriptscriptstyle#3$}%
\dimen0\ss@axisheight
\dimen1\axisheight\advance\dimen1-\dimen0
\dimen2\ht3\advance\dimen2-\dimen0%
\advance\dimen2-0.021em\advance\dimen1 #2\dimen2%
\raise\dimen1\box3}}
\def\shifthook#1#2#3{\setbox1=\hbox{#1$\scriptscriptstyle#3$}\dimen0\wd1%
\divide\dimen0 12\CD@zH{\dimen0}
\dimen1\wd1\advance\dimen1-2\dimen0 \advance\dimen1-2\CD@oI\CD@zH{\dimen1}%
\kern#2\dimen1\box1}

\def\@cmex{\mathchar"03}

\def\make@pbk#1{\setbox\tw@\hbox to\z@{#1}\ht\tw@\z@\dp\tw@\z@\box\tw@}\def
\CD@fH#1{\overprint{\hbox to\z@{#1}}}\def\CD@qH{\kern0.11em}\def\CD@pH{\kern0%
.35em}

\def\dblvert{\def\CD@rH{\kern.5\PileSpacing}}\def\CD@rH{}

\def\SEpbk{\make@pbk{\CD@qH\CD@rH\vrule depth 2.87ex height -2.75ex width 0.%
95em \vrule height -0.66ex depth 2.87ex width 0.05em \hss}}

\def\SWpbk{\make@pbk{\hss\vrule height -0.66ex depth 2.87ex width 0.05em
\vrule depth 2.87ex height -2.75ex width 0.95em \CD@qH\CD@rH}}

\def\NEpbk{\make@pbk{\CD@qH\CD@rH\vrule depth -3.81ex height 4.00ex width 0.%
95em \vrule height 4.00ex depth -1.72ex width 0.05em \hss}}

\def\NWpbk{\make@pbk{\hss\vrule height 4.00ex depth -1.72ex width 0.05em
\vrule depth -3.81ex height 4.00ex width 0.95em \CD@qH\CD@rH}}

\def\puncture{{\setbox0\hbox{A}\vrule height.53\ht0 depth-.47\ht0 width.35\ht
0 \kern.12\ht0 \vrule height\ht0 depth-.65\ht0 width.06\ht0 \kern-.06\ht0
\vrule height.35\ht0 depth0pt width.06\ht0 \kern.12\ht0 \vrule height.53\ht0
depth-.47\ht0 width.35\ht0 }}

\def\NEclck{\overprint{\raise2.5ex\rlap{ \CD@rH$\scriptstyle\searrow$}}}
\def\NEanti{\overprint{\raise2.5ex\rlap{ \CD@rH$\scriptstyle\nwarrow$}}}
\def\NWclck{\overprint{\raise2.5ex\llap{$\scriptstyle\nearrow$ \CD@rH}}}
\def\NWanti{\overprint{\raise2.5ex\llap{$\scriptstyle\swarrow$ \CD@rH}}}
\def\SEclck{\overprint{\lower1ex\rlap{ \CD@rH$\scriptstyle\swarrow$}}}
\def\SEanti{\overprint{\lower1ex\rlap{ \CD@rH$\scriptstyle\nearrow$}}}
\def\SWclck{\overprint{\lower1ex\llap{$\scriptstyle\nwarrow$ \CD@rH}}}
\def\SWanti{\overprint{\lower1ex\llap{$\scriptstyle\searrow$ \CD@rH}}}

\def\rhvee{\mkern-10mu\greaterthan}
\def\lhvee{\lessthan\mkern-10mu}
\def\dhvee{\vboxtoz{\vss\hbox{$\vee$}\kern0pt}}
\def\uhvee{\vboxtoz{\hbox{$\wedge$}\vss}}
\newarrowhead{vee}\rhvee\lhvee\dhvee\uhvee

\def\dhlvee{\vboxtoz{\vss\hbox{$\scriptstyle\vee$}\kern0pt}}
\def\uhlvee{\vboxtoz{\hbox{$\scriptstyle\wedge$}\vss}}
\newarrowhead{littlevee}{\mkern1mu\scriptaxis\rhvee}{\scriptaxis\lhvee}%
\dhlvee\uhlvee\ifx\boldmath\CD@qK
\newarrowhead{boldlittlevee}{\mkern1mu\scriptaxis\rhvee}{\scriptaxis\lhvee}%
\dhlvee\uhlvee\else
\def\dhblvee{\vboxtoz{\vss\boldscript\vee\kern0pt}}
\def\uhblvee{\vboxtoz{\boldscript\wedge\vss}}
\newarrowhead{boldlittlevee}{\mkern1mu\boldscriptaxis\rhvee}{\boldscriptaxis
\lhvee}\dhblvee\uhblvee
\fi

\def\rhcvee{\mkern-10mu\succ}
\def\lhcvee{\prec\mkern-10mu}
\def\dhcvee{\vboxtoz{\vss\hbox{$\curlyvee$}\kern0pt}}
\def\uhcvee{\vboxtoz{\hbox{$\curlywedge$}\vss}}
\newarrowhead{curlyvee}\rhcvee\lhcvee\dhcvee\uhcvee

\def\rhvvee{\mkern-13mu\gg}
\def\lhvvee{\ll\mkern-13mu}
\def\dhvvee{\vboxtoz{\vss\hbox{$\vee$}\kern-.6ex\hbox{$\vee$}\kern0pt}}
\def\uhvvee{\vboxtoz{\hbox{$\wedge$}\kern-.6ex \hbox{$\wedge$}\vss}}
\newarrowhead{doublevee}\rhvvee\lhvvee\dhvvee\uhvvee

\def\rhtriangle{\triangleright\mkern1.2mu}
\def\lhtriangle{\triangleleft\mkern.8mu}
\def\uhtriangle{\vbox{\kern-.2ex \hbox{$\scriptscriptstyle\bigtriangleup$}%
\kern-.25ex}}
\def\dhtriangle{\vbox{\kern-.28ex \hbox{$\scriptscriptstyle\bigtriangledown$}%
\kern-.1ex}}
\def\dhblack{\vbox{\kern-.25ex\nointerlineskip\hbox{$\blacktriangledown$}}}%
\def\uhblack{\vbox{\kern-.25ex\nointerlineskip\hbox{$\blacktriangle$}}}%
\def\dhlblack{\vbox{\kern-.25ex\nointerlineskip\hbox{$\scriptstyle
\blacktriangledown$}}}
\def\uhlblack{\vbox{\kern-.25ex\nointerlineskip\hbox{$\scriptstyle
\blacktriangle$}}}
\newarrowhead{triangle}\rhtriangle\lhtriangle\dhtriangle\uhtriangle
\newarrowhead{blacktriangle}{\mkern-1mu\blacktriangleright\mkern.4mu}{%
\blacktriangleleft}\dhblack\uhblack\newarrowhead{littleblack}{\mkern-1mu%
\scriptaxis\blacktriangleright}{\scriptaxis\blacktriangleleft\mkern-2mu}%
\dhlblack\uhlblack

\def\rhla{\hbox{\setbox0=\lnchar55\dimen0=\wd0\kern-.6\dimen0\ht0\z@\raise
\axisheight\box0\kern.1\dimen0}}
\def\lhla{\hbox{\setbox0=\lnchar33\dimen0=\wd0\kern.05\dimen0\ht0\z@\raise
\axisheight\box0\kern-.5\dimen0}}
\def\dhla{\vboxtoz{\vss\rlap{\lnchar77}}}
\def\uhla{\vboxtoz{\setbox0=\lnchar66 \wd0\z@\kern-.15\ht0\box0\vss}}
\newarrowhead{LaTeX}\rhla\lhla\dhla\uhla

\def\lhlala{\lhla\kern.3em\lhla}
\def\rhlala{\rhla\kern.3em\rhla}
\def\uhlala{\hbox{\uhla\raise-.6ex\uhla}}
\def\dhlala{\hbox{\dhla\lower-.6ex\dhla}}
\newarrowhead{doubleLaTeX}\rhlala\lhlala\dhlala\uhlala

\def\hhO{\scriptaxis\bigcirc\mkern.4mu} \def\hho{{\circ}\mkern1.2mu}%
\newarrowhead{o}\hho\hho\circ\circ
\newarrowhead{O}\hhO\hhO{\scriptstyle\bigcirc}{\scriptstyle\bigcirc}

\def\rhtimes{\mkern-5mu{\times}\mkern-.8mu}\def\lhtimes{\mkern-.8mu{\times}%
\mkern-5mu}\def\uhtimes{\setbox0=\hbox{$\times$}\ht0\axisheight\dp0-\ht0%
\lower\ht0\box0 }\def\dhtimes{\setbox0=\hbox{$\times$}\ht0\axisheight\box0 }%
\newarrowhead{X}\rhtimes\lhtimes\dhtimes\uhtimes\newarrowhead+++++

\newarrowhead{Y}{\mkern-3mu\Yright}{\Yleft\mkern-3mu}\Ydown\Yup

\newarrowhead{->}\rightarrow\leftarrow\downarrow\uparrow

\newarrowhead{=>}\Rightarrow\Leftarrow{\@cmex7F}{\@cmex7E}

\newarrowhead{harpoon}\rightharpoonup\leftharpoonup\downharpoonleft
\upharpoonleft

\def\twoheaddownarrow{\rlap{$\downarrow$}\raise-.5ex\hbox{$\downarrow$}}
\def\twoheaduparrow{\rlap{$\uparrow$}\raise.5ex\hbox{$\uparrow$}}
\newarrowhead{->>}\twoheadrightarrow\twoheadleftarrow\twoheaddownarrow
\twoheaduparrow

\def\ltvee{\mkern-1mu{\lessthan}\mkern.4mu}
\newarrowtail{vee}\greaterthan\ltvee\vee\wedge

\newarrowtail{littlevee}{\scriptaxis\greaterthan}{\mkern-1mu\scriptaxis
\lessthan}{\scriptstyle\vee}{\scriptstyle\wedge}\ifx\boldmath\CD@qK
\newarrowtail{boldlittlevee}{\scriptaxis\greaterthan}{\mkern-1mu\scriptaxis
\lessthan}{\scriptstyle\vee}{\scriptstyle\wedge}\else\newarrowtail{%
boldlittlevee}{\boldscriptaxis\greaterthan}{\mkern-1mu\boldscriptaxis
\lessthan}{\boldscript\vee}{\boldscript\wedge}\fi

\newarrowtail{curlyvee}\succ{\mkern-1mu{\prec}\mkern.4mu}\curlyvee\curlywedge

\def\rttriangle{\mkern1.2mu\triangleright}
\newarrowtail{triangle}\rttriangle\lhtriangle\dhtriangle\uhtriangle
\newarrowtail{blacktriangle}\blacktriangleright{\mkern-1mu\blacktriangleleft
\mkern.4mu}\dhblack\uhblack\newarrowtail{littleblack}{\scriptaxis
\blacktriangleright\mkern-2mu}{\mkern-1mu\scriptaxis\blacktriangleleft}%
\dhlblack\uhlblack

\def\rtla{\hbox{\setbox0=\lnchar55\dimen0=\wd0\kern-.5\dimen0\ht0\z@\raise
\axisheight\box0\kern-.2\dimen0}}
\def\ltla{\hbox{\setbox0=\lnchar33\dimen0=\wd0\kern-.15\dimen0\ht0\z@\raise
\axisheight\box0\kern-.5\dimen0}}
\def\dtla{\vbox{\setbox0=\rlap{\lnchar77}\dimen0=\ht0\kern-.7\dimen0\box0%
\kern-.1\dimen0}}
\def\utla{\vbox{\setbox0=\rlap{\lnchar66}\dimen0=\ht0\kern-.1\dimen0\box0%
\kern-.6\dimen0}}
\newarrowtail{LaTeX}\rtla\ltla\dtla\utla

\def\rtvvee{\gg\mkern-3mu}
\def\ltvvee{\mkern-3mu\ll}
\def\dtvvee{\vbox{\hbox{$\vee$}\kern-.6ex \hbox{$\vee$}\vss}}
\def\utvvee{\vbox{\vss\hbox{$\wedge$}\kern-.6ex \hbox{$\wedge$}\kern\z@}}
\newarrowtail{doublevee}\rtvvee\ltvvee\dtvvee\utvvee

\def\ltlala{\ltla\kern.3em\ltla}
\def\rtlala{\rtla\kern.3em\rtla}
\def\utlala{\hbox{\utla\raise-.6ex\utla}}
\def\dtlala{\hbox{\dtla\lower-.6ex\dtla}}
\newarrowtail{doubleLaTeX}\rtlala\ltlala\dtlala\utlala

\def\utbar{\vrule height 0.093ex depth0pt width 0.4em}
\let\dtbar\utbar
\def\rtbar{\mkern1.5mu\vrule height 1.1ex depth.06ex width .04em\mkern1.5mu}%
\let\ltbar\rtbar
\newarrowtail{mapsto}\rtbar\ltbar\dtbar\utbar
\newarrowtail{|}\rtbar\ltbar\dtbar\utbar

\def\rthooka{\raisehook{}+\subset\mkern-1mu}
\def\lthooka{\mkern-1mu\raisehook{}+\supset}
\def\rthookb{\raisehook{}-\subset\mkern-2mu}
\def\lthookb{\mkern-1mu\raisehook{}-\supset}

\def\dthooka{\shifthook{}+\cap}
\def\dthookb{\shifthook{}-\cap}
\def\uthooka{\shifthook{}+\cup}
\def\uthookb{\shifthook{}-\cup}

\newarrowtail{hooka}\rthooka\lthooka\dthooka\uthooka\newarrowtail{hookb}%
\rthookb\lthookb\dthookb\uthookb

\ifx\boldmath\CD@qK\newarrowtail{boldhooka}\rthooka\lthooka\dthooka\uthooka
\newarrowtail{boldhookb}\rthookb\lthookb\dthookb\uthookb\newarrowtail{%
boldhook}\rthooka\lthooka\dthookb\uthooka\else\def\rtbhooka{\raisehook
\boldmath+\subset\mkern-1mu}
\def\ltbhooka{\mkern-1mu\raisehook\boldmath+\supset}
\def\rtbhookb{\raisehook\boldmath-\subset\mkern-2mu}
\def\ltbhookb{\mkern-1mu\raisehook\boldmath-\supset}
\def\dtbhooka{\shifthook\boldmath+\cap}
\def\dtbhookb{\shifthook\boldmath-\cap}
\def\utbhooka{\shifthook\boldmath+\cup}
\def\utbhookb{\shifthook\boldmath-\cup}
\newarrowtail{boldhooka}\rtbhooka\ltbhooka\dtbhooka\utbhooka\newarrowtail{%
boldhookb}\rtbhookb\ltbhookb\dtbhookb\utbhookb\newarrowtail{boldhook}%
\rtbhooka\ltbhooka\dtbhooka\utbhooka\fi

\def\dtsqhooka{\shifthook{}+\sqcap}
\def\ltsqhooka{\mkern-1mu\raisehook{}+\sqsupset}
\def\rtsqhooka{\raisehook{}+\sqsubset\mkern-1mu}
\def\utsqhooka{\shifthook{}+\sqcup}
\newarrowtail{sqhook}\rtsqhooka\ltsqhooka\dtsqhooka\utsqhooka

\newarrowtail{hook}\rthooka\lthookb\dthooka\uthooka\newarrowtail{C}\rthooka
\lthookb\dthooka\uthooka

\newarrowtail{o}\hho\hho\circ\circ
\newarrowtail{O}\hhO\hhO{\scriptstyle\bigcirc}{\scriptstyle\bigcirc}

\newarrowtail{X}\lhtimes\rhtimes\uhtimes\dhtimes\newarrowtail+++++

\newarrowtail{Y}\Yright\Yleft\Ydown\Yup

\newarrowtail{harpoon}\leftharpoondown\rightharpoondown\upharpoonright
\downharpoonright

\newarrowtail{<=}\Leftarrow\Rightarrow{\@cmex7E}{\@cmex7F}

\newarrowfiller{=}=={\@cmex77}{\@cmex77}
\def\vfthree{\mid\!\!\!\mid\!\!\!\mid}
\newarrowfiller{3}\equiv\equiv\vfthree\vfthree

\def\vfdashstrut{\vrule width0pt height1.3ex depth0.7ex}
\def\vfthedash{\vrule width\CD@LF height0.6ex depth 0pt}
\def\hfthedash{\CD@AJ\vrule\horizhtdp width 0.26em}
\def\hfdash{\mkern5.5mu\hfthedash\mkern5.5mu}
\def\vfdash{\vfdashstrut\vfthedash}
\newarrowfiller{dash}\hfdash\hfdash\vfdash\vfdash

\newarrowmiddle+++++

\iffalse
\newarrow{To}----{vee}
\newarrow{Arr}----{LaTeX}
\newarrow{Dotsto}....{vee}
\newarrow{Dotsarr}....{LaTeX}
\newarrow{Dashto}{}{dash}{}{dash}{vee}
\newarrow{Dasharr}{}{dash}{}{dash}{LaTeX}
\newarrow{Mapsto}{mapsto}---{vee}
\newarrow{Mapsarr}{mapsto}---{LaTeX}
\newarrow{IntoA}{hooka}---{vee}
\newarrow{IntoB}{hookb}---{vee}
\newarrow{Embed}{vee}---{vee}
\newarrow{Emarr}{LaTeX}---{LaTeX}
\newarrow{Onto}----{doublevee}
\newarrow{Dotsonarr}....{doubleLaTeX}
\newarrow{Dotsonto}....{doublevee}
\newarrow{Dotsonarr}....{doubleLaTeX}
\else
\newarrow{To}---->
\newarrow{Arr}---->
\newarrow{Dotsto}....>
\newarrow{Dotsarr}....>
\newarrow{Dashto}{}{dash}{}{dash}>
\newarrow{Dasharr}{}{dash}{}{dash}>
\newarrow{Mapsto}{mapsto}--->
\newarrow{Mapsarr}{mapsto}--->
\newarrow{IntoA}{hooka}--->
\newarrow{IntoB}{hookb}--->
\newarrow{Embed}>--->
\newarrow{Emarr}>--->
\newarrow{Onto}----{>>}
\newarrow{Dotsonarr}....{>>}
\newarrow{Dotsonto}....{>>}
\newarrow{Dotsonarr}....{>>}
\fi

\newarrow{Implies}===={=>}
\newarrow{Project}----{triangle}
\newarrow{Pto}----{harpoon}
\newarrow{Relto}{harpoon}---{harpoon}

\newarrow{Eq}=====
\newarrow{Line}-----
\newarrow{Dots}.....
\newarrow{Dashes}{}{dash}{}{dash}{}

\newarrow{SquareInto}{sqhook}--->

\newarrowhead{cmexbra}{\@cmex7B}{\@cmex7C}{\@cmex3B}{\@cmex38}
\newarrowtail{cmexbra}{\@cmex7A}{\@cmex7D}{\@cmex39}{\@cmex3A}
\newarrowmiddle{cmexbra}{\braceru\bracelu}{\bracerd\braceld}{\vcenter{%
\hbox@maths{\@cmex3D\mkern-2mu}}}
{\vcenter{\hbox@maths{\mkern2mu\@cmex3C}}}
\newarrow{@brace}{cmexbra}-{cmexbra}-{cmexbra}
\newarrow{@parenth}{cmexbra}---{cmexbra}
\def\rightBrace{\d@brace[thick,cmex]}
\def\leftBrace{\u@brace[thick,cmex]}
\def\upperBrace{\r@brace[thick,cmex]}
\def\lowerBrace{\l@brace[thick,cmex]}
\def\rightParenth{\d@parenth[thick,cmex]}
\def\leftParenth{\u@parenth[thick,cmex]}
\def\upperParenth{\r@parenth[thick,cmex]}
\def\lowerParenth{\l@parenth[thick,cmex]}

\let\hEq\rEq
\let\vEq\uEq

\def\labelstyle{
\ifincommdiag
\textstyle
\else
\scriptstyle
\fi}
\let\objectstyle\displaystyle

\newdiagramgrid{pentagon}{0.618034,0.618034,1,1,1,1,0.618034,0.618034}{1.%
17557,1.17557,1.902113,1.902113}

\newdiagramgrid{perspective}{0.75,0.75,1.1,1.1,0.9,0.9,0.95,0.95,0.75,0.75}{0%
.75,0.75,1.1,1.1,0.9,0.9}

\diagramstyle[
dpi=300,
vmiddle,nobalance,
loose,
thin,
pilespacing=10pt,%
shortfall=4pt,
]

\ifx\CD@qK\else\CD@PK\relax\CD@FF\CD@e\fi\fi

\CD@vE\CD@hK\else\fi\else\fi\cdrestoreat
\dateline{Nov 17, 2018}{Jan 14, 2020}{TBD}


\Copyright{  The authors. Released under the CC BY-ND license (International 4.0).}

\title{The classification of homogeneous finite-dimensional permutation structures}

\author{Samuel Braunfeld\\
\small Department of Mathematics\\[-0.8ex]
\small University of Maryland, College park\\[-0.8ex] 
\small College Park, U.S.A.\\
\small\tt sbraunf@umd.edu\\
\and
Pierre Simon \thanks{Supported by NSF (grant no. 1665491) and a Sloan fellowship.}\\
\small Department of Mathematics\\[-0.8ex]
\small University of California, Berkeley\\[-0.8ex]
\small Berkeley, U.S.A.\\
\small\tt simon@math.berkeley.edu}

\begin{document}

\maketitle


\begin{abstract}
We classify the homogeneous finite-dimensional permutation structures, i.e. homogeneous structures in a language of finitely many linear orders, giving a nearly complete answer to a question of Cameron, and confirming the classification conjectured by the first author. The primitive case was proven by the second author using model-theoretic methods, and those methods continue to appear here.
\end{abstract}

\section{Introduction}

A countable relational structure is \textit{homogeneous} if every finite partial automorphism extends to a total automorphism. This notion was introduced by \fraisse to generalize the behavior of the rational order, which is the unique homogeneous linear order. (For the reader unfamiliar with amalgamation and \fraisse limits, we refer to \cite{Cameron}*{\S 2}. For far more information, see \cite{Mac}.) Beginning with the case of partial orders \cite{Schmerl}, a program of classifying homogeneous structures in particular languages developed, which has included graphs \cite{LW}, tournaments \cite{tourn}, directed graphs \cite{dg}, and ongoing work on metrically homogeneous graphs \cite{MHG}.

Along this line, in \cite{Cameron} Cameron classified the homogeneous permutations, which he identified with homogeneous structures consisting of two linear orders. He then posed the problem of classifying, for each $n$, the homogeneous structures consisting of $n$ linear orders, which we call \textit{homogeneous $n$-dimensional permutation structures}. 

A construction producing many new imprimitive examples of such structures was introduced in \cite{Lattice}. The structures produced by a slight generalization of that construction, making use of the subquotient orders introduced in \S 2 in place of linear orders, were put forward as a conjecturally complete catalog in \cite{3dim}, which confirmed the case of 3 linear orders. Here, we confirm the completeness of that catalog as a whole.

\begin{theorem} \label{thm:main}
 $\Gamma$ is a homogeneous finite-dimensional permutation structure iff there is a finite distributive lattice $\Lambda$ such that $\Gamma$ is interdefinable with an expansion of the generic $\Lambda$-ultrametric space by generic subquotient orders, such that every meet-irreducible of $\Lambda$ is the bottom relation of some subquotient order.
\end{theorem}

The primitive case, in which there is no $\emptyset$-definable equivalence relation, is foundational for proving the completeness of the catalog. The Primitivity Conjecture of \cite{Lattice} conjectured that, modulo the agreement of certain orders up to reversal, a primitive homogeneous finite-dimensional permutation structure is the \fraisse limit of all finite $n$-dimensional permutation structures, for some $n$. In the case of 2 \cite{Cameron} and 3 \cite{3dim} linear orders, the conjecture was proven by increasingly involved direct amalgamation arguments. A description of the ways linear orders can interact in certain $\omega$-categorical structures, as well as of the closed sets $\emptyset$-definable in products of such structures, was given in \cite{rank1}, and as an application of these model-theoretic results, the Primitivity Conjecture was confirmed.

After reviewing the catalog and the relevant results of \cite{rank1}, our proof breaks into two sections. First, we examine the lattice of $\emptyset$-definable equivalence relations of a homogeneous finite-dimensional permutation structure, and in particular prove that each meet-irreducible element is convex with respect to some linear order in the language and that the reduct to the language of equivalence relations remains homogeneous. In the next section, we complete the classification by proving a finite-dimensional permutation structure may be presented in a language in which all the subquotient orders are generic.


Despite the fact that the catalog gives a simple description of \textit{all} finite-dimensional permutation structures, it is difficult to determine the corresponding catalog for a \textit{fixed number} of linear orders. This is because it is not known what lattices of $\emptyset$-definable equivalence relations can be realized with a given number of orders, nor is it true that one needs at most $n$ orders to represent a structure with at most $2^n$ 2-types. For some discussion and results regarding these problems, see \cite{BraunThesis}*{\S 4.4}.

\begin{problem}
Given a lattice $\Lambda$, what is the minimal $n$ such that $\Lambda$ is isomorphic to the lattice of $\emptyset$-definable equivalence relations of some homogeneous $n$-dimensional permutation structure?

Given a homogeneous finite-dimensional permutation structure $\Gamma$ presented in a language of equivalence relations and subquotient orders, what is the minimal $n$ such that $\Gamma$ is quantifier-free interdefinable with an $n$-dimensional permutation structure? 
\end{problem}

For the following proposition, see \cite{BraunThesis}*{Corollary 4.4.3} for the upper bound and Corollary \ref{cor:dim bound} for the lower bound.

\begin{proposition}
Let $\Lambda$ be a finite distributive lattice, $\Lambda_0$ the poset of
 meet-irreducibles of $\Lambda \bs \set{\bbzero, \bbone}$, $\LL$ a set of chains covering $\Lambda_0$, and $\ell$ the minimum size of any such $\LL$. Let $d_\Lambda$ be the minimum dimension of a homogeneous permutation structure with lattice of $\emptyset$-definable equivalence relations isomorphic to $\Lambda$. Then \[2\ell \leq d_\Lambda \leq |\LL| + \sum_{L \in \LL} \ceil{\log_2(|L|+1)}.\]
\end{proposition}

 However, neither bound describes the true behavior of $d_\Lambda$. To exceed the lower bound, let $\Lambda$ be a chain. To beat the upper bound, the lattice consisting of the sum of the free boolean algebra on 2 atoms and a single point may be achieved using only 4 orders.
 
 \medskip
 Although we use model-theoretic terminology throughout this paper, in the setting of a homogeneous structure $M$, many of these notions have equivalent presentations. In particular, an \textit{(n)-type} is an orbit of the action of the automorphism group $Aut(M)$ on $M^n$. Equivalently (because of the homogeneity assumption), it is an isomorphism type of $n$ labeled points. A subset $X \subset M^k$ is \textit{definable} over $A \subset M$ (or $A$-definable) if the pointwise stabilizer of $A$ in $Aut(M)$ fixes $X$ setwise. In particular, $\emptyset$-definability is equivalent to $Aut(M)$-invariance. An element $c\in M$ is definable from a set $A$ if the singleton $\{c\}$ is $A$-definable. If $E\subseteq M^2$ is a $\emptyset$-definable equivalence relation, then those notions carry through naturally to the quotient $M/E$.

\section{The catalog}
For proofs and further discussion of the results presented in this section, see \cite{BraunThesis}*{Chap. 3}.

\begin{definition}
Let $M$ be a structure, equipped with an equivalence relation $E$ and linear order $\leq$. Then we say that \textit{$E$ is $\leq$-convex}, or sometimes that \textit{$\leq$ is $E$-convex}, if every $E$-class is convex with respect to $\leq$.
\end{definition}

The construction of the structures in the catalog proceeds roughly as follows. One starts with a finite distributive lattice $\Lambda$ and constructs the generic object with a lattice of $\emptyset$-definable equivalence relations isomorphic to $\Lambda$. This structure is then expanded by linear orders so that every $\emptyset$-definable equivalence relation is convex with respect to at least one $\emptyset$-definable order and the equivalence relations are then interdefinably replaced by additional linear orders. 

However, we do not work directly with linear orders, but rather with certain partial orders which we call \textit{subquotient orders}, which allow our expansion to be generic in a natural sense.

\begin{definition}
Let $X$ be a structure, and $E \leq F$ equivalence relations on $X$. A \textit{subquotient order from $E$ to $F$} is a partial order on $X/E$ in which two $E$-classes are comparable if and only if they lie in the same $F$-class (note, this pulls back to a partial order on $X$). Thus, this partial order provides a linear order of $C/E$ for each $C \in X/F$. We call $E$ the \textit{bottom relation} and $F$ the \textit{top relation} of the subquotient order.

We say that a subquotient order $<$ from $E$ to $F$ is $G$-convex if $E$ refines $G$ and the projection to $X/E$ of any  $G$-class is $<$-convex.
\end{definition} 

Note a linear order is a subquotient order from $\bbzero$ (equality) to $\bbone$ (the trivial relation). Starting with a linear order $\leq$ convex with respect to $E$ and possibly additional equivalence relations, it can be interdefinably exchanged for its restriction within $E$-classes, a subquotient order from $\bbzero$ to $E$, and the order it induces between $E$-classes, a subquotient order from $E$ to $\bbone$. This process may then be iterated on the resulting subquotient orders until all convexity conditions are removed.

Further, instead of working directly in the language of equivalence relations, we find it convenient to work in the language of $\Lambda$-ultrametric spaces.

\begin{definition}
Let $\Lambda$ be a lattice. A \textit{$\Lambda$-ultrametric space} is a metric space where the metric takes values in $\Lambda$ and the triangle inequality uses the join rather than addition.
\end{definition} 

The following proposition shows that $\Lambda$-ultrametric spaces are equivalent to structures equipped with a lattice of equivalence relations isomorphic to $\Lambda$, or to substructures of such structures. While the lattice of equivalence relations may collapse when passing to a substructure, such as a single point, $\Lambda$-ultrametric spaces have the benefit of keeping $\Lambda$ fixed under passing to substructures.

\begin{proposition}[\cite{BraunThesis}*{Theorem 3.3.2}] \label{prop:lambdaisomorphism}
Fix a finite lattice $\Lambda$. Let $\MM_\Lambda$ be the category of $\Lambda$-ultrametric spaces, with isometries as morphisms. Let $\EE\QQ_\Lambda$ be the category of structures consisting of a set equipped with a family of not-necessarily-distinct equivalence relations $\set{E_\lambda| \lambda \in \Lambda}$ satisfying the following conditions, with embeddings as morphisms.
\begin{enumerate}
\item $\set{E_\lambda}$ forms a lattice.
\item The map $L\colon \lambda \mapsto E_\lambda$ is meet-preserving. In particular, if $\lambda_1 \leq \lambda_2$, then $E_{\lambda_1} \leq E_{\lambda_2}$.
\item $E_\bbzero$ is equality and $E_\bbone$ is the trivial relation. 
\end{enumerate}

Then $\EE\QQ_\Lambda$ is isomorphic to $\MM_\Lambda$. Furthermore, the functors of this isomorphism preserve homogeneity.
\end{proposition}

Given a system of equivalence relations as specified above, we get the corresponding $\Lambda$-ultrametric space by taking the same universe and defining $d(x, y) = \bigwedge \set{\lambda \in \Lambda | x E_\lambda y}$. In the reverse direction, given a $\Lambda$-ultrametric space, we get the corresponding structure of equivalence relations by taking the same universe and defining $E_\lambda = \set{(x,y) | d(x,y) \leq \lambda}$. As we wish to work in a finite relational language, we will usually consider our $\Lambda$-ultrametric spaces to be presented using a relation for each possible distance.

The next proposition explains the special status of distributive lattices.

\begin{proposition} [\cite{BraunThesis}*{Proposition 3.3.5, Corollary 5.2.6}] \label{prop:genericLambdaUltrametric}
Let $\Lambda$ be a finite lattice. Then the class of all finite $\Lambda$-ultrametric spaces forms an amalgamation class if and only if $\Lambda$ is distributive.
\end{proposition}

The following theorem states that we may take the generic $\Lambda$-ultrametric space and expand it by the natural analogue for subquotient orders of generic linear orders.

\begin{theorem} [\cite{BraunThesis}*{Theorem 4.2.3}] \label{theorem:amalg} 
Let $\Lambda$ be a finite distributive lattice. Let $\AA^*$ be the class of finite structures $(A,d,\set{<_{E_i}}_{i=1}^n)$ satisfying the following conditions:
\begin{itemize}
\item $(A,d)$ is  a $\Lambda$-ultrametric space;
\item $<_{E_i}$ is a subquotient order with bottom relation $E_i$, for some meet-irreducible $E_i \in \Lambda$, and top relation $F_i \in \Lambda$.
\end{itemize}
Then $\AA^*$ is an amalgamation class.
\end{theorem} 

\begin{definition} \label{def:generic}
Given a finite distributive lattice $\Lambda$, the {\em generic $\Lambda$-ultrametric space} $\Gamma$ is the \fraisse limit of all finite $\Lambda$-ultrametric spaces.

Suppose $\Gamma^*$ is $\Gamma$ equipped with some subquotient orders. We will say those subquotient orders are \textit{generic} if $\Gamma^*$ may be constructed as a \fraisse limit of a class $\AA^*$ as from Theorem \ref{theorem:amalg}.
\end{definition}

\begin{remark}
The condition that the bottom relation of a generic subquotient order be meet-irreducible is analogous to the condition that for a \fraisse class to be expandable by a generic linear order, it must have strong amalgamation. For if $E$ is meet-irreducible, then our amalgamation procedure from Proposition \ref{prop:genericLambdaUltrametric} never forces the identification of $E$-classes. However, if $E = F \meet F'$, then any $E$-class is the unique intersection of some $F$-class with some $F'$-class, so amalgamation may force the identification of $E$-classes. 
\end{remark}

Although the structures produced by Theorem \ref{theorem:amalg} are presented in the language of equivalence relations and subquotient orders, we now give a sufficient condition for them to be representable in a language of linear orders.

\begin{proposition} [\cite{BraunThesis}*{Proposition 3.4.13}]
Let $\AA^*$ be a class as from Theorem \ref{theorem:amalg}, such that every meet-irreducible of $\Lambda$ is the bottom relation of some subquotient order. Then the \fraisse limit of $\AA^*$ is interdefinable with a finite-dimensional permutation structure.
\end{proposition}

Finally, the structures in our catalog are constructed as follows. Let $\Lambda$ be a finite distributive lattice. Take the generic $\Lambda$-ultrametric space, and expand by generic subquotient orders with meet-irreducible bottom relation, such that every meet-irreducible of $\Lambda$ is the bottom relation of at least one subquotient order.

\section{Linear orders in $\omega$-categorical structures}

In this section, we review material from \cite{rank1} leading to the proof of the Primitivity Conjecture, as well as introducing definitions and results that will be used later. For proofs and further discussion of the results presented in this section, see \cite{rank1}, particularly \S 3.

\begin{notation}
Throughout this section, we will assume that $(V; \leq, \cdots)$ is a $\emptyset$-definable substructure of an $\omega$-categorical structure, equipped with a distinguished $\emptyset$-definable linear order $\leq$, and possibly other $\emptyset$-definable structure. Similarly for $(V_i; \leq_i, \cdots)$.
\end{notation}

We first define the sorts of linear orders we will be concerned with, and the ways they can interact.

\begin{definition} \label{def:minimal}
We say $(V; \leq, \cdots)$ is \textit{weakly transitive} if it is dense and the set of realizations of any 1-type $p(x)$ over $\emptyset$ concentrating on $V$ is dense in $V$.

We say $(V; \leq, \cdots)$ has \textit{topological rank 1} if it does not admit any parameter-definable $\leq$-convex equivalence relation with infinitely many infinite classes.

Finally, $(V; \leq, \cdots)$ is \textit{minimal} if it is weakly transitive and has topological rank 1.
\end{definition}

\begin{definition} \label{def:cut}
By a \textit{cut} in a dense order $(V, \leq)$, we mean an initial segment of it which is neither empty nor the whole of $V$ and has no last element. We denote by $\overline{V}$ the set of parameter-definable cuts of $V$.
\end{definition}

\begin{definition} \label{def:intertwined independent}
We say two $\emptyset$-definable weakly transitive orders $(V_0; \leq_0, \cdots)$ and $(V_1; \leq_1, \cdots)$ are \textit{intertwined} if there is a $\emptyset$-definable non-decreasing map $f\colon V_0 \to \overline{V_1}$.

If $(V_0; \leq_0, \cdots)$ and $(V_1; \leq_1, \cdots)$ are minimal, we say they are \textit{independent} if $V_0$ is intertwined with neither $V_1$ nor its reverse.
\end{definition}

The definition of independence in \cite{rank1} is on the face of it stronger, however Lemma 3.19 of that paper states that for minimal orders, independence is equivalent to the definition that we give here. The stronger property will be useful for us though, and we record it in the following lemma.

\begin{lemma} [follows from \cite{rank1}*{Lemma 3.19}] \label{lemma:ind over par}
 Let $(V_0; \leq_0, \cdots)$,$(V_1; \leq_1, \cdots)$ be minimal independent linear orders. Let $X_0, X_1$ be infinite $A$-definable subsets of $V_0, V_1$ respectively, transitive over $A$. Then $X_0, X_1$ are independent over $A$.
\end{lemma}

The following proposition is a special case of Proposition 3.23 in \cite{rank1}.

\begin{proposition}\label{prop:determined reduct}
 Let $(M; \leq_1, \ldots, \leq_n, \cdots)$ be $\omega$-categorical, transitive, with each $\leq_i$ a linear order of topological rank 1. Assume that no two distinct orders are intertwined. Then
the reduct of $M$ to the language $L_0 = \set{\leq_1 , \ldots, \leq_n}$ is completely determined up to
isomorphism by whether, $\leq_i, \leq_j$ are  equal, reverse of each other, or independent, for any $i, j \leq n$.
\end{proposition}

To apply this proposition, we need to know that a primitive finite-dimensional permutation structure has topological rank 1.

\begin{definition} \label{def:binary}
We say an $\omega$-categorical structure is \textit{binary} if it eliminates quantifiers in a finite binary relational language.
\end{definition}

\begin{definition} \label{def:top primitive}
We say $(V; \leq, \cdots)$ is \textit{topologically primitive} if it does not admit any proper $\emptyset$-definable $\leq$-convex equivalence relation besides equality.
\end{definition}


\begin{lemma} [\cite{rank1}*{Lemma 7.3}] \label{lemma:top prim top rank1}

Let $(M; \leq, \cdots)$ be a binary structure which is topologically primitive. Then $(M;\leq, \cdots)$ has topological rank 1.
\end{lemma}

The proof of Lemma \ref{lemma:top prim top rank1} uses the following result, which is a special case of \cite{rank1}*{Lemma 7.1}.

\begin{lemma}\label{lemma:finite rank}
	Let $M$ be a binary structure. Then we cannot find a sequence $(F_k)_{k<\omega}$ of uniformly parameter-definable equivalence relations and a sequence $M\supset C_0\supset C_1\supset \cdots $ such that each $C_k$ is an $F_k$-class which splits into infinitely many $F_{k+1}$-classes.
\end{lemma}

From those results, one obtains the following theorem 
confirming the Primitivity Conjecture.

\begin{theorem} [\cite{rank1}*{Theorem 7.4}] \label{thm:primitivity conj}
 Let $(\Gamma; \leq_1, \ldots, \leq_n)$ be a primitive homogeneous finite-dimensional permutation structure such that no two orders are equal or opposite of each other. Then $\Gamma$ is the \fraisse limit of all finite sets equipped with $n$ orders.
\end{theorem}

Finally, we close with several results that will also be used later. The first two propositions describe the closed sets $\emptyset$-definable in a minimal linear order and then in a product of pairwise independent linear orders. 

\begin{proposition} [\cite{rank1}*{Proposition 3.11}] \label{prop:dense definable}
Let $(V; \leq, \cdots)$ be a minimal definable linear order. Let $p(x_0, \ldots , x_{n-1})$ be a type in $V^n$ such that $p \proves x_0 < x_1 < \ldots < x_{n-1}$. Then given open $\leq$-intervals $I_0 < \cdots < I_{n-1}$ of $V$ , we can find $a_i \in I_i$ such that $(a_0 , \ldots , a_{n-1}) \satisfies p$.
\end{proposition}

\begin{lemma} [\cite{rank1}*{Lemma 3.1}] \label{lemma:acl}
Let $(V, \leq, \dots)$ be infinite and transitive. Then $\leq$ is dense, and for any $a \in V$, $\acl(a) = \set{a}$.
\end{lemma}

\begin{proposition} [\cite{rank1}*{Proposition 3.21}] \label{prop:product definable}
Choose pairwise independent minimal orders. $(V_0; \leq_0, \cdots), \dots, (V_{n-1}; \leq_{n-1}, \cdots)$. Then any $\emptyset$-definable closed set $X \subseteq V_0^{k_0} \times \cdots \times V_{n-1}^{k_{n-1}}$ is a finite union of products of the form $D_0 \times \cdots \times D_{n-1}$, where each $D_i$ is a $\emptyset$-definable closed subset of $V_i^{k_i}$.
\end{proposition}

\begin{proposition} [\cite{rank1}*{Proposition 6.1}] \label{prop:trivial acl}
Assume that $M$ is NIP and binary. Let $X,Y \subset M$ be $\emptyset$-definable, and let $p(x, y)$ be the complete type of a thorn-independent pair from $X \times Y$. Let $(V; \leq, \cdots)$ be a $\emptyset$-definable minimal linear order. and let $f\colon p(X \times Y) \to V$ be a $\emptyset$-definable function. Then for any $(a, b ) \satisfies p$, $f(a, b) \in \dcl(a) \cup \dcl(b)$.

In particular, for any $A \subset V$, $\dcl(A) \cap V = A$.
\end{proposition}

\begin{remark}\label{rem:trivial acl}
A definition of thorn-independence in our setting may be found in \cite{rank1}*{\S 2.3}. However, we will only need the following three facts:
\begin{enumerate}
	\item We can always find a thorn-independent pair in a product of $\emptyset$-definable sets.
	\item If $(a,b)$ is a thorn-independent pair and $E$ is a $\emptyset$-definable equivalence relation with infinitely many classes, then $a$ and $b$ are in different $E$-classes.
	\item If $(a,b)$ is a thorn-independent pair, $a'\in \dcl(a)$ and $b'\in \dcl(b)$, then $(a',b')$ is a thorn-independent pair.
\end{enumerate}

\end{remark}

For applications of Proposition \ref{prop:trivial acl}, note that a homogeneous finite-dimensional permutation structure is NIP, as it has quantifier elimination, NIP is preserved by boolean combinations, and ``$x \leq y$'' is NIP.

\section{The lattice of $\emptyset$-definable equivalence relations}

In this section, we investigate the $\emptyset$-definable equivalence relations of a homogeneous finite-dimensional permutation structure.
The main result for the first subsection, Lemma \ref{lemma:intersectionconvex}, is that each meet-irreducible element of the lattice is convex with respect to some linear order in the language, and that of the next subsection, Proposition \ref{prop:homreduct}, is that the reduct to the language of $\emptyset$-definable equivalence relations is generic.

Some lemmas will be proven in a more general setting, so we introduce the following definition.

\begin{definition}
A structure $M$ is \textit{order-like} if for any complete type $p(x,y)$ in two variables over $\emptyset$, we have $p(x,y)\wedge p(y,z)\to p(x,z)$.
\end{definition}

Note that a homogeneous finite-dimensional permutation structure is order-like. Conversely, we do not know if every transitive, order-like, binary NIP structure is (bi-definable with) a finite-dimensional permutation structure.

\begin{notation}
For this section, $\leq$ will denote a linear order, as will $\leq_i$.
\end{notation}

\subsection{Convexity}
We first establish an analogue of Lemma \ref{lemma:top prim top rank1} when working in the quotient of a binary structure.

\begin{lemma} \label{lemma:toprank1}
Let $(M; \leq, \cdots)$ be $\omega$-categorical and binary. Let $E$ be the coarsest proper $\emptyset$-definable $\leq$-convex equivalence relation. Then $(M/E; \leq, \cdots)$ has topological rank 1.
\end{lemma}
\begin{proof}
The proof is an adaptation of that of \cite{rank1}*{Lemma 7.3}. Assume that there is a proper parameter-definable $\leq$-convex equivalence relation $F$ on $(M/E;\leq, \cdots)$. Let $F$ be defined over $\bar a$ and write $F=F_{\bar a}$. Let $R(x,y)$ be the relation on $M/E$ which holds of a pair $(c,d)$ if for every $\bar b$ having the same type as $\bar a$ over $\emptyset$, there are finitely many $F_{\bar b}$-equivalence classes between $c$ and $d$. Then $R$ is $\emptyset$-definable and is a $\leq$-convex equivalence relation. By the maximality assumption on $E$, $R$ is equality. Then for any $F_{\bar a}$-class $C$, there is $\bar b\equiv \bar a$ such that $C$ splits into infinitely many $F_{\bar b}$-classes. Let $F_0=F$, $C_0=C$, $F_1=F_{\bar b}$, and let $C_1$ be any $F_1$-class inside $C_0$. We can iterate the construction to obtain a sequence $(F_k)_{k<\omega}$ of equivalence relations on $M/E$ and a decreasing sequence $(C_k)_{k<\omega}$ such that each $C_k$ is an $F_k$-class that splits into infinitely many $F_{k+1}$-classes. This entire situation lifts to $M$ and contradicts Lemma \ref{lemma:finite rank}.
\end{proof}

\begin{lemma} \label{lemma:nocut}
	Let $(M;\leq,\cdots)$ be $\omega$-categorical order-like, transitive, and binary. Let $E$ be the coarsest proper convex $\emptyset$-definable equivalence relation. Then given $a\in M$, for any $a$-definable cut $c$ of $(M,\leq)$, we have $\inf(a/E)\leq c \leq \sup(a/E)$.
\end{lemma}

\begin{proof}
	We write $a\ll b$ for $a/E < b/E$, equivalently $a<b$ and $a/E \neq b/E$. If $c$ is a cut, then $a\ll c$ means that $c$ is greater than the supremum of the $E$-class of $a$.
	
	Assume that there is a cut $c$ definable from $a$, with $a\ll c$. Let $c(a)$ be the minimal such cut. Consider the cut $c_*:=\sup\{c(a'):a'Ea\}$. Note that $c_*$ depends only on the $E$-class of $a$, so we can write $c_*=f_*(a/E)$ for some function $f_*$. Assume $c_*$ is not $+\infty$ and let $g_*(a/E)$ be the image of $f_*(a/E)$ in the quotient $M/E$. Then $g_*$ is a function from $(M/E,\leq)$ to its Dedekind completion with $x<g_*(x)$. By Lemma \ref{lemma:toprank1}, $M/E$ has topological rank 1. Therefore by Proposition \ref{prop:dense definable}, the graph of $g_*$ must be dense in $\{(x,y):x<y\}$. We can then find $b\in M$ such that $a\ll b\ll c(a) \ll c(b)$. Then as $c(b)$ is the minimal cut definable from $b$ above $b/E$, there is $d\in M$ such that $\tp(a,b)=\tp(b,d)$ and $a\ll b\ll c(a) \ll d \ll c(b)$. So $\tp(a,d)\neq \tp(a,b)$ and this is a contradiction to $M$ being order-like.
	
	If $c_*$ is $+\infty$, then we can also find $b$ as above, just by definition of $c_*$.
\end{proof}

\begin{corollary} \label{cor:dense class}
Let $(M;\leq,\cdots)$ be $\omega$-categorical order-like, transitive, and binary. Let $E$ be the coarsest proper $\leq$-convex $\emptyset$-definable equivalence relation. Let $F$ be a $\emptyset$-definable equivalence relation not refining $E$. Then no $F$-class defines a cut in $(M;\leq,\cdots)$.

In particular:
\begin{enumerate}
\item Every $F$-class intersects a dense set of $E$-classes.
\item Suppose $F$ is the coarsest proper $\leq'$-convex $\emptyset$-definable equivalence relation, for some $\emptyset$-definable order $\leq'$. Then $(M/E;\leq,\cdots)$ and $(M/F;\leq',\cdots)$ are independent.
\end{enumerate}
\end{corollary}
\begin{proof}
Let $C$ be an $F$-class, and suppose that a cut in $(M;\leq,\cdots)$ is definable from $C$. Then that cut is definable from any $a \in C$. Let $a_1, a_2 \in C$ lie in distinct $E$-classes. By Lemma \ref{lemma:nocut}, the only cut of $(M/E, \leq)$ definable from $a_i$ is that of $a_i/E$. As these cuts are distinct, $C$ can define neither.

For $(1)$, let $C$ be an $F$-class, and let $\sim$ be the $\leq$-convex equivalence relation on $M/E$ defined by: \[x \sim y \iff \text{the interval }[x, y]\text{ lies in the complement of C}.\] As $M/E$ has topological rank 1 by Lemma \ref{lemma:toprank1}, $\sim$ has finitely many classes. There must be multiple $\sim$-classes, but then their endpoints would be $C$-definable cuts.

For $(2)$, note that an intertwining map would require each $F$-class to define a cut in $(M/E;\leq,\cdots)$, which we ruled out above.  
\end{proof}

\begin{definition}
Let $(\Gamma, \leq_1, \ldots, \leq_n)$ be homogeneous, and $E$ a $\emptyset$-definable equivalence relation. We say $E$ is \textit{convex} if it is $\leq_i$-convex for some $i$.
\end{definition}


\begin{lemma} \label{lemma:refineconvex}
Let $(\Gamma, \leq_1, \ldots, \leq_n)$ be homogeneous. Then any maximal $\emptyset$-definable equivalence relation $F$ is convex with respect to at least two linear orders in the language.
\end{lemma}
\begin{proof}
 For each $i \leq n$, let $E_i$ denote the maximal $\leq_i$-convex $\emptyset$-definable equivalence relation, let $V_i = (\Gamma/E_i; \leq_i, \cdots)$ be the structure induced on the quotient, and let $W_1,\ldots, W_k$ be representatives of the $V_i$'s up to $\emptyset$-definable monotonic bijection. Then by Corollary \ref{cor:dense class}(2) and Lemma \ref{lemma:toprank1}, the $W_i$'s are pairwise independent topological rank 1 ordered sets. 
 
 First, suppose $F$ is not convex. Then by Corollary \ref{cor:dense class}(1), each $F$-class projects densely on each $W_i$. For any $F$-class $C$, we may expand the language by a unary predicate naming $C$. Each resulting $W_i$ is still minimal, as $C$ can define no cuts by Corollary \ref{cor:dense class}. Let $C^* \subset \prod W_i$, where each element of $C^*$ is the tuple of projections onto each $W_i$ of an element of $C$. By Proposition \ref{prop:product definable},  $C^*$ is dense in $\prod W_i$ equipped with the product topology, i.e. if a non-empty open $\leq_i$-interval is chosen for each $W_i$, then there is some $c \in C^*$ lying in all the chosen intervals.  
 By quantifier elimination $\bigwedge_i x<_i y$ implies a complete type on $(x,y)$. However, by the density of $C^*$ for any $F$-class $C$, this type is consistent both with $F(x,y)$ and $\neg F(x,y)$.
 
 Now suppose $F$ is only $\leq_j$-convex. Then we carry out the same argument, omitting $W_j$. Again, each $F$-class is dense in the product $\prod_{i \neq j} W_i$.  Let $C_1 <_j C_2$ be two $F$-classes. By density, we can find $a,a'\in C_1$ and $b,b'\in C_2$ such that $\bigwedge_{i\neq j} a <_i b$ and $\bigwedge_{i\neq j} b'<_i a'$. It follows, both $x<_j y \wedge  \bigwedge_{i \neq j} x<_i y$ and $y<_j x \wedge \bigwedge_{i\neq j} x<_i y$ are consistent with $\neg F(x,y)$. By quantifier elimination, those formulas must imply $\neg F(x,y)$. However, by density of $C_1$, we can find $(c,d) \in C_1^{2}$ satisfying one of those formulas. This is a contradiction.
\end{proof}



\begin{lemma} \label{lemma:intersectionconvex}
Let $(\Gamma; \leq_1, \ldots, \leq_n)$ be homogeneous, with lattice of $\emptyset$-definable equivalence relations $\Lambda$. Then any meet-irreducible $E\in \Lambda$ is convex with respect to at least two linear orders in the language.
\end{lemma}
\begin{proof}
Let $E \in \Lambda$ be meet-irreducible, and $E^+$ the cover of $E$. Fix an $E^+$-class $C^+$, and $C \subset C^+$ an $E$-class. By Lemma \ref{lemma:refineconvex}, there are distinct $i,j \leq n$ such that $C$ is $\leq_i$-convex and $\leq_j$-convex in $C^+$.

If $E^+$ is both $\leq_i$-convex and $\leq_j$-convex, we are finished, so assume neither $E^+$ nor $E$ is $\leq_i$-convex. Let $\overline{C^+}$ be the $\leq_i$-convex closure of $C^+$. The structure $(\overline{C^+};\leq_1,\ldots,\leq_n)$ is homogeneous. Let $G$ be the maximal $\emptyset$-definable equivalence relation that is $\leq_i$-convex in $\overline {C^+}$, so $G$ is also $\leq_i$-convex in $\Gamma$. Then $E^+$ does not refine $G$, since otherwise $C^+/G$ would equal $\overline{C^+}$; thus $E$ also does not refine $G$, since we cannot have $E=G$ as $E$ is not $\leq_i$-convex. By Corollary \ref{cor:dense class}, the projections of both $C$ and $C^+$ are dense in $(\overline{C^+}/G; \leq_i)$. As $C$ is $\leq_i$-convex in $C^+$, these projections must be equal. As $(\overline{C^+}/G; \leq_i)$ is without endpoints, applying $\leq_i$-convexity again gives $C = C^+$, which is a contradiction.
\end{proof}

\begin{corollary}
	Let $(\Gamma; \leq_1, \ldots, \leq_n)$ be homogeneous, then any $\emptyset$-definable equivalence relation is an intersection of convex $\emptyset$-definable equivalence relations.
\end{corollary}

\begin{remark}
By the proof of \cite{BraunThesis}*{Lemma 3.4.10}, any intersection of convex equivalence relations is convex for some $\emptyset$-definable order (not necessarily one of $\leq_1,\ldots,\leq_n$).
\end{remark}

\begin{corollary} \label{cor:dim bound}
Let $\Lambda$ be a finite distributive lattice, $\Lambda_0$ the poset of meet-irreducibles of $\Lambda \bs \set{\bbzero, \bbone}$, and $\ell$ the minimum number of chains needed to cover $\Lambda_0$. Let $d_\Lambda$ be the minimum dimension of a homogeneous permutation structure with lattice of $\emptyset$-definable equivalence relations isomorphic to $\Lambda$. Then $2\ell \leq d_\Lambda$.
\end{corollary}
\begin{proof}
By Lemma \ref{lemma:intersectionconvex}, every element of $\Lambda_0$ must be convex for at least two linear orders in the language. However, if $E, F \in \Lambda$ are incomparable, then they cannot be convex with respect to the same order.
\end{proof}

\subsection{Distributivity and genericity}
We now prove distributivity of the lattice of $\emptyset$-definable equivalence relations and genericity of the reduct to the language of $\emptyset$-definable equivalence relations. Distributivity is proven first, although by Proposition \ref{prop:genericLambdaUltrametric} it follows from genericity, as we will use Proposition \ref{prop:genericLambdaUltrametric} to prove genericity.

\begin{definition}
Two equivalence relations $E$ and $F$ are \textit{cross-cutting} if every $E$-class intersects every $F$-class.
\end{definition}

We now prove that two $\emptyset$-definable equivalence relations are cross-cutting if their join is $\bbone$. Note this would be immediate if we already knew the genericity of the reduct to the language of $\emptyset$-definable equivalence relations.

\begin{lemma}\label{lemma:two relations}
	Let $(M;E,F,\cdots)$ be $\omega$-categorical, transitive, and order-like, where $E$ and $F$ are $\emptyset$-definable equivalence relations. Let $a,a',b\in M$ such that $aEa'$ and $aFb$, then there is $b'\in M$ with $bEb'$ and $a'Fb'$.
\end{lemma}
\begin{proof}
	Let $e_1$ be the $E$-class of $a$ and $e_2$ the $E$-class of $b$. Take $e_3$ so that $\tp(b,e_3)=\tp(a',e_2)$. Next take $c\in e_2$ such that $\tp(c,e_3)=\tp(a,e_2)$ (this is possible as $\tp(e_1,e_2)=\tp(e_2,e_3)$). Finally let $d\in e_3$ be such that $\tp(a,c)=\tp(c,d)$, so, since $M$ is order-like, $\tp(a,d)=\tp(a,c)$. Let $d'$ be such that $\tp(a,d,d')=\tp(a,c,b)$. Then we have $aFb$ and $dEd'$. So $d'\in e_3$ and $bFd'$. Since $\tp(b,e_3)=\tp(a',e_2)$, there is $b'\in e_2$ such that $a'Fb'$. 
\end{proof}

\begin{corollary} \label{cor:twocross}
Let $(M;E,F,\cdots)$ be $\omega$-categorical, transitive, and order-like, with lattice of $\emptyset$-definable equivalence relations $\Lambda$. Let $E, F \in \Lambda$ such that $E \join F = \bbone$. Then $E$ and $F$ are cross-cutting.
\end{corollary}
\begin{proof}
We first show that Lemma \ref{lemma:two relations} implies that given two $E$-classes $e_1$ and $e_2$, either $e_1$ and $e_2$ intersect the same $F$-classes, or they intersect disjoint sets of $F$-classes. Let $f$ be an $F$-class such that $e_1$ and $e_2$ intersect $f$. Let $a \in e_1 \cap f$ and $b \in e_2 \cap f$. Given any $a' \in e_1$, let $f' = a'/F$. Then by Lemma \ref{lemma:two relations}, we may find some $b' \in e_2 \cap f'$.

Now, let $G$ be the equivalence relation on $M$ which holds for $(a,b)$ if the $E$-class of $a$ and the $E$-class of $b$ intersect the same $F$-classes. Then $G$ is definable and is coarser than both $E$ and $F$, and in fact $G = E \join F$. As $E \join F = \bbone$, $E$ and $F$ must be cross-cutting.
\end{proof}


\begin{lemma}\label{lemma:nice lattice}
Let $(\Gamma;\leq_1,\ldots,\leq_n)$ be homogeneous. Let $E,F,G$ be $\emptyset$-definable equivalence relations such that $E$ is maximal, $F \join G = \bbone$, and neither $F$ nor $G$ refines $E$. Then $F\wedge G$ does not refine $E$.
\end{lemma}
\begin{proof}
Assume that $F\wedge G$ refines $E$. Pick any $F$-class $C_F$, let $C_G$ be a thorn-independent $G$-class, and let $p(x,y)$ be the type of $(C_F, C_G)$. By Corollary \ref{cor:twocross}, we may find $a$ such that $a/F = C_F$ and $a/G = C_G$. We then have a function $f\colon p(\Gamma/F,\Gamma/G) \to \Gamma/E$, given by $f(x/F,x/G)=x/E$. This is well-defined as $F\wedge G \leq E$. By Proposition \ref{prop:trivial acl}, $a/E \in \dcl(a/F) \cup \dcl(a/G)$. Assume say $a/E \in \dcl(a/F)$, then the $F$-class of $a$ is included in an $E$-class (by transitivity of $\Gamma$), so $F$ refines $E$.
\end{proof}

\begin{proposition} \label{prop:distributivity}
Let $(\Gamma;\leq_1,\ldots,\leq_n)$ be homogeneous, with lattice of $\emptyset$-definable equivalence relations $\Lambda$. Then $\Lambda$ is distributive.
\end{proposition}
\begin{proof}
We must prove the lattices $M_3$ and $N_5$ (see Figure \ref{fig:lat}) do not appear as sublattices of $\Lambda$.

\begin{figure}[h]
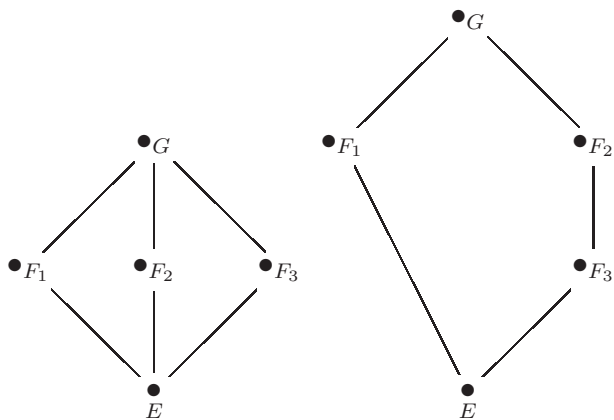

\begin{diagram}[size=2em]
& & & & & & & \bullet_G \\
& & & & & & \ruLine & & \luLine \\
& & \bullet_G & & &  \bullet_{F_1} & & & & \bullet_{F_2} \\
& \ruLine & \uLine & \luLine & & & \rdLine(2,4)  & &  & \uLine \\
 \bullet_{F_1} & & \bullet_{F_2} & & \bullet_{F_3}  & {} & &  & & \bullet_{F_3} &   \\
& \luLine & \uLine  & \ruLine  & & &   &  & \ldLine  \\
& & \underset{E}{\bullet} & & & & &  \underset{E}{\bullet}  \\
\end{diagram}
\caption{$M_3$ and $N_5$}
\label{fig:lat}
\end{figure}

Suppose $M_3$ appears as a sublattice. Let $E$ be the minimum element, $G$ the maximum element, and $F_1, F_2, F_3$ the non-trivial elements. We pick a $G$-class and work within it, so we may assume $G = \bbone$. Let $F' \geq F_1$ be maximal below $G$. Then neither $F_2$ nor $F_3$ refines $F'$, so by Lemma \ref{lemma:nice lattice} neither does $F_2 \meet F_3$. But this is a contradiction.

Now suppose $N_5$ appears as a sublattice. Let $E$ be the minimum element, $G$ the maximum element, and $F_1, F_2, F_3$ the non-trivial elements, with $F_2 > F_3$. We pick a $G$-class and work within it, so we may assume $G = \bbone$. Then $F_1$ and $F_3$ are cross-cutting by Corollary \ref{cor:twocross}, as $F_1 \join F_3 = \bbone$. As there are infinitely many $F_3$-classes in each $F_2$-class (see \cite{BraunThesis}*{Lemma 5.2.2}), we cannot have that $F_1 \meet F_3 = F_1 \meet F_2$.
\end{proof}

\begin{proposition} \label{prop:homreduct}
Let $(\Gamma, \leq_1, \ldots, \leq_n)$ be homogeneous. Then the reduct of $\Gamma$ to the language of $\emptyset$-definable equivalence relations is generic.
\end{proposition}
\begin{proof}
Let $\Lambda$ be the lattice of $\emptyset$-definable equivalence relations of $\Gamma$. Let $A \subset \Gamma$ be finite. We must realize any maximal quantifier-free 1-type $p(x/A)$, in the language of equivalence relations, that is consistent with the $\Lambda$-triangle inequality. (For the equivalence between being a \fraisse limit and satisfying such a 1-point extension property, see \cite{Hodges}*{\S 7.1}.) We proceed by induction on the height of $\Lambda$. The statement is trivial if $\Lambda$ has 2 elements.

Let $a \in A$ such that $p \proves xFa$ for some maximal $F \in \Lambda$, and let $C = a/F$. We now wish to inductively continue inside $C$, but might not have $A \subset C$. For every $E_i \in \Lambda$ such that $F$ and $E_i$ are incomparable and every $a_i \in A$ such that $p \proves d(x, a_i) = E_i$, we will find some $a'_i \in \Gamma$ such that $a'_i \in C \cap a_i/E_i$. Let $A'$ be $A$ with each $a_i$ replaced by $a'_i$. We then create a new type $p'(x/A')$ by choosing distances $d_i = d(a'_i, x) \leq E_i$ such that $p'$ is still consistent with the $\Lambda$-triangle inequality. By induction, there will be some element realizing $p'$, which will then also realize $p$.

For each $E_i$ incomparable to $F$, we have that $E_i$ and $F$ are cross-cutting by Corollary \ref{cor:twocross}. Thus for each $a_i$ such that $p \proves d(x, a_i) = E_i$, we have that $C \cap a_i/E_i \neq \emptyset$. We may pick any element of this intersection to be $a'_i$. 

We view the assignment of the distances $d_i$ as an amalgamation problem. The base is $A$, the first factor is $A \cup \set{x}$, and the second factor is $A \cup A'$. As $\Lambda$ is distributive by Proposition \ref{prop:distributivity}, we can apply Proposition \ref{prop:genericLambdaUltrametric} to complete the amalgamation diagram while respecting the $\Lambda$-triangle inequality. This forces $d_i \leq d(a'_i, a_i) \join d(a_i, x) = E_i$.
\end{proof}

\section{Classification}

All homogeneous finite-dimensional permutation structures are assumed to be presented in a language of equivalence relations and subquotient orders.

It is not immediate that the quotient of a homogeneous finite-dimensional permutation structure by a $\emptyset$-definable equivalence relation is again homogeneous. However, the next few results establish the Primitivity Conjecture (or actually something slightly stronger) when working in the quotient.

\begin{lemma} \label{lemma:genericquotient}
Let $\Gamma$ be a homogeneous permutation structure with lattice of $\emptyset$-definable equivalence relations $\Lambda$. Let $E,F \in \Lambda$ with $E < F$. Let $C$ be an $F$-class, and let $\leq_1, \dots, \leq_n$ be $\emptyset$-definable topologically primitive linear orders on $C/E$. Then $\leq_1, \dots, \leq_n$ are generic, modulo the agreement of certain orders up to reversal.
\end{lemma}
\begin{proof}
By Lemma \ref{lemma:toprank1}, $C/E$ has topological rank 1 with respect to each $\leq_i$. By Proposition \ref{prop:determined reduct}, it suffices to show that no order is intertwined with another, or its reverse.

Suppose $\leq_i$ is intertwined with $\leq_j$ via some intertwining map $f$. Then, given any $a \in C/E$, $f$ produces an $a$-definable cut in $\leq_j$. If $\leq_i, \leq_j$ are not equal, there must be some $a \in C/E$ for which $f(a) \neq a$. But this contradicts Lemma \ref{lemma:nocut}.
\end{proof}

\begin{definition}
Let $\Gamma$ be a homogeneous finite-dimensional permutation structure, $<$ a subquotient order on $\Gamma$, and $E$ a $\emptyset$-definable equivalence relation. Then the \textit{restriction of $<$ to $E$}, when defined, is the subquotient order $< \upharpoonright_E$ with top relation $E$ given by $x < \upharpoonright_E y$ iff $(xEy) \wedge (x < y)$.
\end{definition}

\begin{lemma} \label{lemma:equalrestrictions}
Let $\Gamma$ be a homogeneous finite-dimensional permutation structure, with lattice of $\emptyset$-definable equivalence relations $\Lambda$. Let $<_i, <_j$ be subquotient orders with bottom relation $E \in \Lambda$ and top relations $F_i, F_j \in \Lambda$, respectively, with $F_i \leq F_j$. Assume $<_i, <_j$ are convex with respect to no $\emptyset$-definable equivalence relations between $E$ and $F_i$.

 Let $G \in \Lambda$ with $E < G < F_i$. If ${<_i \upharpoonright_{G}} = {<_j \upharpoonright_{G}}$, then ${<_i} = {<_j \upharpoonright_{F_i}}$.
\end{lemma}
\begin{proof}
Suppose not. Let $C$ be an $F_i$-class. By Lemma \ref{lemma:genericquotient}, $<_i$ and $<_j$ are independent on $C/E$. Let $a \in C/E$ and $A = a/G$. Let $b \in C/E \cap (A \bs \set{a})$, $p = \tp(b/a)$, and $P \subset C/E$ the realizations of $p$. Then $P$ is infinite, as $\acl(a) = \set{a}$ by Lemma \ref{lemma:acl}, and so $<_i$ and $<_j$ are independent on $P$ by Lemma \ref{lemma:ind over par}. However, this is a contradiction as $P$ lies within a single $G$-class, so ${<_i} = {<_j}$ on $P$. 
\end{proof}

\begin{corollary} \label{corollary:homquotient}
Let $\Gamma$ be a homogeneous finite-dimensional permutation structure, with lattice of $\emptyset$-definable equivalence relations $\Lambda$. Let $E$ be meet-irreducible in $\Lambda$ and $E^+$ its unique cover. Let $C$ be an $E^+$-class, and consider $C/E$ equipped with the restriction to $E^+$ of every subquotient order with bottom relation $E$. If none of the original subquotient orders are equal up to reversal to any restriction of another, then $C/E$ is generic. 
\end{corollary}
\begin{proof}
By Lemma \ref{lemma:genericquotient}, $C/E$ is generic modulo the agreement of certain orders up to reversal. By Lemma \ref{lemma:equalrestrictions}, as none of the original subquotient orders were equal up to reversal, none of the restricted subquotient orders are either. 
\end{proof}

The structures in the catalog have subquotient orders with only meet-irreducible bottom relations. The next few results show we may ignore the possibility of subquotient orders with meet-reducible bottom relations.




\begin{lemma} \label{lemma:0to1}
Let $\Gamma$ be a homogeneous finite-dimensional permutation structure with lattice of $\emptyset$-definable equivalence relations $\Lambda$. Let $<$ be a subquotient order from $E$ to $F$, convex with respect to no intermediate $\emptyset$-definable equivalence relation. Let $E = G_1 \meet G_2$. Then $F \not \geq G_1 \join G_2$.
\end{lemma}
\begin{proof}
Suppose $F \geq G_1 \join G_2$. Define $\Gamma'$ to be $\Gamma/E$ restricted to a single $G_1 \join G_2$-class. Then $<$ is a transitive linear order on $\Gamma'$, and by Lemma \ref{lemma:toprank1} has topological rank 1, and thus is minimal. Choose a thorn-independent pair $a_1, a_2 \in \Gamma'$. By Remark \ref{rem:trivial acl}(2), $d(a_1, a_2) = G_1 \join G_2$. Also, by Remark \ref{rem:trivial acl}(3), $(a_1/G_1,a_2/G_2)$ is a thorn-independent pair. But $a_1/G_1 \cap a_2/G_2 \in \Gamma'$ is neither in $\dcl(a_1/G_1)$ nor in $\dcl(a_2/G_2)$, contradicting Proposition \ref{prop:trivial acl}.
\end{proof}

\begin{lemma} \label{lemma:nameclass}
Let $\Gamma$ be a homogeneous finite-dimensional permutation structure, $E$ a $\emptyset$-definable equivalence relation, and $C_1, C_2 \subset \Gamma$ two $E$-classes. Then $C_1$ remains homogeneous after expanding the language by a unary predicate naming $C_2$. 
\end{lemma}
\begin{proof}
Let $A \subset C_1$ be finite. It suffices to find some $c \in C_2$ such that $c$ has the same type over every $a \in A$, as we may then always extend any finite partial isomorphism of $C_1$ to one fixing $c$.

Let $F \geq E$ be maximal such that $C_1, C_2$ lie in distinct $F$-classes. Note $F$ must be meet-irreducible, so let $F^+$ be its cover. Let $C$ be the $F^+$-class of $C_1$, which is also the $F^+$-class of $C_2$ and we now work in $C$. So we may assume $F^+ = \bbone$. Let $C'_2$ be the $F$-class of $C_2$. 

We now move to the language of linear orders. If $\leq_i$ is an order for which there is $\leq_i$-convex $G$ with $E \leq G \leq F$, then any $x \in C'_2$ has the same $\leq_i$-type over every $a \in A$. Enumerate the remaining orders---those for which there is no such $G$---as $\leq_1, \ldots, \leq_m$. For each $i\leq m$, let $E_i \in \Lambda$ be maximal $\leq_i$-convex below $F^+$ and let $V_i = (C/E_i; \leq_i)$. By Corollary \ref{cor:dense class}, $C'_2$ projects densely onto each $V_i$.

We now work inside $C'_2$. Let $F_i \in \Lambda$ be the maximal $\leq_i$-convex relation below $F$. Then $E$ does not refine $F_i$, as we ignored orders where that would be the case. Now let $W_i = (C'_2/F_i; \leq_i)$, and $X_1, \ldots, X_k$ representatives of the $W_i$'s up to monotonic bijection. Let $C_2^* \subset \prod X_i$, where each element of $C_2^*$ is the tuple of projections onto each $X_i$ of an element of $C_2$. Then the same argument as in Lemma \ref{lemma:refineconvex} gives that $C_2^*$ is dense in the product $\prod_i X_i$ equipped with the product topology. Namely, by Corollary \ref{cor:dense class}(1), $C_2$ is dense in each $X_i$. As each $X_i$ is minimal, Proposition \ref{prop:product definable} gives density in $\prod_i X_i$. By the last sentence of the previous paragraph, each $W_i$ contains a point $\leq_i$-greater than all of $A$. Thus we may find some $c \in C_2$ that is $\leq_i$-greater than all of $A$ for each $i$.
\end{proof}

For the next lemma, note that if $<'$ is a subquotient order in a generic $\Lambda$-ultrametric space from $F$ to $F \join G$, then $<' \upharpoonright_G$ is a subquotient order from $F \meet G$ to $G$ (see \cite{BraunThesis}*{Lemma 3.4.7}).

\begin{lemma} \label{lemma:0toE}
Let $\Gamma$ be a homogeneous finite-dimensional permutation structure with lattice of $\emptyset$-definable equivalence relations $\Lambda$. Let $E \in \Lambda$ be meet-reducible, and $<$ a subquotient order from $E$ to $G$, convex with respect to no intermediate $\emptyset$-definable equivalence relations. Then there exists some $F>E$ and subquotient order $<'$ from $F$ to $F \join G$ such that ${<} = {<' \upharpoonright_G}$.

Furthermore, $<$ and $<'$ are interdefinable.
\end{lemma}
\begin{proof}
By Lemma \ref{lemma:0to1}, $G$ cannot be above two covers of $E$, so we can find $F \in \Lambda$ such that $E = F \meet G$. Suppose the first part of the lemma is false for this $F$. Then there exist $F$-classes $C_1, C_2$ and $x_i, y_i \in C_i$ such that $x_1 < x_2$ and $y_1 > y_2$. In particular, $x_1Gx_2$ and $y_1Gy_2$, so $x_1,y_1$, and $C_2$ determine $x_2/E$ and $y_2/E$. We wish to produce an automorphism of $C_1$ sending $(x_1/E, x_2/E)$ to $(y_1/E, y_2/E)$, which will yield a contradiction. It suffices to produce an automorphism sending $x_1$ to $y_1$ and leaving $C_2$ invariant. By Lemma \ref{lemma:nameclass}, there is such an automorphism, and we are finished.

For the last part, first note any projection of $<'$ is $\emptyset$-definable from $<'$. For the other direction, we have that $F$ and $G$-classes within the same $(F \join G)$-class are cross-cutting by Corollary \ref{cor:twocross}, so we may define $x <' y \iff \exists z ((xFz) \wedge (z < y))$.
\end{proof}

\begin{corollary} \label{corollary:no meet red}
Let $\Gamma$ be a homogeneous finite-dimensional permutation structure with lattice of $\emptyset$-definable equivalence relations $\Lambda$. Let $<$ be a subquotient order with a meet-reducible bottom relation, convex with respect to no intermediate $\emptyset$-definable equivalence relations. Then $<$ is interdefinable with some subquotient order $<'$ with a meet-irreducible bottom relation.
\end{corollary}
\begin{proof}
Starting with $<$, iteratively apply Lemma \ref{lemma:0toE} until a subquotient order with meet-irreducible bottom relation is produced. This must eventually happen, as at each step the bottom relation moves upward in $\Lambda$, and the maximal elements of $\Lambda$ are meet-irreducible.
\end{proof}

Finally, we establish that the subquotient orders are generic by proving a suitable 1-point extension property.

\begin{lemma} \label{lemma:1pointextension}
Let $\Gamma$ be a homogeneous finite-dimensional permutation structure. Suppose every subquotient order of $\Gamma$ has a meet-irreducible bottom relation, and that no one is equal up to reversal to the restriction of another. Then the subquotient orders of $\Gamma$ are generic.
\end{lemma} 
   \begin{proof}
   We assume that no subquotient orders are convex with respect to any intermediate $\emptyset$-definable equivalence relation. Let $\Lambda$ be the lattice of $\emptyset$-definable equivalence relations of $\Gamma$.
   
   Let $A \subset \Gamma$ be finite, and $p(x)$ a complete quantifier-free 1-type over $A$ in the language of equivalence relations and subquotient orders, such that its restriction to the language of equivalence relations and any single subquotient order is consistent. To satisfy Definition \ref{def:generic}, it suffices to show $p(x)$ is realized. (For the equivalence between being a \fraisse limit and satisfying such a 1-point extension property, see \cite{Hodges}*{\S 7.1}.) Our plan will be to produce a consistent type $q(x)$ solely in the language of $\emptyset$-definable equivalence relations such that $q(x) \implies p(x)$. As the reduct to the language of $\emptyset$-definable equivalence relations is generic, we may then realize $q(x)$.
   
   We will produce a finite sequence of types $p_0(x), \dots, p_{\ell}(x) = q(x)$ such that the reduct of each to the language of $\emptyset$-definable equivalence relations is consistent, $p_i(x) \implies p_{i-1}(x)$, and $p_0(x) \implies p(x)$. Let $p_0(x)$ be $p(x)$, but removing any condition $x <_{i,j} a$ or $x >_{i,j} a$ for any $a \in A$ and any subquotient order $<_{i,j}$ with bottom relation $E_i$ such that there is some $b \in A$ such that $p \proves x E_i b$. By the consistency conditions on $p$, $p_0(x) \implies p(x)$.
   
   Let $E_1, \dots, E_\ell \in \Lambda$, ordered such that if $i < j$ then $E_i \not< E_j$, be the meet-irreducibles of $\Lambda$ such that there is some subquotient order $<_{i,j}$ with bottom relation $E_i$ and some $a_i \in A$ such that $p_0 \proves a_i <_{i,j} x$ or $p_0 \proves x <_{i,j} a_i$. Given $p_{k-1}$, we will produce $p_k$ by removing any condition $x <_{k,j} a$ or $x >_{k,j} a$ for any $a \in A$ and any subquotient order $<_{k,j}$ with bottom relation $E_k$, then finding a suitable $b_k \in \Gamma$ and adding $xE_k b_k$.
   
   Suppose we have found $p_{k-1}$.  Enumerate the subquotient orders with bottom relation $E_k$ as $<_i$ for $i \leq n$, let $F_i$ be the top relation of $<_i$, and let $E^+_k$ be the cover of $E_k$. As the reduct of $\Gamma$ to the language of equivalence relations is generic, the corresponding reduct of $p_{k-1}(x)$ is realized. Pick a realization, and let $C_{E_k^+}$ be its $E^+_k$-class. As $C_{E_k^+}$ is homogeneous, $C_{E_k^+}/E_k$ is generic by Corollary \ref{corollary:homquotient}. 
    
    For each $i \leq n$, $p_{k-1}(x)$ restricts $x/E_k$ to a $<_i$-interval $J_i$ with endpoints in $A \cup \set{\pm \infty}$. It also restricts $x$ to lie in a given $F_i$-class $C_{F_i}$, with $C_{F_i}$ containing the endpoints of $J_i$. By the consistency of the reduct of $p_{k-1}(x)$ to the language of $\emptyset$-definable equivalence relations, we have $C_{F_i} = C_{E_k^+}/F_i$.
    
     Let $I_i = C_{E_k^+}/E_k \cap J_i$, so $I_i$ is a $<_i$-interval in $C_{E_k^+}/E_k$. Working within $C_{F_i}/E_k$, since $<_i$ is not convex with respect to any intermediate $\emptyset$-definable equivalence relations, we may apply Corollary \ref{cor:dense class}(1) (where $E$ and $F$ there are $E_k$ and $E^{+}_k$ here, respectively) to get that $C_{E_k^+}/E_k$ is $<_i$-dense in $C_{F_i}/E_k$, and so $I_i\neq \emptyset$. By genericity of $C_{E_k^+}/E_k$, $\bigcap I_i \neq \emptyset$. Let $y \in \bigcap I_i$ be an $E_k$-class not among the finitely many $p_{k-1}(x)$ specifies are to be avoided. Then we may take $b_k \in \Gamma$ to be any element of $y$.
     
     We now show the reduct of $p_k(x)$ to the language of $\emptyset$-definable equivalence relations is consistent. Suppose the reduct of $p_{k-1}(x)$ forces $x$ to be in some $G_i$-class $C_{G_i}$, for some $G_i \in \Lambda$. Then for each such $i$, $G_i \not< E_k$ by assumption, so $\bigwedge G_i \not< E_k$ since $E_k$ is meet-irreducible and $\Lambda$ is distributive, so $((\bigwedge G_i) \join E_k) \geq E^+_k$. As $C_{E_k^+}$ is the $E^+_k$-class of some realization of $p_{k-1}$, we have that the $((\bigwedge G_i) \join E_k)$-class of $C_{E_k^+}$ equals the $((\bigwedge G_i) \join E_k)$-class of $\bigcap C_{G_i}$. Then by Corollary \ref{cor:twocross}, $\bigcap C_{G_i}$ meets any $E_k$-class in $C_{E_k^+}$, and so meets $b_k/E_k$.
   \end{proof}

\begin{theorem}
Let $\Gamma$ be a homogeneous finite-dimensional permutation structure. Then there is a finite distributive lattice $\Lambda$ such that $\Gamma$ is interdefinable with an expansion of the generic $\Lambda$-ultrametric space by generic subquotient orders, such that every meet-irreducible of $\Lambda$ is the bottom relation of some subquotient order.
\end{theorem}
\begin{proof}
Let $\Lambda$ be the lattice of $\emptyset$-definable equivalence relations of $\Gamma$. Then $\Lambda$ is distributive by Proposition \ref{prop:distributivity}, and the reduct of $\Gamma$ to the language of $\emptyset$-definable equivalence relations gives the generic $\Lambda$-ultrametric space by Proposition \ref{prop:homreduct}.

We may assume $\Gamma$ is presented in a language such that no subquotient order is equal up to reversal to the restriction of another, nor convex with respect to any intermediate $\emptyset$-definable equivalence relation. By Corollary \ref{corollary:no meet red}, we may also assume every subquotient order has a meet-irreducible bottom relation. By Lemma \ref{lemma:intersectionconvex}, every meet-irreducible is the bottom relation of some subquotient order. By Lemma \ref{lemma:1pointextension}, all subquotient orders are generic.
\end{proof}

\subsection*{Acknowledgements}

We would like to thank Gregory Cherlin for looking over some early forms of these results. The first author further thanks Gregory Cherlin for introducing him to this problem, and for many discussions on and around it during the his thesis work. The second author would  like to thank David Bradley-Williams, who brought this problem to his attention.
We thank the referee for helpful comments.


\end{document}